\documentclass[a4paper,12pt]{article}

\usepackage{geometry}
\usepackage{amsfonts,amsmath,amssymb,amsthm}
\usepackage{mathcomp,mathrsfs}
\usepackage{graphicx} % Required for inserting images
\usepackage{tikz-cd}
\usepackage{enumerate}
\usepackage{enumitem}
\usepackage{bm}
\usepackage[linktoc=page]{hyperref}
\usepackage[capitalise]{cleveref} 
\usepackage[title,titletoc]{appendix}

\theoremstyle{definition}

\newtheorem{definition}{Definition}[subsection]
\newtheorem{definitionApp}{Definition}[section]
\newtheorem*{definition*}{Definition}
\newtheorem{notation}[definition]{Notation}
\newtheorem{notationApp}[definitionApp]{Notation}
\newtheorem*{notation*}{Notation}

\newtheorem*{assumption*}{Assumption}

\theoremstyle{plain}

\newtheorem{theorem}[definition]{Theorem}
\newtheorem{theoremAlph}{Theorem}

\newtheorem*{theorem*}{Theorem}

\newtheorem*{conjecture*}{Conjecture}
\newtheorem{proposition}[definition]{Proposition}
\newtheorem{propositionApp}[definitionApp]{Proposition}
\newtheorem*{proposition*}{Proposition}
\newtheorem{lemma}[definition]{Lemma}
\newtheorem*{lemma*}{Lemma}
\newtheorem{corollary}[definition]{Corollary}
\newtheorem*{corollary*}{Corollary}

\theoremstyle{remark}

\newtheorem{example}[definition]{Example}
\newtheorem{exampleApp}[definitionApp]{Example}
\newtheorem{remark}[definition]{Remark}

\DeclareMathOperator{\Aut}{Aut}
\DeclareMathOperator{\End}{End}
\DeclareMathOperator{\Ker}{Ker}
\DeclareMathOperator{\GL}{GL}
\DeclareMathOperator{\Hom}{Hom}
\DeclareMathOperator{\Id}{Id}
\DeclareMathOperator{\Mat}{Mat}
\DeclareMathOperator{\Rep}{Rep}
\DeclareMathOperator{\rk}{rk}
\DeclareMathOperator{\Spec}{Spec}
\DeclareMathOperator{\Stab}{Stab}
\DeclareMathOperator{\Tr}{Tr}

\newcommand{\sst}{\mathrm{ss}}
\newcommand{\st}{\mathrm{s}}

\newcommand\GG{\mathbb{G}}
\newcommand\ZZ{\mathbb{Z}}
\newcommand\gl{\mathfrak{gl}}
\newcommand\bbA{\mathbb{A}}
\newcommand\bbC{\mathbb{C}}
\newcommand\bbQ{\mathbb{Q}}
\newcommand\bbZ{\mathbb{Z}}

\newcommand\bff{\mathbf{f}}
\newcommand\bfm{\mathbf{m}}
\newcommand\bfq{\mathbf{q}}
\newcommand\bfr{\mathbf{r}}
\newcommand\bfs{\mathbf{s}}
\newcommand\rmT{\mathrm{T}}

\def\cA{\mathcal A}
\def\cO{\mathcal O}

\DeclareMathOperator{\rSpec}{\underline{Spec}}
\DeclareMathOperator{\Proj}{Proj}
\DeclareMathOperator{\rProj}{\underline{Proj}}

\newcommand\lpow{[\![}
\newcommand\rpow{]\!]}
\newcommand\llau{(\!(}
\newcommand\rlau{)\!)}
\newcommand\git{/\!/}
\newcommand\limit{\underset{t\rightarrow0}{\lim}\ }

\newcommand{\geqp}{%
  \mathrel{\raisebox{-0.5ex}{$\scriptscriptstyle($}}%
  \geq
  \mathrel{\raisebox{-0.5ex}{$\scriptscriptstyle)$}}%
}

\usepackage[disable]{todonotes}

\usepackage[style=alphabetic]{biblatex}
\title{Moduli spaces of representations of quivers with multiplicities via non-reductive GIT}
\author{Victoria Hoskins\thanks{Universit\"{a}t Duisburg-Essen, \texttt{victoria.hoskins@uni-due.de}} \and Joshua Jackson\thanks{St John's College, University of Cambridge, \texttt{jjj26@cam.ac.uk}} \and Tanguy Vernet\thanks{Institute of Science and Technology Austria (ISTA), \texttt{tanguy.vernet@ist.ac.at}}}
\date{}

\begin{document}

\maketitle

\setcounter{tocdepth}{1}

\begin{abstract}
We construct new moduli spaces of quiver representations with multiplicities, i.e. over rings of truncated power series. This includes moduli of framed representations and analogues of Nakajima's quiver varieties. Our construction relies on tools from relative affine Geometric Invariant Theory for non-reductive groups and new stability conditions for quiver representations with multiplicities. We also study the cohomology of smooth moduli spaces of quiver representations with multiplicities, and show that several of these moduli spaces are cohomologically pure, using torus actions, as is the case for Nakajima's quiver varieties.
\end{abstract}

\tableofcontents

\section{Introduction}

\subsection{Background and motivation}

Moduli spaces of quiver representations are pervasive in moduli theory and geometric representation theory. They play a pivotal role, as they serve as local models for many moduli spaces parametrising linear objects, such as coherent sheaves, local systems or connections on algebraic varieties \cite{Mei15,Toda18,Dav24}. Moreover, there is a profound relation between quiver moduli and (generalised) Kac-Moody algebras (see for instance \cite{S18}). This allows one to use Lie-theoretic methods to analyse the topology of quiver moduli, and thus the topology of several moduli spaces. In return, the use of geometric techniques often reveals new structures in the study of Kac-Moody algebras and related quantum groups. Among many applications, one could mention geometric realisations of well-behaved bases of Kac-Moody algebras and their representations \cite{Lus91,KS97,Nak98,Sai02} and the explicit decomposition of the cohomology of several moduli stacks (of sheaves on curves, del Pezzo surfaces, symplectic surfaces) in terms of the intersection cohomology or BPS cohomology of their good moduli spaces \cite{Mei15,DHSM23}.

The geometric study of quiver moduli strongly relies on Geometric Invariant Theory (GIT) or, more generally, the use of reductive groups. Firstly, moduli spaces of quiver representations are constructed as GIT quotients \cite{Kin94}. Furthermore, the local modeling of moduli spaces by quiver moduli and the stratifications used in the study of their cohomology both rest on the construction of \'{e}tale slices \cite{Lun73,AHR20}. In light of \cite{BDINKP25}, one may also view the Hall induction techniques used in \cite{Mei15,DHSM23} as, at heart, a feature of moduli stacks which are locally modelled on quotient stacks by reductive groups.

In this paper, we construct new moduli spaces parametrising representations of a quiver with multiplicities (see \cref{first main thm}). A quiver with multiplicities $(Q,\bfm)$ is a quiver $Q= (Q_0,Q_1,s,t)$ and a collection $\bfm$ of positive integers called \emph{multiplicities} indexed by the vertices of $Q$. A representation of $(Q,\bfm)$ over $k$ consists of a free module over the truncated polynomial ring $k_{m_i} := k[\epsilon]/(\epsilon^{m_i})$ for each vertex $i \in Q_0$, and an appropriate linear map for each arrow (see \cref{def/quiver rep with mult}). Unlike the aforementioned spaces, moduli of quiver representations with multiplicities are naturally presented as quotients by non-reductive groups. Our construction therefore relies on Non-Reductive Geometric Invariant Theory (NRGIT), as developed in \cite{BDHK18,BDHK20,BK24,HHJ24,HHJ}. This leads us to define a notion of stability for representations of quivers with multiplicities, in the spirit of \cite{Kin94}.

Representations of quivers with multiplicities have found applications in both moduli theory and geometric representation theory over the past decade. On the one hand, certain moduli spaces of \emph{irregular} connections on the projective line can be realised as moduli of representations of star-shaped quivers with multiplicities \cite{Yam10,HWW23}. On the other hand, in \cite{GLS16,GLS18a}, Geiss, Leclerc and Schröer constructed geometric realisations of symmetrisable (as opposed to symmetric) Kac-Moody algebras, using moduli of representations of quivers with multiplicities.

Furthermore, certain counts of representations of quivers with multiplicities over finite fields enjoy good properties (polynomiality, positivity) \cite{HLRV24,HWW23,Ver24}. For quivers without multiplicities, such properties result from interpreting these counts as weight polynomials of moduli spaces, which are cohomologically pure \cite{CBVB04,HLRV13b,Dav18,Dav23c}. This is expected, for instance, for certain moduli spaces of connections known as open de Rham spaces \cite[Conj.\ 5.2.7]{HWW23}. One of the main motivations for this work was to establish cohomological purity for moduli spaces of quiver representations with multiplicities. Our second main result (\cref{main thm purity}) gives cohomological purity of various moduli spaces of quiver representations with multiplicities. This relies on a more general purity result for NRGIT quotients, which might be of independent interest. However, in the case of open de Rham spaces we could not find a suitable choice of stability to apply our results, so \cite[Conj.\ 5.2.7]{HWW23} remains open.

\subsection{NRGIT construction}

Let us now explain more concretely how we construct moduli spaces of quiver representations with multiplicities as NRGIT quotients.

Moduli spaces of representations of a quiver $Q$ (over a field $k$) of dimension vector $\bfr$ were constructed by King \cite{Kin94} as a reductive GIT quotient of the action
\begin{equation}\label{red GIT action for quivers}
    \GL_{\bfr}:= \prod_{i \in Q_0} \GL_{r_i}(k) \curvearrowright R(Q,\bfr):= \prod_{a \colon i \rightarrow j} \Hom(k^{\oplus r_i},k^{\oplus r_j})
\end{equation}
using a linearisation associated to a stability parameter $\theta \in \ZZ^{Q_0}$. This yields an open semistable set $ R(Q,\bfr)^{\theta-\sst}$ admitting a good quotient $M_{Q,\bfr}^{\theta-\sst}:= R(Q,\bfr) \git_\theta \GL_{\bfr}$, which is a coarse moduli space for (S-equivalence classes) of $\theta$-semistable representations. 

If we now consider a quiver with multiplicities $(Q,\bfm)$, analogously to above, the isomorphism classes of representations of $(Q,\bfm)$ with rank vector $\bfr$ are in bijection with the orbits of the following group action
\begin{equation}\label{NRGIT action for quivers}
    \GL_{\bfm,\bfr} := \prod_{i \in Q_0} \GL_{r_i}(k_{m_i}) \curvearrowright R(Q,\bfm;\bfr):= \prod_{a \colon i \rightarrow j} \Hom(k_{m_i}^{\oplus r_i},k_{m_j}^{\oplus r_j}).
\end{equation}
The stack of rank $\bfr$ representations of $(Q,\bfm)$ is $\Rep_{Q,\bfm;\bfr} = [R(Q,\bfm;\bfr)/\GL_{\bfm,\bfr}]$, whose dimension $\dim \Rep_{Q,\bfm;\bfr} =-\langle \bfr,\bfr \rangle_{Q,\bfm}$ is computed using an Euler pairing analogously to the classical case. However, the group $\GL_{\bfm,\bfr} $ is non-reductive. Indeed one has $\GL_{\bfm,\bfr} \cong \GL_\bfr \ltimes U_{\bfm,\bfr}$, where $U_{\bfm,\bfr}$ is the unipotent radical. Therefore, non-reductive GIT must be employed to construct moduli spaces. Certain Nakajima quiver varieties associated to star shaped quivers with constant multiplicities on the legs were constructed in \cite{HWW23} by proving an associated ring of invariants was finitely generated. In \cite{HHJ24}, a non-reductive GIT set-up was considered for quivers with constant multiplicities and moduli spaces of certain toric representations (i.e.\ $\bfr = \bm{1}$) were constructed. 

In this paper, we construct a quotient of the action \eqref{NRGIT action for quivers} using the relative approach to non-reductive GIT taken in \cite{HHJ}, which builds quotients of (graded) equivariant actions on affine morphisms. For us, this affine morphism is given by a truncation map $\tau$ which associates to a representation of $(Q,\bfm)$ a representation of $Q$; in the case of constant multiplicities, this truncation map is induced by the ring homomorphism $k_{m} \rightarrow k$ whose kernel is $(\epsilon)$ and simply corresponds to forgetting the positive powers of $\epsilon$, see \cite[Def.\ 7.2]{HHJ24}. In fact, via this truncation, the above actions \eqref{red GIT action for quivers} and \eqref{NRGIT action for quivers} fit into the following equivariant action:
\begin{equation}\label{equiv action quivers}
    \begin{tikzcd}[column sep=tiny]
    \GL_{\bfm,\bfr}\ar[d,"\tau_{\GL}"] & \curvearrowright & R(Q,\bfm;\bfr)\ar[d,"\tau"] \\
    \GL_{\bfr} & \curvearrowright & R(Q;\bfr),
    \end{tikzcd}
\end{equation}
where $\tau$ is equivariant with respect to the group homomorphism $\tau_{\GL}$. The kernel of $\tau_{\GL}$ is the unipotent radical of $\GL_{\bfm,\bfr}$. The above truncation maps induce a truncation map on stacks. Moreover, the truncation morphism $\tau$ also admits a natural section which enables us to view $\GL_\bfr$ as a Levi subgroup of $\GL_{\bfm,\bfr}$ and to view a representation $V$ of $Q$ as a representation of $(Q,\bfm)$, and there is naturally an induced homomorphism of automorphism groups  
\[ \tau_V \colon \Aut_{Q,\bfm}(V) \rightarrow \Aut_{Q}(V).\]
In fact, if $V$ is the representation associated to a point $x \in R(Q;\bfr)$, then $\Aut_{Q,\bfm}(V) \simeq \Stab_{\GL_{\bfm,\bfr}}(x)$ and $\Aut_Q(x) \simeq \Stab_{\GL_{\bfr}}(x)$ are stabiliser groups for the above action \eqref{equiv action quivers}.

For both actions there are subgroups $\Delta_m \subset \GL_{\bfm,\bfr}$ and $\Delta \subset \GL_{\bfr}$ acting trivially, which correspond to scalar automorphisms, and $\tau_{\GL}$ induces a morphism $\overline{\tau} \colon G_{\bfm,\bfr}:= \GL_{\bfm,\bfr} /\Delta_m \rightarrow \GL_{\bfr} /\Delta$. From the perspective of GIT, we could equivalently consider the equivariant action of $\overline{\tau}$, but it is more convenient to remember these constant stabilisers and work with the above action \eqref{equiv action quivers}.

\subsection{Main results and plan of the paper}

After collecting the required background on quivers with multiplicities and relative NRGIT in $\S$\ref{sec/background quivers mult} and $\S$\ref{sec/overview rel NRGIT} respectively, we proceed with proving \cref{first main thm} in $\S$\ref{sec/constr moduli spaces}. In order to apply \cite{HHJ} to the equivariant action \eqref{equiv action quivers}, one needs:
\begin{enumerate}[label=\roman*)]
    \item a good quotient of the action of $\GL_{\bfr}$ on an open subset of $R(Q,\bfr)$,
     \item a choice of character of $\GL_{\bfm,\bfr}$ to take the relative quotient,
     \item an assumption on the unipotent stabilisers on the open subset in i),
    \item  a \lq grading' of the action \eqref{equiv action quivers} by a multiplicative group $\GG_m$.
\end{enumerate}
Point i) is achieved by taking King's reductive GIT quotient associated to a choice of stability parameter $\theta$, and ii) corresponds to an additional choice of stability parameter $\rho$. Points iii) and iv) require a careful analysis (see $\S$ \ref{Sect/extGrad}-$\S$\ref{Sect/unipotentStab}). They are always present in NRGIT in some form to construct the unipotent part of the quotient using (local) slices.

Concerning iii), we only present \cref{first main thm} under the strongest possible stabiliser assumptions (see assumption \ref{U} below), which gives the nicest possible results; however, one can always construct a quotient by restricting to a locus of points with fixed dimensional unipotent stabilisers (see \cref{rmk NRGIT with non-trivial stabilisers}). \cref{prop suff cond for trivial unip stabilisers} collects some sufficient conditions to guarantee \ref{U} holds (e.g.\ if $\theta$ is generic with respect to $\bfr$). If \ref{U} fails, then one can still construct moduli spaces (see \cref{rmk moduli with non-trivial unip stabilisers}).

Point iv) is achieved by constructing a new external grading, which in the case of constant multiplicities, coincides with the one used in \cite{HHJ24}.

Given points i)-iv) above, we are able to construct a quotient of an open semistable set, which we moduli-theoretically interpret as a notion called $(\theta,\rho)$-semistability (see \cref{def/stabilityQuivMult}), and the quotient we obtain is relative to King's moduli space. Before we give the first main theorem, we (moduli-theoretically) introduce the strongest possible stabiliser assumption required to construct a moduli space.

\begin{assumption*}
We say assumption \ref{U} holds for rank $\bfr$ representations of $(Q,\bfm)$ if 
\begin{equation}\label{U}
			\overline{\tau}_V \colon \Aut_{Q,\bfm}(V)/\Delta_m \rightarrow \Aut_Q(V)/\Delta \text{ is injective } \forall \:  V \in \Rep_{Q,\bfr}^{\theta-\sst}  \tag*{$[U;\theta]$}
\end{equation}
\end{assumption*}

Our first main result under these strong stabiliser assumptions is as follows.

\begin{theoremAlph}\label{first main thm}
Let $(Q,\bfm)$ be a quiver with multiplicities,  $\bfr\in\bbZ_{\geq0}^{Q_0}$ be a rank vector and $\theta,\rho\in\bbZ^{Q_0}$ be stability parameters. Suppose the assumption \ref{U} holds for rank $\bfr$ representations of $(Q,\bfm)$. Then the following statements hold:
\begin{enumerate}[label=\roman*)]
    \item There are open sets $R(Q,\bfm;\bfr)^{\theta,\rho-\st} \subset R(Q,\bfm;\bfr)^{\theta,\rho-\sst} \subset R(Q,\bfm;\bfr)$ and a good $\GL_{\bfm,\bfr}$-quotient $q:R(Q,\bfm;\bfr)^{\theta,\rho-\sst} \rightarrow M_{Q,\bfm;\bfr}^{\theta,\rho-\sst}$, which restrict to a geometric quotient on $R(Q,\bfm;\bfr)^{\theta,\rho-\st}$.    
    \item There is a morphism $M_{Q,\bfm;\bfr}^{\theta,\rho-\sst} \rightarrow M_{Q,\bfr}^{\theta-\sst}$ which is projective over affine. Moreover, $M_{Q,\bfm;\bfr}^{\theta,\rho-\sst}$ is a normal irreducible quasi-projective variety. If the stable locus is non-empty, then $M_{Q,\bfm;\bfr}^{\theta,\rho-\sst}$ has dimension $\delta-\langle\bfr,\bfr\rangle_{Q,\bfm}$.
    \item The geometric points of $R(Q,\bfm;\bfr)^{\theta,\rho-\sst}$ correspond to $(\theta,\rho)$-semistable representations of $(Q,\bfm)$. As a consequence, $M_{Q,\bfm;\bfr}^{\theta,\rho-\sst}$ is a coarse moduli space for S-equivalence classes of $(\theta,\rho)$-semistable rank $\bfr$ representations of $(Q,\bfm)$. 
    \item If $\bfr$ is indivisible and $(\theta,\rho)$ is generic with respect to $\bfr$, then $q$ is a principal $G_{\bfm,\bfr}$-bundle, hence $M_{Q,\bfm;\bfr}^{\theta,\rho-\sst}$ is smooth and moreover a fine moduli space.
    \end{enumerate}
\end{theoremAlph}

This theorem is proved in $\S$\ref{sec proof of first main theorem}, and a natural generalisation leads to the construction of framed moduli spaces with multiplicities (see \cref{thm framed moduli with multiplicities}). We can also build analogues of Nakajima quiver varieties \cite{N94} for quivers with multiplicities (see \cref{Thm/quiverVarieties}). Indeed, as in the case of classical quivers, the cotangent bundle $\mathrm{T}^*R(Q,\bfm;\bfr)$ can be identified with the space of representations of the doubled quiver $(\overline{Q},\bfm)$. There is an explicit moment map:
\[
\mu_{(Q,\bfm),\bfr} \colon R(\overline{Q},\bfm;\bfr)\simeq\mathrm{T}^*R(Q,\bfm;\bfr)\rightarrow\gl_{\bfm,\bfr}^\vee\simeq\gl_{\bfm,\bfr}.
\]
Given $\gamma=(\gamma_i\cdot\Id)\in\gl_{\bfm,\bfr}$, our quiver varieties with multiplicities are defined as:
\[
N_{Q,\bfm;\bfr}^{\theta,\rho-\sst}(\gamma):=\mu_{(Q,\bfm),\bfr}^{-1}(\gamma)^{\theta,\rho-\sst}\git\GL_{\bfm,\bfr}.
\]

Finally in $\S$\ref{sec/purity}, as for classical quiver moduli \cite{ER09,Hau10}, we prove that several moduli spaces of quiver representations with multiplicities are cohomologically pure, i.e.\ the Hodge structure carried by their singular cohomology groups is pure \cite{Del74a}. This is the case for the smooth moduli spaces built in \cref{first main thm}.

We also prove cohomologically purity for certain modified Nakajima quiver varieties with multiplicities, which we denote by $\widetilde{N}_{Q,\bfm;\bfr}^{\theta,\rho-\sst}(\gamma)$. The modification concerns the stability condition, as our proof requires an externally grading $\GG_m$-action which makes the moment map $\mu_{(Q,\bfm),\bfr}$ equivariant. Since this is not the case for the external grading used above, we modify the grading for this proof, which also modifies the truncation morphism $\tau$ and the associated stability condition (see $\S$\ref{sec/purity for quiver varieties with multiplicities} for details). Our second main result reads as follows.

\begin{theoremAlph}\label{main thm purity}
Let $(Q,\bfm)$ be a quiver with multiplicities and $\bfr\in\bbZ_{\geq0}^{Q_0}$ be an indivisible rank vector. Let $\theta,\rho\in\bbZ^{Q_0}$ be stability parameters and assume that $(\theta,\rho)$ is generic with respect to $\bfr$. Suppose further that assumption \ref{U} holds for rank $\bfr$ representations of $(Q,\bfm)$. Then the following statements hold.
\begin{enumerate}[label=\roman*)]
\item\label{main thm purity first part}  The moduli space $M_{Q,\bfm;\bfr}^{\theta,\rho-\sst}$ is cohomologically pure.

\item\label{main thm purity second part} Let $\gamma=(\gamma_i\cdot\Id)\in\gl_{\bfm,\bfr}$. Then $\widetilde{N}_{Q,\bfm;\bfr}^{\theta,\rho-\sst}(\gamma)$ is cohomologically pure.
\end{enumerate}
\end{theoremAlph}

This theorem is proved in $\S$\ref{sec/purity for moduli of quiver reps with multiplicities}-$\S$\ref{sec/purity for quiver varieties with multiplicities}. It follows from a general result about purity of NRGIT quotients (see \cref{Thm/purityNRGITquotients}) when there is a conical $\GG_m$ acting on the affine morphism and commuting with the action of the (external) grading $\GG_m$. From this 2-dimensional torus action, one can produce a semi-projective $\GG_m$-action by taking an appropriate subtorus, and this is then used to deduce cohomological purity.

Finally, in Appendix \ref{appendix}, we explain how the grading of the action (\ref{equiv action quivers}) by a multiplicative group $\GG_m$ can be obtained from the grading introduced in \cite[\S 7]{HHJ24}, by unfolding the quiver with multiplicities $(Q,\bfm)$ into a quiver with constant multiplicities.

\subsubsection*{Notation and conventions} We let $k$ be a field, which from \cref{sec/overview rel NRGIT} onwards will be assumed to be algebraically closed of characteristic zero.

Given a block-matrix with $m\times n$ blocks ($m$ block-rows and $n$ block-columns), we say that the block on the $i$th block-row ($1\leq i\leq m$) and on the $j$th block-column ($1\leq j\leq n$) has coordinates $(i,j)$. Block-rows (resp. block-columns) are labeled by increasing integers from top to bottom (resp. from left to right).

Given a truncated power series $x\in k[\epsilon]/(\epsilon^m)$ ($m\geq1$), we write coefficients of $x$ as follows:
\[
x=x_0+x_1\epsilon+\ldots+x_{m-1}\epsilon^{m-1}.
\]

\subsubsection*{Acknowledgements} 
The authors would like to thank the Isaac Newton Institute for Mathematical Sciences, Cambridge, for support and hospitality during the programme New equivariant methods in algebraic and differential geometry, where this project was started. This programme was supported by EPSRC grant EP/R014604/1. T.V. was supported
by the European Union’s Horizon 2020 research and innovation programme under the Marie
Skłodowska-Curie Grant Agreement No. 101034413. He thanks Tam\'{a}s Hausel for his support and helpful discussions. We also thank Eloise Hamilton for useful conversations about quivers with multiplicities.

\section{Background on quivers with multiplicities}\label{sec/background quivers mult}

In this section, we recall basic facts on representations of quivers with multiplicities and classical moduli of quiver representations, following \cite{Kin94}. Throughout this section, we let $k$ be an arbitrary field.

\subsection{Representations of quivers with multiplicities}

We start with the definition of quivers with multiplicities and their representations.

\begin{definition}
A quiver is the datum $Q=(Q_0,Q_1,s,t)$ of a set of vertices $Q_0$, a set of arrows $Q_1$, along with source and target maps $s,t \colon Q_1\rightarrow Q_0$ indicating the orientation.

A quiver with multiplicities is the datum $(Q,\bfm)$ of a quiver $Q$ along with a collection of positive integers $\bfm=(m_i)_{i\in Q_0}$, called multiplicities.
\end{definition}

Classical representations of $Q$ (without multiplicities) are given by vector spaces over each vertex with corresponding linear maps over each arrow. Representations of $(Q,\bfm)$ replace the $k$-vector spaces with (free) modules over truncated polynomial rings, determined by the multiplicities. To define the appropriate linearity condition over the arrows, we introduce the following notation.

\begin{notation}\label{notation gcd of multiplicities}
For $m \in \mathbb{Z}_{>0}$, let $k_m:=k[\epsilon]/(\epsilon^m)$. For multiplicities $\bfm=(m_i)_{i\in Q_0}$, we introduce the following notation (consistent with \cite{GLS17a}) 
\begin{itemize}
    \item $m_{ij}:=\gcd(m_i,m_j)$,
    \item $\mu_{ij}:=\mathrm{lcm}(m_i,m_j)$,
    \item $f_{ji}:=\frac{m_i}{m_{ij}}$,
    \item $\delta:=\gcd(m_i,\ i\in Q_0)$.
\end{itemize}
Note that via the $k$-algebra homomorphism $k_{m_{ij}} \rightarrow k_{m_i}$ given by $\epsilon\mapsto\epsilon^{f_{ji}}$, we can view  $k_{m_i}$ as a $k_{m_{ij}}$-algebra. Moreover $k_\delta$ is a subring of each $k_{m_i}$.
\end{notation}

We will also use the following bilinear form in certain dimension counts.

\begin{notation}\label{Notation Euler pairing}
We denote by $\langle\bullet,\bullet\rangle_{Q,\bfm}$ the following quadratic form:
\[
\begin{array}{ccc}
\bbZ^{Q_0}\times\bbZ^{Q_0} & \rightarrow & \bbZ \\
(\bfr,\bfs) & \mapsto & \sum_{i\in Q_0}m_ir_is_i-\sum_{\substack{a\in Q_1 \\ a \colon i\rightarrow j}}\mu_{ij}r_is_j.
\end{array}
\]
\end{notation}

Although we are primarily concerned with free $k_m$-modules, we will later need a notion of rank for potentially non-free $k_m$-modules (see \cref{def/stabilityQuivMult}).

\begin{definition}
Let $M$ be a finitely generated $k_m$-module. Since $M$ is $\epsilon^m$-torsion and finitely generated over $k[\epsilon]$, it admits a decomposition:
\[
M\simeq\bigoplus_{1\leq j\leq m}k_j^{\oplus s_j},
\]
for unique non-negative integers $s_i$. We define the rank of $M$ to be the rank of its free part (over $k_m)$; that is, $\rk M:=s_m$. Note that, for any $k_m$-submodule $N$ of a free module $M$ of finite rank, we have:
\[
\rk(N)=\dim_kN/(N\cap\epsilon M).
\]
\end{definition}

We can now define our main object of study.

\begin{definition}\label{def/quiver rep with mult}
A representation of $(Q,\bfm)$ is a collection $M$ of $k_{m_i}$-modules $M_i$ for each vertex $ i\in Q_0$, along with $k_{m_{ij}}$-linear maps $M_a \colon M_{i}\rightarrow M_{j}$ for each arrow $a \colon i \rightarrow j$ in $Q_1$. We say $M$ is locally free (resp.\ finitely generated) if the $k_{m_i}$-module $M_i$ is free (resp.\ finitely generated) for all $i \in Q_0$. When $M$ is finitely generated and locally free, we call $\rk(M):=(\rk M_i)_{i\in Q_0}$ the rank vector of $\bfr$.

Given two representations $M,M'$ of $(Q,\bfm)$, a morphism $f \colon M\rightarrow M'$ is a collection of $k_{m_i}$-linear maps $f_i \colon M_i\rightarrow M'_i$ such that, for all $a\in Q_1$, the following diagram commutes:
\[
\begin{tikzcd}
M_{s(a)} \ar[r,"M_a"]\ar[d,"f_{s(a)}"] & M_{t(a)} \ar[d,"f_{t(a)}"] \\
M'_{s(a)} \ar[r,"M'_a"] & M'_{t(a)}.
\end{tikzcd}
\]
\end{definition}

In this paper, we always work with representations of finite rank and consider moduli of locally free representations. Non-locally free representations only appear in the statement of stability conditions.

\begin{remark}[Euler form]
The category of finitely generated representations of $(Q,\bfm)$ is abelian. Within this category, locally free representations are characterised as representations of projective dimension $1$ \cite[Prop.\ 3.5]{GLS17a}. The Euler pairing for locally free representations only depends on the rank vectors and is given by $\langle\bullet,\bullet\rangle_{Q,\bfm}$ \cite[Prop.\ 4.1]{GLS17a}.
\end{remark}

\begin{remark}[Scalar automorphisms]\label{rmk/scalarAut}
Let $M$ be a locally free representation of $(Q,\bfm)$. Then by definition, all maps $\varphi_a \colon M_{s(a)}\rightarrow M_{t(a)}$ are $k_\delta$-linear. Therefore, the automorphism group of $M$ contains a copy of $k_\delta^{\times}$, which we call scalar automorphisms.
\end{remark}

In order to geometrise the classification of quiver representations with multiplicities, we introduce a notion of families of such representations over an affine base scheme.

\begin{definition}\label{def/stack rep with mult}
Let $\bfr\in\bbZ_{\geq0}^{Q_0}$ be a rank vector. For a $k$-algebra $R$, a family of locally free rank $\bfr$ representations of $(Q,\bfm)$ over $\Spec(R)$ consists of the following data:
\begin{itemize}
\item a projective $R\otimes_kk_{m_i}$-module $M_i$ of rank $r_i$ for each $i\in Q_0$;
\item $R\otimes_kk_{m_{ij}}$-linear maps $M_a \colon M_{s(a)}\rightarrow M_{t(a)}$ for each arrow $a \colon i \rightarrow j $ in  $Q_1$.
\end{itemize}
Given a morphism of $k$-algebras $R\rightarrow R'$, the pullback of a family $\{M_i,M_a\}$ from $\Spec(R)$ to $\Spec(R')$ is given by the data $\{M_i\otimes_RR',M_a\otimes\mathrm{id}\}$.

The moduli stack of locally free rank $\bfr$ representations of $(Q,\bfm)$ over $k$ is defined as follows:
\[
\begin{array}{crcl}
\Rep_{Q,\bfm;\bfr} \colon & k\mathrm{-Alg} & \rightarrow & \mathrm{Gpd} \\
& R & \mapsto & \Rep_{Q,\bfm;\bfr}(R),
\end{array}
\]
where $\Rep_{Q,\bfm;\bfr}(R)$ is the groupoid of families of locally free, rank $\bfr$ representations of $(Q,\bfm)$ over $\Spec(R)$.
\end{definition}

Standard descent arguments show that $\Rep_{Q,\bfm;\bfr}$ is a stack for the fppf topology. We will show it is a quotient stack in \cref{Prop/ModuliStackVSQuotientStack} below.

\subsection{Moduli of quiver representations as a quotient}

In this section, we introduce a group action whose orbits correspond to isomorphism classes of representations of a quiver with multiplicities of fixed rank, and thus describe the moduli stack as a quotient stack. In the case of a quiver without multiplicities, we recall how King uses (reductive) geometric invariant theory to construct moduli spaces of semistable quiver representations \cite{Kin94}.  

\subsubsection{Presentation as a quotient stack}

We start with the group action used in the presentation of the moduli stack of quiver representations with multiplicities.

\begin{definition}\label{Def action on rep space}
Let $(Q,\bfm)$ be a quiver with multiplicities and $\bfr\in\bbZ_{\geq0}^{Q_0}$ be a rank vector. The representation space of $(Q,\bfm)$ with rank $\bfr$ is the following affine space over $k$:
\[
R(Q,\bfm;\bfr):=\prod_{a\in Q_1}\Hom_{k_{m_{ij}}}\left(k_{m_i}^{\oplus r_i},k_{m_j}^{\oplus r_j}\right).
\]
Let $\GL_{m,r}:= \GL_r(k_{m})$ be the automorphism group of $k_m^{\oplus r}$. The representation space is endowed with an action of the following algebraic group:
\[
\GL_{\bfm,\bfr}:=\prod_{i\in Q_0}\GL_{m_i,r_i}.
\]
In coordinates, the action is given by the following formula:
\[
(g_i)_{i\in Q_0}\cdot (x_a)_{a\in Q_1}=\left(g_{t(a)}x_ag_{s(a)}^{-1}\right)_{a\in Q_1}.
\]
Let $\Delta_\bfm\subset\GL_{\bfm,\bfr}$ denote the image of the diagonal embedding:
\[
k_\delta^{\times}\hookrightarrow\prod_{i\in Q_0}k_{m_i}^{\times}\hookrightarrow \prod_{i\in Q_0}\GL_{m_i,r_i}= \GL_{\bfm,\bfr}
\]
and let $G_{\bfm,\bfr}:=\GL_{\bfm,\bfr}/\Delta_\bfm$ denote the quotient group.
\end{definition}

The subgroup $\Delta_\bfm$ acts trivially on $R(Q,\bfm;\bfr)$ and accounts for the scalar automorphisms of the moduli stack $\Rep_{Q,\bfm;\bfr}$ (see \cref{rmk/scalarAut}). 

\begin{notation}[Notation for quivers with trivial multiplicities]
When $\bfm=\mathbf{1}$, we simply omit $\bfm$ from the notation and write $R(Q,\bfr):=R(Q,\mathbf{1};\bfr)$ and we write the above algebraic groups as follows
\[ \Delta:=\Delta_{\mathbf{1}} \hookrightarrow \GL_\bfr:=\GL_{\mathbf{1},\bfr} \twoheadrightarrow G_{\bfr}:=G_{\mathbf{1},\bfr}.\]
\end{notation}

Following \cite[\S 4]{Ser58}, we say that an algebraic group $G$ is special if all principal $G$-bundles are Zariski-locally trivial.

\begin{lemma}\label{Lem/unipotentRad}
The algebra homomorphisms $k_{m_i}\twoheadrightarrow k$ for each $ i\in Q_0$ induce a surjective morphism of algebraic groups $\GL_{\bfm,\bfr}\twoheadrightarrow\GL_\bfr$, whose kernel is the unipotent radical of $\GL_{\bfm,\bfr}$. As a consequence, $\GL_{\bfm,\bfr}$ is special.
\end{lemma}

\begin{proof}
Surjectivity is routine, as an element of $\Mat_{r_i\times r_i}(k_{m_i})$ is invertible if and only if its image in $\Mat_{r_i\times r_i}(k)$ is invertible.

The kernel $K$ of the morphism $\GL_{\bfm,\bfr}\twoheadrightarrow\GL_\bfr$ is a normal subgroup and is unipotent, as its elements act unipotently on the faithful $\GL_{\bfm,\bfr}$-representation $\bigoplus_{i\in Q_0}k_{m_i}^{\oplus r_i}$. Since the quotient $\GL_\bfr$ is reductive, $K$ is the unipotent radical of $\GL_{\bfm,\bfr}$.

Finally, $\GL_{\bfm,\bfr}$ is special, since it is an extension of special groups \cite[Prop.\ 8]{Ser58}. Indeed, the unipotent radical $K$ is an extension of additive groups, which are special \cite[\S 4.4]{Ser58}\footnote{Note that Serre works over an algebraically closed field \cite{Ser58}, but the statements we use are valid over an arbitrary base field.} and the Levi quotient $\GL_\bfr$ is also special \cite[Lem.\ III.4.10]{Mil80}.
\end{proof}

We can now prove the desired presentation of $\Rep_{Q,\bfm;\bfr}$.

\begin{proposition}\label{Prop/ModuliStackVSQuotientStack}
The moduli stack $\Rep_{Q,\bfm;\bfr}$ is isomorphic to $\left[R(Q,\bfm;\bfr)/\GL_{\bfm,\bfr}\right]$. As a consequence, $\Rep_{Q,\bfm;\bfr}$ is a smooth Artin stack of dimension $-\langle\bfr,\bfr\rangle_{Q,\bfm}$.
\end{proposition}

\begin{proof}[Sketch of proof]
We indicate how to construct an equivalence of groupoids
\[
\left[R(Q,\bfm;\bfr)/\GL_{\bfm,\bfr}\right](R)\overset{\sim}{\rightarrow}\Rep_{Q,\bfm;\bfr}(R)
\]
for any $k$-algebra $R$, which is functorial in $R$.

Consider a principal $\GL_{\bfm,\bfr}$-bundle $P\rightarrow S=\Spec(R)$, endowed with a $\GL_{\bfm,\bfr}$-equivariant morphism $P\rightarrow R(Q,\bfm;\bfr)$. Since $\GL_{\bfm,\bfr}$ is special, we know that $P$ is trivialised by an affine open cover $S'=\Spec(R')\rightarrow S$, so that $P':=P\times_SS'\simeq S'\times\GL_{\bfm,\bfr}$. Using the section of $P'$ at the identity element, we obtain a morphism $S'\rightarrow R(Q,\bfm;\bfr)$ that corresponds to $x'\in R(Q,\bfm;\bfr)(R')$.

Define $M'\in\Rep_{Q,\bfm;\bfr}(R')$ as follows: let $M'_i:=(R'\otimes_kk_{m_i})^{\oplus r_i}$ for each $i\in Q_0$, with maps $M'_a$ induced by the above point $x'$. The descent data (from $S'$ to $S$) for the principal bundle $P'$ gives descent data for $M'$, so we obtain an object $M\in\Rep_{Q,\bfm;\bfr}(R)$. We claim that this defines a functorial equivalence of groupoids. The argument is similar to \cite[Prop.\ 3.1.4]{BDFHMT22}.

The dimension count is a straightforward check, as we have:
\[
\dim\Hom_{k_{m_{ij}}}(k_{m_i}^{\oplus r_i},k_{m_j}^{\oplus r_j})
=
\dim\Hom_{k_{m_{ij}}}(k_{m_{ij}}^{\oplus r_if_{ji}},k_{m_{ij}}^{\oplus r_jf_{ji}})
=
m_{ij}f_{ij}f_{ji}r_ir_j=\mu_{ij}r_ir_j
\]
for each arrow $a \colon i\rightarrow j$.
\end{proof}

\begin{notation}
For a closed point $x\in R(Q,\bfm;\bfr)$ (resp.\ $x\in R(Q,\bfr)$), we call $M(x)$ (resp.\ $V(x)$) the corresponding representation of $(Q,\bfm)$ (resp.\ $Q$).
\end{notation}

\subsubsection{Moduli of semistable quiver representations (after King)}

In this subsection, we outline King's construction of moduli spaces of semistable representations of a quiver $Q$ without multiplicities ($\bfm=\mathbf{1}$) as a GIT quotient of the  reductive group $\GL_{\bfr}$ acting on $R(Q,\bfr)$ with respect a stability parameter $\theta\in\bbZ^{Q_0}$.

\begin{definition}
For $\theta\in\bbZ^{Q_0}$, we say a representation $M$ of $Q$ is $\theta$-semistable (resp.\ $\theta$-stable) if $\theta\cdot\dim M=0$ and, for all non-zero, proper subrepresentations $N\subsetneq M$, we have:
\[
\theta\cdot\dim N\geq0 \
\left(
\text{resp. } \theta\cdot\dim N>0
\right).
\]
\end{definition}

Fix a rank vector $\bfr\in\bbZ_{\geq0}^{Q_0}$ and choose $\theta\in\bbZ^{Q_0}$ such that $\theta\cdot\bfr=0$. Then the character $\chi_\theta \colon \GL_\bfr\rightarrow\GG_m$ defined by
\[
\chi_\theta(g_i,\ i\in Q_0)=\prod_{i\in Q_0}\det(g_i)^{\theta_i}
\]
descends to a character of $G_\bfr=\GL_\bfr/\Delta$. 
We say $f\in k[R(Q,\bfr)]$ is a $\chi_{\theta}^n$-semi-invariant (where $n\geq1$) if it satisfies the formula:
\[
f(g\cdot x)=\chi_{\theta}^n(g)\cdot f(x),
\]
for $g\in\GL_\bfr$ and $x\in R(Q,\bfr)$. We denote by $k[R(Q,\bfr)]_{\chi_{\theta}^n}^{\GL_\bfr}$ the space of $\chi_{\theta}^n$-semi-invariant functions.

We sum up below key facts about King's moduli spaces of semistable quiver representations.

\begin{theorem}\emph{\cite[Prop.\ 3.1]{Kin94}\cite[Lem.\ 6.5]{Rei03}}
The subset $R(Q,\bfr)^{\theta-\sst}\subset R(Q,\bfr)$ (resp.\ $R(Q,\bfr)^{\theta-\st}\subset R(Q,\bfr)$) formed by geometric points corresponding to $\theta$-semistable (resp.\ $\theta$-stable) representations is open and $\GL_\bfr$-invariant.

The domain of definition of the rational map
\[
R(Q,\bfr)\dashrightarrow M_{Q,\bfr}^{\theta-\sst}:=R(Q,\bfr)\git_{\theta}\GL_\bfr=\Proj\left(\bigoplus_{n\geq0}k[R(Q,\bfr)]_{\chi_{\theta}^n}^{\GL_\bfr}\right)
\]
is $R(Q,\bfr)^{\theta-\sst}$ and the induced map is a good quotient, and restricts to a geometric quotient $R(Q,\bfr)^{\theta-\st} \rightarrow M_{Q,\bfr}^{\theta-\st}$. Moreover, the action of $G_\bfr$ on $R(Q,\bfr)^{\theta-\st}$ is scheme-theoretically free and the quotient map is a principal $G_\bfr$-bundle, so $M_{Q,\bfr}^{\theta-\st}$ is smooth. Finally, $M_{Q,\bfr}^{\theta-\sst}$ is projective over the affine scheme
\[
M_{Q,\bfr}:=R(Q,\bfr)\git\GL_\bfr=\Spec\left( k[R(Q,\bfr)]^{\GL_\bfr}\right).
\]
\end{theorem}

We will sometimes use the following condition on $\theta$, which guarantees that there are no strictly $\theta$-semistable points, and thus that $M_{Q,\bfr}^{\theta-\sst}$ is smooth.

\begin{definition}\label{Def/genericStabilityParameter}
Let $\bfr\in\bbZ_{\geq0}^{Q_0}$ be a rank vector and $\theta\in\bbZ^{Q_0}$ be a stability parameter such that $\theta\cdot\bfr=0$. We say that $\theta$ is generic with respect to $\bfr$ if, for all $0<\bfr'<\bfr$, we have $\theta\cdot\bfr'\ne0$.
\end{definition}

For a rank vector $\bfr\in\bbZ_{\geq0}^{Q_0}$, there exist stability parameters which are generic with respect to $\bfr$ if and only if the vector $\bfr$ is indivisible.

\subsection{Moduli of framed quiver representations as a quotient}

One can rigidify moduli of quiver representations using additional \lq framing' data. There are various motivations for using framed representations (without multiplicities): i) they provide smooth models of moduli spaces of semistable quiver representations \cite{ER09}, ii) framed quiver moduli spaces can be used to (cohomologically) approximate the morphism from the stack of semistable representations to its corresponding good moduli space by proper morphisms \cite[$\S$4.1]{DM20}, and iii) framed quiver moduli spaces, and in particular Nakajima quiver varieties \cite{N94}, play an important role in geometric representation theory. 

In this section, we will describe the stack of framed quiver representations with multiplicities as a quotient stack. For a quiver without multiplicities, we state the construction and properties of moduli spaces of semistable framed representations as in \cite{ER09}.

\subsubsection{Framed quiver representations with multiplicities}

Let $(Q,\bfm)$ be a quiver with multiplicities. Fix a vector $\bff\in\bbZ_{\geq0}^{Q_0}$, which we call the framing vector.

\begin{definition}
An $\bff$-framed representation of $(Q,\bfm)$ is a pair $(M,b)$ consisting of a representation $M$ of $(Q,\bfm)$ and framing data $b$ consisting of $k_{m_i}$-linear maps $b_{i}\colon k_{m_i}^{\oplus f_i}\rightarrow M_i$ for each $i \in Q_0$. We say $(M,b)$ is locally free (resp. finitely generated) if $M$ is, and refer to $\rk M$ as the rank vector of $(M,b)$.

A morphism $h \colon (M,b)\rightarrow (M',b')$ of $\bff$-framed representations of $(Q,\bfm)$ is a morphism $h \colon M \rightarrow M'$ of $(Q,\bfm)$-representations that is compatible with the framing data in the sense that for all $i\in Q_0$, the following diagram commutes:
\[
\begin{tikzcd}
k_{m_i}^{\oplus f_i} \ar[r,"b_i"]\ar[d,equal] & M_{i} \ar[d,"h_i"] \\
k_{m_i}^{\oplus f_i} \ar[r,"b_i'"] & M'_{i}.
\end{tikzcd}
\]
\end{definition}

\subsubsection{The stack of framed quiver representations with multiplicities}

There is a moduli stack $\Rep_{Q,\bfm;\bfr}^{\bff}$ of $\bff$-framed locally free rank $\bfr$ representations of $(Q,\bfm)$ defined analogously to \cref{def/stack rep with mult}.

We now recall Crawley-Boevey's trick \cite{CB01}, which allows one to present the moduli stack $\Rep_{Q,\bfm;\bfr}^{\bff}$ as a rigidified moduli stack of representations of an associated framed quiver with multiplicities $(Q_\bff,\widehat{\bfm})$ defined as follows.

\begin{definition}\label{Def quiver with mult associated to framing}
Given $(Q,\bfm)$ and framing vector $\bff$ and rank vector $\bfr \in \mathbb{N}^{Q_0}$, we define the associated framed quiver with multiplicities $(Q_\bff,\widehat{\bfm})$ and rank vector $\hat{\bfr}$ for $Q_\bff$ by:
\begin{itemize}
\item $(Q_\bff)_0:=Q_0\sqcup\{\infty\}$ and  $(Q_\bff)_1:=Q_1\sqcup\bigsqcup_{i\in Q_0}\{b_{i,f} \colon \infty\rightarrow i,\ 1\leq f\leq f_i\}$,
\item $\widehat{m}_i:=
\left\{
\begin{array}{ll}
m_i & i\in Q_0, \\
1 & i=\infty.
\end{array}
\right.
$
\item $
\widehat{r}_i=
\left\{
\begin{array}{ll}
r_i & i\in Q_0, \\
1 & i=\infty.
\end{array}
\right.
$
\end{itemize}
\end{definition}

Note that $\Delta_{\widehat{\bfm}}=\GL_{\widehat{r}_\infty}(k_{\widehat{m}_\infty})=k^{\times}$ and $G_{\widehat{\bfm},\widehat{\bfr}}=\GL_{\widehat{\bfm},\widehat{\bfr}}/\Delta_{\widehat{\bfm}}\simeq\GL_{\bfm,\bfr}$.

\begin{proposition}
The moduli stack $\Rep_{Q,\bfm;\bfr}^{\bff}$ is isomorphic to  $[R(Q_\bff,\hat{\bfm};\hat{\bfr})/G_{\widehat{\bfm},\widehat{\bfr}}]$. As a consequence, $\Rep_{Q,\bfm;\bfr}^{\bff}$ is a smooth Artin stack of dimension $\sum_{i\in Q_0}m_if_ir_i-\langle\bfr,\bfr\rangle_{Q,\bfm}$.
\end{proposition}

\begin{proof}[Sketch of proof]
An argument similar to the proof of Proposition \ref{Prop/ModuliStackVSQuotientStack} gives the following isomorphism of stacks:
\[
\Rep_{Q,\bfm;\bfr}^{\bff}\simeq \left[ \left(R(Q,\bfm;\bfr)\times\prod_{i\in Q_0}\Hom_{k_{m_i}}\left(k_{m_i}^{\oplus f_i},k_{m_i}^{\oplus r_i}\right)\right)/\GL_{\bfm,\bfr} \right].
\]
The first part of the result then follows as the isomorphism
\[
R(Q,\bfm;\bfr)\times\prod_{i\in Q_0}\Hom_{k_{m_i}}\left(k_{m_i}^{\oplus f_i},k_{m_i}^{\oplus r_i}\right)\simeq R(Q_\bff,\widehat{\bfm};\widehat{\bfr})
\]
is equivariant with respect to the action of $\GL_{\bfm,\bfr}\simeq G_{\widehat{\bfm},\widehat{\bfr}}$. From the computation:
\[
\dim R(Q_{\bff},\hat{\bfm};\hat{\bfr})=\dim R(Q,\bfm;\bfr)+\sum_{i\in Q_0}m_if_ir_i,
\]
the dimension count follows directly.
\end{proof}

\subsubsection{Moduli of semistable framed representations (after Engel, Reineke)}

In this subsection, we describe moduli of semistable framed representations of a quiver $Q$ without multiplicities ($\bfm=\mathbf{1}$) after Engel and Reineke \cite{ER09}. We first recall the notion of semistable framed representations.

\begin{definition}\label{Def framed semistability}
For $\theta\in\bbZ^{Q_0}$, we say a framed representation $(M,b)$ is $\theta$-semistable if the following two conditions hold:
\begin{itemize}
\item $M$ is $\theta$-semistable (in particular $\theta\cdot\dim(M)=0$);
\item for all proper subrepresentations $N\subsetneq M$ containing the image of $b$, we have $\theta\cdot\dim N>0$.
\end{itemize}
\end{definition}

For $\theta = 0$, we note that $0$-semistability for a framed representation $(M,b)$ imposes no condition on $M$ but requires that the framing data $b$ \lq generates' $M$ (so the image of $b$ is not contained in a proper subrepresentation of $M$). Note that if $M$ is $\theta$-stable, then the condition on the image of $b$ automatically holds.

The following theorem explains how framed quiver moduli spaces can be used to give a (higher dimensional) \lq resolution' of unframed quiver moduli spaces. 

\begin{theorem}\label{thm/classicalframedquiver}\emph{\cite[Def.\ 3.1 and Prop.\ 3.3]{ER09}}
Given $Q$ with rank vector $\bfr$ and framing vector $\bff$ and stability parameter $\theta \in Z^{Q_0}$, there exists $\widehat{\theta}\in\bbZ^{(Q_\bff)_0}$ with the following properties.
\begin{enumerate}[label=(\roman*)]
    \item The parameter $\widehat{\theta}$ is generic with respect to $\widehat{\bfr}$, so $R(Q_\bff,\widehat{\bfr})^{\widehat{\theta}-\sst}=R(Q_\bff,\widehat{\bfr})^{\widehat{\theta}-\st}$ and the geometric points of this scheme correspond to $\theta$-semistable $\bff$-framed representations of $Q$ of rank $\bfr$.
    \item The quotient $M_{Q,\bfr}^{\bff;\theta-\sst}:=M_{Q_\bff,\hat{\bfr}}^{\widehat{\theta}-\sst}$ is a geometric $\GL_{\widehat{\bfr}}$-quotient and the quotient map is a principal bundle for $G_{\widehat{\bfr}}\simeq \GL_\bfr$. In particular, $M_{Q,\bfr}^{\bff;\theta-\sst}$ is a smooth variety.
    \item There is a projective morphism $M_{Q,\bfr}^{\bff;\theta-\sst} \rightarrow M_{Q,\bfr}^{\theta-\sst}$, which is a projective fibration over $M_{Q,\bfr}^{\theta-\st}$.
\end{enumerate}
\end{theorem}

The fibres of this projective morphism over the strictly semistable set are also explicitly described in \cite{ER09}, and this description is later used in \cite{MR19} to describe the intersection cohomology of $M_{Q,\bfr}^{\theta-\sst}$.

\subsection{Truncating representations}\label{Sect/Truncation}

Let $(Q,\bfm)$ be a quiver with multiplicities and $\bfr\in\bbZ_{\geq0}^{Q_0}$ be a rank vector. In this section, we define the truncation of a representation of $(Q,\bfm)$ as a representation of $Q$ (without multiplicities) and build a truncation morphism $R(Q,\bfm;\bfr)\rightarrow R(Q,\bfr)$, which is equivariant with respect to the morphism of algebraic groups $\GL_{\bfm,\bfr}\twoheadrightarrow\GL_\bfr$.

Given a representation $M$ of $(Q,\bfm)$, we wish to build an induced representation of $Q$ such that the $k$-vector space at vertex $i$ is given by $M_i/\epsilon M_i$. However, defining the induced $k$-linear maps for each arrow $a \colon i \rightarrow j$ is slightly subtle, as $M_a \colon M_i \rightarrow M_j$ does not satisfy $M_a(\epsilon M_i)\subset\epsilon M_j$, unless $f_{ji}=1$ (i.e. $m_i$ divides $m_j$). This motivates the following definition.

\begin{definition}\label{Def/intrinsicTruncation}
Let $M$ be a locally free representation of $(Q,\bfm)$. We let $\sigma(M)$ denote the representation of $(Q,\bfm)$ given by $\sigma(M)_i:=M_i$ for each $ i\in Q_0$ and $\sigma(M)_a:= M_a\epsilon^{f_{ji}-1}$ for each arrow $a \colon i\rightarrow j$. We define a representation $\tau(M)$ of $Q$, called the (classical) truncation of $M$, by setting $\tau(M)_i := M_i/\epsilon M_i$ for each vertex $i$ and, for each arrow $a \colon i \rightarrow j$, we define the corresponding $k$-linear map as the composition
\[
\begin{tikzcd}[ampersand replacement=\&]
\tau(M)_a \colon M_i/\epsilon M_i
\ar[r,"\epsilon^{f_{ji}-1}\cdot"]\ar[rr,"\sigma(M)_a",bend right] \&
\epsilon^{f_{ji}-1}M_i/\epsilon^{f_{ji}}M_i
\ar[r,"M_a"] \&
M_j/\epsilon^{f_{ij}}M_j 
\ar[r,two heads] \&
M_j/\epsilon M_j,
\end{tikzcd}
\]
with $m_i = m_{ij}f_{ji}$ and $m_j = m_{ij}f_{ij}$, and $M_a$ and $\sigma(M)_a$ are $k_{m_{ij}}$-linear. 
\end{definition}

We now give a more concrete description of this truncation, which allows us to construct an equivariant truncation morphism $\tau \colon R(Q,\bfm;\bfr)\rightarrow R(Q,\bfr)$. We will define this morphism on each arrow as follows. 

For $m_i,m_j\geq1$ and $r_i,r_j\geq1$, let $m_{ij}:=\gcd(m_i,m_j)$ and consider the $k$-vector space 
\begin{equation}\label{Eq decomposition Hom space of fixed arrow}
\Hom_{k_{m_{ij}}}\left( k_{m_i}^{\oplus r_i},k_{m_j}^{\oplus r_j}\right) \simeq \bigoplus_{\substack{1\leq r\leq r_j \\ 1\leq s\leq r_i}}\Hom_{k_{m_{ij}}}(k_{m_i},k_{m_j}),
\end{equation}
where we view elements of the left side as $r_j\times r_i$ with entries in $\Hom_{k_{m_{ij}}}\!( k_{m_i},k_{m_j})$. Below we define a truncation morphism:
\[
\tau \colon \Hom_{k_{m_{ij}}}(k_{m_i},k_{m_j})\rightarrow\Hom_k(k,k),
\]
which induces a truncation morphism (which we will also denote by $\tau$)
\[
\tau \colon \Hom_{k_{m_{ij}}}\left( k_{m_i}^{\oplus r_i},k_{m_j}^{\oplus r_j}\right) \rightarrow \Hom_k(k^{\oplus r_i},k^{\oplus r_j})
\]
by applying the above morphism $\tau$ entrywise. 

\begin{notation}\label{Not/BlockMatrices}
To describe $k_{m_{ij}}$-linear maps from $k_{m_i}$ to $k_{m_j}$, we recall that $m_i=m_{ij}f_{ji}$ and $m_j=m_{ij}f_{ij}$, and we take $k$-bases:
\begin{align}\label{Eq useful basis}
\begin{split}
    k_{m_i}= & (k\cdot\epsilon^{m_i-1}\oplus\ldots\oplus k\cdot\epsilon^{f_{ji}(m_{ij}-1)})\oplus\ldots\oplus(k\cdot\epsilon^{f_{ji}-1} \oplus\ldots\oplus k\cdot 1),\\
    k_{m_j}= &(k\cdot\epsilon^{m_j-1}\oplus\ldots\oplus k\cdot\epsilon^{f_{ij}(m_{ij}-1)})\oplus\ldots\oplus(k\cdot\epsilon^{f_{ij}-1} \oplus\ldots\oplus k\cdot 1). 
    \end{split}
\end{align}
Then multiplication by $\epsilon^{f_{ji}}$ on $k_{m_i}$ (resp.\ $\epsilon^{f_{ij}}$ on $k_{m_j}$) is given by the following $m_{ij}\times m_{ij}$ block matrix with square blocks of size $f_{ji}$ (resp.\ $f_{ij}$):
\[
\left(
\begin{array}{cccc}
0 & \Id_{f} & \ldots &  0 \\
0 & \ddots & \ddots & \vdots \\
\vdots & \ddots & \ddots  & \Id_{f} \\
0 & \ldots & 0 & 0
\end{array}
\right)
,
\]
where $f=f_{ji}$ (resp.\ $f_{ij}$). 

Consider an element $x\in\Hom_{k_{m_{ij}}}(k_{m_i},k_{m_j})$. Since $x$ is $k_{m_{ij}}$-linear, it is determined by its image on $k\cdot\epsilon^{f_{ji}-1} \oplus\ldots\oplus k\cdot 1$. Moreover, we can write $x$ with respect to the above bases \eqref{Eq useful basis} as
\begin{equation}\label{Eq linear map in terms of useful basis}
\left(
\begin{array}{cccc}
x_0 & x_1 & \ldots  & x_{m_{ij}-1} \\
0 & x_0  & \ddots & \vdots \\
\vdots & \ddots & \ddots  & x_1 \\
0 &  \ldots & 0  & x_0
\end{array}
\right)
,
\end{equation}
where the blocks have size $f_{ij}\times f_{ji}$. For $0\leq m\leq m_{ij}-1$, the matrix $x_m$ corresponds to the $k$-linear map induced by $x$ between the vector spaces:
\[
\bigoplus_{f=0}^{f_{ji}-1}k\cdot\epsilon^f
\longrightarrow
\bigoplus_{f=0}^{f_{ij}-1}k\cdot\epsilon^{mf_{ij}+f}.
\]
These linear maps completely determine $x$, as it represents a $k_{m_{ij}}$-linear map.
\end{notation}

\begin{definition}\label{Def/matrixTruncation}
With the above notation, we define a truncation morphism
\[
\begin{array}{rrcl}
\tau \colon & \Hom_{k_{m_{ij}}}(k_{m_i},k_{m_j}) & \rightarrow & \Hom_k(k,k) \\
& x & \mapsto & \left[ x_0\right]_{f_{ij},1}.
\end{array}
\]
This then defines the truncation morphism
\[
\begin{array}{rrcl}
\tau \colon & R(Q,\bfm;\bfr) & \rightarrow & R(Q,\bfr) \\
& (x_a)_{a\in Q_1} & \mapsto & (\tau(x_a))_{a\in Q_1}.
\end{array}
\]
\end{definition}

In other words, the truncation of $x\in\Hom_{k_{m_{ij}}}(k_{m_i},k_{m_j})$ is the matrix coefficient corresponding to the induced map $k\cdot\epsilon^{f_{ji}-1}\rightarrow k\cdot 1$.

\begin{example}
Suppose that $\bfm=m\mathbf{1}$ for some $m\geq1$. Then $\Hom_{k_m}(k_m,k_m)\simeq k_m$ and the truncation map recovers the usual quotient $k_m\twoheadrightarrow k$. In that case, the truncation morphism $\tau$ coincides with the one defined in \cite[Def.\ 7.2]{HHJ24}.
\end{example}

\begin{remark}\label{Rmk/motivationTruncation}
The above construction might look somewhat ad hoc. However, it is also justified by the construction of a $\bbZ$-grading on $R(Q,\bfm;\bfr)$ induced by an action of $\GG_m$ which will provide an external grading (in the sense of NRGIT). By unfolding the quiver with multiplicities $(Q,\bfm)$ as in Appendix \ref{sec/unfolding}, this external grading may be constructed from the one defined in \cite[Prop.\ 7.7]{HHJ24}. We refer to \cref{Sect/extGrad} for further details.
\end{remark}

\begin{proposition}
The truncation morphism $\tau \colon R(Q,\bfm;\bfr)\rightarrow R(Q,\bfr)$ is equivariant with respect to the morphism of algebraic groups $\GL_{\bfm,\bfr}\twoheadrightarrow\GL_\bfr$. In particular, it induces a truncation morphism between moduli stacks
\[ \tau \colon \Rep_{Q,\bfm;\bfr} \simeq [R(Q,\bfm;\bfr) / \GL_{\bfm,\bfr}] \longrightarrow \Rep_{Q,\bfr} \simeq [R(Q,\bfr)/ \GL_\bfr].\]
\end{proposition}

\begin{proof}
For each arrow $a \colon i \rightarrow j$ in $Q_1$ using the decomposition \eqref{Eq decomposition Hom space of fixed arrow}, it suffices to check that the truncation map:
\[
\tau \colon \Hom_{k_{m_{ij}}}(k_{m_i},k_{m_j})\rightarrow\Hom_k(k,k)
\]
is an equivariant homomorphism of bimodules with respect to the homomorphism of $k$-algebras $\tau \colon k_{m_i}\otimes_{k_{m_{ij}}} k_{m_j}\twoheadrightarrow k\otimes_k k$.

For $(\lambda,\mu)\in k_{m_i}\times k_{m_j}$, multiplication by $\lambda$ (resp.\ $\mu$) in the bases \eqref{Eq useful basis} corresponds to the following matrices
\[
\left(
\begin{array}{cccc}
\lambda_0 & \lambda_1 & \ldots & \lambda_{m_i-1} \\
0 & \lambda_0  & \ddots & \vdots \\
\vdots & \ddots & \ddots & \lambda_1 \\
0 & \ldots & 0 & \lambda_0
\end{array}
\right)
\text{, resp. }
\left(
\begin{array}{cccc}
\mu_0 & \mu_1 & \ldots & \mu_{m_j-1} \\
0 & \mu_0 & \ddots & \vdots \\
\vdots & \ddots & \ddots & \mu_1 \\
0 & \ldots & 0 & \mu_0
\end{array}
\right)
,
\]
where $\lambda=\lambda_0+\ldots+\lambda_{m_i-1}\epsilon^{m_i-1}$ (resp.\ $\mu=\mu_0+\ldots+\mu_{m_j-1}\epsilon^{m_j-1}$) and thus $\tau(\lambda,\mu) =(\lambda_0, \mu_0)$. For $x\in\Hom_{k_{m_{ij}}}(k_{m_i},k_{m_j})$, we can write $x$ with respect to the above bases as described in \eqref{Eq linear map in terms of useful basis}. Then checking that $\tau(\mu x\lambda)=\mu_0\tau(x)\lambda_0$ is routine.
\end{proof}

We will also make use of a section of $\tau$, which we denote by $\iota$.

\begin{definition}\label{definition iota}
With the above notation, we define:
\[
\begin{array}{rrcl}
\iota \colon & \Hom_k(k,k) & \rightarrow & \Hom_{k_{m_{ij}}}(k_{m_i},k_{m_j}) \\
& x & \mapsto & 
\left(
\begin{array}{cccc}
A_0(x) & 0 & \ldots & 0 \\
0 & A_0(x) & \ddots & \vdots \\
\vdots & \ddots & \ddots & 0 \\
0 & \ldots & 0 & A_0(x)
\end{array}
\right)
,
\end{array}
\]
where blocks have size $f_{ij}\times f_{ji}$ and $A_0(x)$ satisfies:
\[
[A_0(x)]_{f_1,f_2}:=
\left\{
\begin{array}{ll}
x & \text{if } (f_1,f_2)=(f_{ij},1), \\
0 & \text{else.}
\end{array}
\right.
\]
Working entrywise, this induces a morphism
\[
\begin{array}{rccl}
\iota \colon & R(Q,\bfr) & \rightarrow & R(Q,\bfm;\bfr) \\
& (x_a)_{a\in Q_1} & \mapsto & (\iota(x_a))_{a\in Q_1}.
\end{array}
\]
\end{definition}

\subsection{Stability conditions for quivers with multiplicities}

Let $(Q,\bfm)$ be a quiver with multiplicities and $\bfr\in\bbZ_{\geq0}^{Q_0}$ be a rank vector. We will introduce a notion of stability for locally free rank $\bfr$ representations of $(Q,\bfm)$ depending on two stability parameters $\theta,\rho\in\bbZ^{Q_0}$ satisfying $\theta\cdot\bfr=\rho\cdot\bfr=0$. The first stability parameter $\theta$ is used to determine stability of the truncated quiver representation, and the second parameter $\rho$ is additionally used to formulate a notion of stability in terms of subrepresentations with multiplicities. We will later prove that this notion of semistability corresponds to the one arising from the non-reductive GIT construction of moduli spaces of quiver representations with multiplicities (see \cref{Thm/HMcrit} below).

\begin{definition}\label{def/stabilityQuivMult}
Let $M$ be a locally free rank $\bfr$ representation of $(Q,\bfm)$. We say that $M$ is $(\theta,\rho)$-semistable (resp. stable) if the following conditions hold:
\begin{itemize}
\item for any (\emph{not necessarily locally free}) subrepresentation\footnote{The representation $\sigma(M)$ is used in \cref{Def/intrinsicTruncation} to define the truncation $\tau(M)$ and the condition $N \subset \sigma(M)$ ensures that $\tau(N) \subset \tau(M)$.} $N\subset \sigma(M)$ such that $0<\rk N<\rk M$, we have:
\[
\theta\cdot\rk N\geq0
;
\]
\item for any \emph{locally free} subrepresentation $N\subset M$ such that $0<\rk N<\rk M$ and $\theta\cdot\rk N=0$, we have:
\[
\rho\cdot\rk N\geq0
\text{ (resp. }
\rho\cdot\rk N>0
\text{).}
\]
\end{itemize}
When $\rho=0$, we simply say that $M$ is $\theta$-(semi)stable.
We call a representation \emph{naively} $(\theta,\rho)$-polystable if it is isomorphic to a direct sum of $(\theta,\rho)$-stable representations.
\end{definition}

Similarly to \cref{Def/genericStabilityParameter}, we define an open condition on $(\theta,\rho)$ which ensures that there are no strictly $(\theta,\rho)$-semistable representations.

\begin{definition}
We say that $(\theta,\rho)$ is generic with respect to $\bfr$ if, for any $0<\bfr'<\bfr$, we have $\theta\cdot\bfr'\ne0$ or we have $\theta\cdot\bfr'=0$ and $\rho\cdot\bfr'\ne0$.
\end{definition}

\begin{remark}\label{Rmk/genericTheta}
\begin{enumerate}[label=\roman*)]
\item When $\rho=0$, a locally free representation $M$ is $(\theta,0)$-semistable if and only if its truncation $\tau(M)$ is $\theta$-semistable. However, a $(\theta,0)$-semistable representation $M$ is $(\theta,0)$-stable if and only if it does not have any locally free subrepresentation $N$ such that $0<\rk N<\rk M$ and $\theta\cdot\rk N=0$. This condition is weaker than requiring that $\tau(M)$ be $\theta$-stable, as not all subrepresentations of $\tau(M)$ come from a locally free subrepresentation of $M$. Indeed this is a general phenomena for relative NRGIT: in general there is not an induced map between the geometric quotients of the stable loci (see \cite[Remark 3.20]{HHJ}).

\item When $\theta$ is generic with respect to $\bfr$, then the second condition is empty and $\rho$ does not play any role, as all $(\theta,\rho)$-semistable representations have $\theta$-stable truncation. In particular, $(\theta,\rho)$ is generic with respect to $\bfr$.

\item More generally, for $\eta$ small enough (depending only on $\theta$, $\rho$ and $\bfr$) and for all $0<\bfr'<\bfr$, we have that, $(\theta+\eta\rho)\cdot\bfr'\ne0$ if and only if either $\theta\cdot\bfr'\ne0$ or $\theta\cdot\bfr'=0$ and $\rho\cdot\bfr'\ne0$. Thus, $(\theta,\rho)$ is generic with respect to $\bfr$ if and only if $\theta+\eta\rho$ is generic with respect to $\bfr$. In that case, all $(\theta,\rho)$-semistable representations of $(Q,\bfm)$ are $(\theta,\rho)$-stable. Note that there exist generic parameters $(\theta,\rho)$ with respect to $\bfr$ if and only if $\bfr$ is indivisible.
\end{enumerate}
\end{remark}

We now define Jordan-H\"older filtrations and associated graded representations in the spirit of \cite{Kin94}. Since locally free quiver representations with multiplicities do not form an abelian category, the Jordan-H\"older factors of an object are not well-defined a priori (see for instance \cite{Eno22}). Nevertheless, we can still define a suitable notion of polystable representations.

\begin{definition}
Let $M$ be a $(\theta,\rho)$-semistable locally free rank $\bfr$ representation of $(Q,\bfm)$. A \emph{naive} Jordan-H\"{o}lder filtration of $M$ is a finite filtration $F_\bullet$ of $M$
\[
0\subsetneq F_nM\subsetneq\ldots\subsetneq F_1M\subsetneq F_0M=M
\]
by locally free subrepresentations, such that $\theta\cdot\rk F_iM=\rho\cdot\rk F_iM=0$ and $F_{i-1}M/F_iM$ is $(\theta,\rho)$-stable for all $1\leq i\leq n$.
The associated graded representation
\[
\mathrm{gr}_{F_\bullet}M:=\bigoplus_{i=1}^nF_{i-1}M/F_iM
\]
is naively $(\theta,\rho)$-polystable, by construction.
\end{definition}

\begin{remark}
Any $(\theta,\rho)$-semistable representation admits a \emph{naive} Jordan-H\"{o}lder filtration, by induction on the rank vector.
\end{remark}

\begin{definition}\label{Def/JHfiltrations}
Let $M$ be a $(\theta,\rho)$-semistable, locally free, rank $\bfr$ representation of $(Q,\bfm)$. We say that $M$ is $(\theta,\rho)$-polystable if any \emph{naive} Jordan-H\"older filtration $F_\bullet$ of $M$ splits, i.e. there is an isomorphism $M\simeq\mathrm{gr}_{F_\bullet}M$. In particular, $M$ is naively polystable.

A Jordan-H\"older filtration is defined as a \emph{naive} Jordan-H\"older filtration whose associated graded representation is $(\theta,\rho)$-polystable.

Two $(\theta,\rho)$-semistable rank $\bfr$ representations $M_1,M_2$ of $(Q,\bfm)$ are S-equivalent if they admit Jordan-H\"older filtrations $F_{1\bullet}$ and $F_{2\bullet}$ such that $\mathrm{gr}_{F_{1\bullet}}M_1\simeq\mathrm{gr}_{F_{2\bullet}}M_2$.
\end{definition}

\begin{remark}
In \cref{Prop/JHassociatedGraded}, we will prove that any $(\theta,\rho)$-semistable representation admits a Jordan-H\"older filtration.

When the stability parameter $\theta$ is generic with respect to $\bfr$, any naively $\theta$-polystable representation of rank $\bfr$ is $\theta$-stable. Thus, all \emph{naive} Jordan-H\"older filtrations are trivial and any such representation is also $\theta$-polystable. For general $\theta$ and $\rho$, it would be interesting to know whether any naively $(\theta,\rho)-$polystable representation is $(\theta,\rho)-$polystable. 
\end{remark}

One can also introduce a notion of (semi)stability for $\bff$-framed representations of $(Q,\bfm)$, but this only depends on the parameter $\theta\in\bbZ^{Q_0}$. The stability parameter $\rho$ has no effect on the notion of (semi)stability. Indeed the stability parameter $\widehat{\theta}$ for the framed quiver $Q_{\bff}$ that appears in \cref{thm/classicalframedquiver} is generic (in the sense of \cref{Def/genericStabilityParameter}), and so we cannot have $\widehat{\theta} \cdot \rk N = 0$ for some $0 < \rk N < \rk M$ (see \cref{Rmk/genericTheta}).

\begin{definition}\label{Def/stabilityFramedQuivMult}
Let $(M,b)$ be a $\bff$-framed representation of $(Q,\bfm)$, with rank vector $\bfr$. We say that $(M,b)$ is $\theta$-semistable if:
\begin{itemize}
\item $M$ is $\theta$-semistable;
\item for any (\emph{not necessarily locally free}) subrepresentation $N\subset\sigma(M)$ such that $0<\rk N< \rk M$ and $\mathrm{Im}(b)\subset N$, we have $\theta\cdot\rk N>0$.
\end{itemize}
\end{definition}

\subsection{Moment maps of quivers with multiplicities}

We now recall basic facts on moment maps associated to quivers with multiplicities. Consider the cotangent bundle $\rmT^*R(Q,\bfm;\bfr)\simeq R(Q,\bfm;\bfr)\times R(Q,\bfm;\bfr)^{\vee}$. The above $\GL_{\bfm,\bfr}$-action on $R(Q,\bfm;\bfr)$ induces an action on $\rmT^*R(Q,\bfm;\bfr)$, which is hamiltonian. Therefore, there exists a moment map $\mu_{(Q,\bfm),\bfr} \colon \rmT^*R(Q,\bfm;\bfr)\rightarrow\gl_{\bfm,\bfr}^{\vee}$, which is characterised by the following property:
\[
\forall (x,y,\xi)\in \mathrm{T}^*R(Q,\bfm;\bfr)\times\gl_{\bfm,\bfr},\ \langle\xi\cdot x,y\rangle=\langle\xi,\mu_{(Q,\bfm),\bfr}(x,y)\rangle
.
\]
We give an explicit formula for $\mu_{(Q,\bfm),\bfr}$ below, which already appeared in \cite{Yam10,GLS17a,HWW23}.

\begin{notation}\label{Not/TracePairing}
For $m\geq1$, we equip the ring $k_{m}$ with the following $k$-linear form:
\[
\begin{array}{rccl}
r \colon  & k_m & \rightarrow & k \\
 & \lambda(\epsilon)=\sum_k\lambda_k\cdot \epsilon^k & \mapsto &\underset{\epsilon=0}{\mathrm{Res}}\left( \epsilon^{-m}\lambda(\epsilon)\right):=\lambda_{m-1}
\end{array}
.
\]
For $r_{1},r_{2}\geq1$, the trace pairing
\[
\begin{array}{ccl}
\Hom_{k_m}(k_m^{\oplus r_1},k_m^{\oplus r_2}) \times \Hom_{k_m}(k_m^{\oplus r_2},k_m^{\oplus r_1}) & \rightarrow & k \\
(A,B) & \mapsto & \langle A,B\rangle:=r\left(\mathrm{Tr}(AB)\right)
\end{array}
\]
induces an isomorphism $\Hom_{k_{m}}(k_{m}^{\oplus r_{1}},k_{m}^{\oplus r_{2}})^{\vee}\simeq\Hom_{k_{m}}(k_{m}^{\oplus r_{2}},k_{m}^{\oplus r_{1}})$, see for instance \cite[\S 8]{GLS17a}.
\end{notation}

We thus directly obtain an isomorphism $\gl_{\bfm,\bfr}^{\vee}\simeq\gl_{\bfm,\bfr}$. Given an arrow $a \colon i\rightarrow j$ in $Q_1$, we use the isomorphism
\[
\Hom_{k_{m_{ij}}}(k_{m_{i}}^{\oplus r_{i}},k_{m_{j}}^{\oplus r_{j}})\simeq\Hom_{k_{m_{ij}}}(k_{m_{ij}}^{\oplus f_{ji}r_{i}},k_{m_{ij}}^{\oplus f_{ij}r_{j}})
\]
and the trace pairing to obtain an isomorphism $R(Q,\bfm;\bfr)^{\vee}\simeq R(Q^{\mathrm{op}},\bfm;\bfr)$, where $Q^{\mathrm{op}}$ is the opposite quiver of $Q$. Thus if $\overline{Q}$ denotes the doubled quiver, then this provides an isomorphism $\rmT^*R(Q,\bfm;\bfr)\simeq R(\overline{Q},\bfm;\bfr)$.

\begin{proposition}\emph{\cite[\S 1.2]{Ver24}}
 Under the above identifications, the moment map is described as follows. For $(x,y)\in R(\overline{Q},\bfm;\bfr)$:
\[
\mu_{(Q,\bfm),\bfr}(x,y)=
\left(
\sum_{\substack{a\in Q_1 \\ a \colon j\rightarrow i}}
\sum_{f=0}^{f_{ji}-1}
\epsilon^{f}x_ay_a\epsilon^{f_{ji}-1-f}
-
\sum_{\substack{a\in Q_1 \\ a \colon i\rightarrow j}}
\sum_{f=0}^{f_{ji}-1}
\epsilon^{f}y_ax_a\epsilon^{f_{ji}-1-f}
\right)_{i\in Q_0}
.
\]
\end{proposition}

\begin{remark}
Given a quiver with multiplicities $(Q,\bfm)$ and a rank vector $\bfr\in\bbZ_{\geq0}^{Q_0}$, it is natural to ask whether the moment map is compatible with truncation, i.e.\ whether we have:
\[
\mu_{Q,\bfr}\circ\tau=\tau\circ\mu_{(Q,\bfm),\bfr}.
\]
This is true when multiplicities are constant (that is, $\bfm=m\mathbf{1}$ for some $m\geq1$), but not in general. See \cref{Lem/extGrad3} for a more detailed analysis.
\end{remark}

\begin{notation}
Consider the Lie algebra $\mathfrak{d}_{\bfm}\subset\gl_{\bfm,\bfr}$ of the subgroup $\Delta_{\bfm}\subset\GL_{\bfm,\bfr}$. Then the Lie algebra of $G_{\bfm,\bfr}$ is the quotient $\gl_{\bfm,\bfr}/\mathfrak{d}_\bfm$. We denote by $\gl_{\bfm,\bfr}^0\subset\gl_{\bfm,\bfr}$ the orthogonal Lie algebra of $\mathfrak{d}_{\bfm}$ under the above pairing. Its dimension is $\dim\gl_{\bfm,\bfr}-\delta$. Note that, since $\Delta_{\bfm}$ acts trivially on $R(Q,\bfm;\bfr)$, the moment map restricts to:
\[
\mu_{(Q,\bfm),\bfr} \colon \mathrm{T}^*R(Q,\bfm;\bfr)\rightarrow (\gl_{\bfm,\bfr}/\mathfrak{d}_\bfm)^{\vee}\simeq\mathfrak{gl}_{\bfm,\bfr}^0.
\]
\end{notation}

\section{Overview of results from relative NRGIT}\label{sec/overview rel NRGIT}

In this section, we review the main results of \cite{HHJ} regarding non-reductive geometric invariant theory in the relative affine setting. We recall the explicit construction of NRGIT quotients in the setting we need and the Hilbert-Mumford description of the associated (semi)stable locus. Throughout this section, we assume that $k$ is an algebraically closed field of characteristic zero.

\subsection{Gradings of equivariant actions}

Let us introduce the notion of an equivariant action and a grading of such an action.

\begin{definition}
Given a homomorphism $\varphi \colon G \rightarrow H$ of affine algebraic groups over $k$, an \emph{equivariant action} of $\varphi \colon G \rightarrow H$ on a morphism $f \colon X \rightarrow Z$ of $k$-schemes is given by a pair of actions $(\sigma_X \colon G \times X \rightarrow X,\sigma_Z \colon H \times Z \rightarrow Z)$ such that $f \circ \sigma_X = \sigma_Z \circ (\varphi \times f)$.  If $H = \mathrm{Id}$, then we say $G$ acts fibrewise on $f$.
\end{definition}

Let $G$ be a linear algebraic group; then we can choose to express $G$ as a semi-direct product of its unipotent radical $U \subset G$ and a Levi subgroup $R$, as $k$ has characteristic zero. There is a surjection $G = U \rtimes R \twoheadrightarrow R$. In this article we will only be interested in constructing quotients for equivariant actions of $G = U \rtimes R \twoheadrightarrow R$ on affine morphisms $f \colon X \rightarrow Z$.

A key assumption that enables the construction of NRGIT quotients is the presence of a multiplicative group that grades this action in the following sense.

\begin{definition}
We say that an equivariant action of a homomorphism $\varphi \colon G \rightarrow H$ of affine algebraic groups on a morphism $f \colon X \rightarrow Z$ of $k$-schemes is \begin{enumerate}
    \item \emph{(internally) graded} by a subgroup $\GG_m \subset G$ if the following properties hold:
    \begin{enumerate}[label=\roman*)]
        \item $\varphi(\GG_m)$ acts trivially on $Z$ (thus $\GG_m$ acts fibrewise on $f$);
        \item For the fibrewise $\GG_m$-action on $f$, the morphism $f$ induces an isomorphism $X^{\GG_m} \cong Z$ and $\lim_{t \rightarrow 0} t \cdot x$ exists for all $x \in X$;
        \item For the action of $\GG_m$ on $G$ by conjugation, we have that $\varphi$ induces an isomorphism $G^{\GG_m} \cong H$ and $\lim_{t \rightarrow 0} tgt^{-1}$ exists for all $g \in G$.
    \end{enumerate}
    \item \emph{(externally) graded} by a subgroup $\GG_m \subset \Aut(G)$ if $\varphi$ admits an extension to $\widetilde{\varphi} \colon \widetilde{G}:= G \rtimes \GG_m \rightarrow H$ and the action of $\varphi$ on $f$ extends to a graded equivariant action of $\widetilde{\varphi} \colon \widetilde{G} \rightarrow H$ on $f$.
\end{enumerate}
\end{definition}

\begin{remark}\label{Rmk/extGradedAction}
Suppose that $X$ is an affine $G$-scheme and that there exists a subgroup $\GG_m\subset G$ such that $\lim_{t \rightarrow 0} t \cdot x$ (resp.\ $\lim_{t \rightarrow 0} tgt^{-1}$) exists for all $x\in X$ (resp.\ for all $g\in G$). Then the mapping $x \mapsto \lim_{t \rightarrow 0} t\cdot x$ induces a morphism $f \colon X\rightarrow Z:=X^{\GG_m}$ which is equivariant with respect to the morphism of algebraic groups $G\rightarrow H:=G^{\GG_m}$ defined in the same way. By construction, this action is (internally) graded.
\end{remark}

\subsection{Relative affine NRGIT using an external grading}

In this paper, we only use the construction of NRGIT quotients in the externally graded case; details on the internal case can be found in \cite[$\S$5.1]{HHJ}. There are two ways to describe the NRGIT quotient in the externally graded setting: the first we give in \cref{def NRGIT quotient} and the second is described in \cref{Thm/extGradQuot} \ref{Item 2 Thm/extGradQuot}, where we realise this NRGIT quotient as an open subset inside an associated internally graded GIT quotient of $\widetilde{X} := X \times \mathbb{A}^1$ where the extra $\mathbb{A}^1$ is added to compensate for additionally quotienting by the grading $\GG_m$. Properties of both constructions are needed in the rest of the paper: the first construction gives the relative structure of the quotient and is needed for purity results later on, and the second gives a Hilbert--Mumford criterion that enables us to interpret (semi)stability.

\paragraph*{Relative affine setup (externally graded case)}
In this subsection, we suppose that $G = U \rtimes R$ is a linear algebraic group with unipotent radical $U$ and we have an equivariant action of $G \twoheadrightarrow R$ on an affine morphism $f \colon X \rightarrow Z$ of $k$-schemes, which is externally graded by a subgroup $\GG_m \subset \Aut(G)$. \\

Since $f$ is affine, we can write $X = \rSpec_Z \cA$ where $\cA = f_* \cO_X$ is a sheaf of $\cO_Z$-algebras. Furthermore, the fibrewise $\GG_m$-action on $f$ induces a $\ZZ$-grading $\cA = \oplus_i \cA_i$ and as this action is graded, this is concentrated in non-positive degrees $\cA = \oplus_{i \leq 0} \cA_i$ and $\cA_0 = \cO_Z$.

\begin{definition}\label{def NRGIT quotient}
    In the above situation, suppose we have a good $R$-quotient $q \colon Z' \rightarrow W$ of an open set $j \colon Z' \hookrightarrow Z$. Let $f' \colon X' \rightarrow Z'$ denote the base change of $f$ along $j$. \\
    For any character $\rho \colon G \rightarrow \GG_m$, we let $(q_*j^*\cA)^G_{\rho}$ denote the subsheaf of $q_*j^*\cA$ whose sections are $\rho$-semi-invariants; that is, for any open $U \subset W$, we have:
    \[
    (q_*j^*\cA)^G_{\rho}(U)=
    \left\{
    h \in \cO_{X'}((q \circ f')^{-1}(U))\ \left\vert\
    \begin{array}{l}
    \forall g \in G,\\
    \forall x \in (q \circ f')^{-1}(U),
    \end{array}
    \ h(g \cdot x) = \rho(g) h(x)
    \right.
    \right\}
    \]
    This defines a sheaf \[q_*j^*\cA^{G,\rho} = \bigoplus_{n \geq 0} (q_*j^*\cA)^G_{\rho^n}\] of graded $\cO_W$-algebras, which we refer to as the sheaf of \emph{$\rho$-twisted semi-invariants}. \\
    The inclusion $q_*j^*\cA^{G,\rho} \hookrightarrow \oplus_{n \geq 0} q_*j^*\cA$ of sheaves of graded $\cO_W$-algebras induces a rational map over $W$
    \[  X' = \rSpec_W q_{*} j^*\cA \dashrightarrow X/\!/_{\hspace{-2pt}q,\rho} G:= \rProj_W q_*j^*\cA^{G,\rho} \]
    whose domain of definition is denoted $X^{q,\rho-\sst}$ and called the \emph{relative $\rho$-twisted semistable set (with respect to $q$)}. We call the induced $W$-morphism \[\pi \colon X^{q,\rho-\sst} \rightarrow X/\!/_{\hspace{-2pt} q,\rho} G\] the \emph{$\rho$-twisted relative affine GIT quotient (with respect to $q$)}.
\end{definition}

If we take $\rho$ to be the trivial character, we omit this from the notation and call $X/\!/_{\hspace{-2pt} q} G$ the \emph{(untwisted) relative affine GIT quotient (with respect to $q$)}. By construction $X/\!/_{\hspace{-2pt} q,\rho} G$ is projective over $X/\!/_{\hspace{-2pt} q} G$, which is itself affine over $W$.

We can now state the main properties of this NRGIT quotient. For this, we recall the following notion.

\begin{notation}
A one-parameter subgroup (or 1PS for short) of a group $G$ is any non-trivial homomorphism $\lambda \colon \GG_m\rightarrow G$.
\end{notation}

\begin{theorem}\label{Thm/extGradQuot}\emph{\cite[Thm 1.3]{HHJ}}
    For an equivariant action of $G = U \rtimes R \twoheadrightarrow R$ on an affine morphism $f \colon  X \rightarrow Z$ which is externally graded by $\GG_m \subset \Aut(G)$, suppose we have a good $R$-quotient $q \colon Z' \rightarrow W$ of an open subset $Z' \subset Z$ and that
    \[ \dim \Stab_{U}(z) = 0 \quad \text{for all } z \in Z'.\]
    For a character $\rho \colon G \rightarrow \GG_m$, the following statements hold.
    \begin{enumerate}[label=\roman*)]
        \item The $\rho$-twisted relative affine GIT quotient $\pi \colon X^{q,\rho-\sst} \rightarrow X/\!/_{\hspace{-2pt}q,\rho} G$ is a good $G$-quotient, and  $X/\!/_{\hspace{-2pt}q,\rho} G$ is projective-over-affine over $W$.
        \item\label{Item 2 Thm/extGradQuot} There is a relative projective completion $X/\!/_{\hspace{-2pt}q,\rho} G \hookrightarrow \widetilde{X}/\!/_{\hspace{-2pt} q,\widetilde{\rho} \hspace{1pt}} \widetilde{G}$ over $W$ given by a GIT quotient of an associated internally graded equivariant action of 
        \[ \widetilde{\varphi} \colon \widetilde{G} : = G \rtimes \GG_m \twoheadrightarrow \widetilde{R} : = R \times \GG_m \text{ on } \widetilde{f} \colon \widetilde{X} := X \times_Z \mathbb{A}^1_Z \rightarrow Z\]  
        with respect to a character $\widetilde{\rho} = (\rho,N)$ for $N>\!>0$. This inclusion is induced by the $G$-equivariant inclusion $X \hookrightarrow \widetilde{X}$ given by $x \mapsto (x,1)$.
        \item The restriction of $\pi$ to the open stable set $X^{q,\rho-\st}= X \cap \widetilde{X}^{q,\widetilde{\rho}-\st,\widetilde{G}}$ is a geometric $G$-quotient whose image is quasi-projective over $W$.
        \item The (semi)stable loci $X^{q,\rho-\mathrm{(s)s}}$ admit explicit Hilbert--Mumford descriptions as follows: we have $X^{q,\rho-\mathrm{(s)s}}= X \cap \widetilde{X}^{q,\widetilde{\rho}-\mathrm{(s)s},\widetilde{G}}$, and thus 
\[X^{q,\rho-\mathrm{(s)s}}  = \left\{ x \in X' \left| \begin{array}{c} \langle \widetilde{\rho}, \widetilde{\lambda} \rangle \: \geqp \: 0 \text{ for all 1PSs } \widetilde{\lambda}  \colon \GG_m \rightarrow \widetilde{G} \\ \text{such that } \lim_{t \rightarrow 0} \widetilde{\lambda}(t) \cdot (x,1)\text{ exists in } \widetilde{X}'  \end{array} \right. \right\}
\]
        where $\widetilde{X}':= \widetilde{f}^{-1}(Z') = X' \times_{Z'} \mathbb{A}^1_{Z'}$.
        \item The closed points of the good quotient $X/\!/_{\hspace{-2pt} q, \rho} G$ are in bijection with $G$-orbits in $X^{q,\rho-\sst}$ which are closed (so-called polystable orbits). Moreover, a $G$-orbit in $X^{q,\rho-\sst}$ is closed if and only if it is closed under all flows along 1-PSs $\lambda$ of $G$. In particular, the limit as $ t \rightarrow 0$ of the action of $\lambda(t) \cdot x$ exists in $X^{q,\rho-\sst}$ if and only if this limit exists in $X'$ and $\langle\lambda,\rho\rangle=0$.
        \item Every semistable $G$-orbit contains a unique closed (i.e.\ polystable) orbit in its closure in $X^{q,\rho-\sst}$. From any semistable point, one can flow to a polystable point in the orbit closure using a 1PS of $G$. We say two semistable points are S-equivalent if their orbits have the same polystable orbit in their closure. The closed points of $X/\!/_{\hspace{-2pt} q, \rho} G$ are in bijection with S-equivalence classes of semistable $G$-orbits.
    \end{enumerate}
\end{theorem}

\begin{remark}\label{rmk/alt descr HM}
The Hilbert--Mumford description in the external case follows from the following Hilbert--Mumford description for the internally graded action of $\widetilde{\varphi}$ on $\widetilde{f}$:
\[ \widetilde{X}^{q,\widetilde{\rho}-\mathrm{(s)s}, \widetilde{G}} = \left\{ x \in \widetilde{X}' \left| \begin{array}{c} \langle \widetilde{\rho}, \widetilde{\lambda} \rangle \: \geqp \: 0 \text{ for all 1PSs } \widetilde{\lambda}  \colon \GG_m \rightarrow \widetilde{G} \\ \text{such that } \lim_{t \rightarrow 0} \lambda(t) \cdot x \text{ exists in } \widetilde{X}' \end{array} \right. \right\}. \]
Furthermore, we can conjugate any 1PS of $\widetilde{G}$ by an element in $U$, so its image is contained in $\widetilde{R}$, thus we have
\[ \widetilde{X}^{q,\rho-\mathrm{(s)s}, \widetilde{G}} = \bigcap_{u \in U} u X^{q,\rho-\mathrm{(s)s},\widetilde{R}}.\]
\end{remark}

The above results were stated in the optimal setting, where all points in the open subset $Z' \subset Z$ have trivial unipotent stabilisers. However, relative NRGIT can be applied more generally when this fails as explained in the following remark.

\begin{remark}[NRGIT with non-trivial unipotent stabilisers]\label{rmk NRGIT with non-trivial stabilisers}
 If the generic $U$-stabiliser in $Z'$ is trivial, but not all points in $Z'$ have trivial stabiliser, then one can base change $f$ along the open set in $Z'$ where the unipotent stabilisers are trivial and apply these results to this base-changed morphism, where now the quotient $q$ is replaced by its restriction. This avoids complicated blow-up procedures to get rid of the non-generic stabilisers, but gives something which lies only over an open subset in $W$.
    
    If the generic $U$-stabiliser in $Z'$ has positive dimension and we assume first for simplicity that $U$ is abelian, then one can also construct relative NRGIT quotients provided the unipotent stabilisers have constant dimension (which again can be achieved by base changing $f$ along the open set in $Z'$ where the unipotent stabilisers have minimal dimension). The idea is to locally find complementary subgroups $U' \subset U$ that act freely (here the assumption that $U$ is abelian is used so it suffices to find a complementary subspace to the stabiliser inside the Lie algebra) and take the quotients by these complementary subgroups instead. If $U$ is not abelian, then one fixes a sequence of normal subgroups whose quotients are abelian, and deals with this iteratively by fixing the dimensions of the stabilisers for the subgroups in this sequence. In particular, once one fixes the dimensions of all these stabilisers, one can construct a relative NRGIT quotient (see \cite[$\S$4.3]{HHJ} for details).
\end{remark}

\section{Moduli spaces of quiver representations with multiplicities}\label{sec/constr moduli spaces}

In this section, we construct moduli spaces of semistable quiver representations with multiplicities and prove \cref{first main thm}. To do this, we check step by step that the equivariant action \eqref{equiv action quivers} satisfies the requirements of \cref{Thm/extGradQuot}, and then moduli-theoretically interpret GIT (semi)stability. Throughout the section, we fix a quiver with multiplicities $(Q,\bfm)$ and a rank vector $\bfr\in\bbZ_{\geq0}^{Q_0}$.

\subsection{External grading and truncation}\label{Sect/extGrad}

We start by constructing an external grading for the action $\GL_{\bfm,\bfr}$-action on $R(Q,\bfm;\bfr)$ 
and show that it lifts to an external grading of the equivariant action of the homomorphism $\GL_{\bfm,\bfr}\twoheadrightarrow\GL_\bfr$ on the truncation morphism $\tau \colon R(Q,\bfm;\bfr)\rightarrow R(Q,\bfr).$

We consider a modified version of the external grading \cite[\S 7]{HHJ24}, which is tailored to quivers with varying multiplicities. We first set up some notation.

\begin{notation}
Let $m,r\geq1$. Using \cref{Not/BlockMatrices}, we view $\GL_{m,r}$ as the following subgroup of $\GL_{mr}$:
\begin{align}\label{Eqn/blockDecompositionGroup}
\GL_{m,r}=
\left\{
\left(
\begin{array}{cccc}
g_0 & g_1 & \ldots & g_{m-1} \\
0 & g_0 & \ddots & \vdots \\
\vdots & \ddots & \ddots & g_1 \\
0 & \ldots & 0 & g_0
\end{array}
\right)
\ \left\vert\ 
\begin{array}{l}
g_0\in\GL_r \\
\forall 1\leq l\leq m-1,\ g_l\in\gl_r
\end{array}
\right.
\right\}
\end{align}
Let $\Phi_{m,r} \colon \GG_m\rightarrow\GL_{mr}$ be the morphism given by:
\[
\Phi_{m,r}(t)=
\left(
\begin{array}{cccc}
t^{m-1}\Id & 0 & \ldots & 0 \\
0 & \ddots & \ddots & \vdots \\
\vdots & \ddots & t\Id & 0 \\
0 & \ldots & 0 & \Id
\end{array}
\right)
.
\]
Let $*$ denote the following $\GG_m$-action on $\GL_{m,r}$ by group automorphisms:
\[
\begin{array}{ccl}
\GG_m\times\GL_{m,r} & \rightarrow & \GL_{m,r} \\ 
(t,g) & \mapsto & t*g:=\Phi_{m,r}(t)g\Phi_{m,r}(t)^{-1}.
\end{array}
\]
For $m_1, m_2 \geq 1$, let $m_{12} = \gcd(m_1,m_2)$ and $f_{ji} = m_{i}/m_{ij}$ as in \cref{notation gcd of multiplicities}. For $r_1,r_2 \geq 1$, we view $\Hom_{k_{m_{12}}}(k_{m_1}^{\oplus r_1},k_{m_2}^{\oplus r_2})$ as the following subspace of $\Hom_{k}(k^{\oplus m_1r_1},k^{\oplus m_2r_2})$:

\begin{align}\label{Eqn/blockDecompositionRep}
\Hom_{k_{m_{12}}}(k_{m_1}^{\oplus r_1},k_{m_2}^{\oplus r_2})
=
\left\{
\left.
\left(
\begin{array}{cccc}
x_0 & x_1 & \ldots &  x_{m_{12}-1} \\
0 & x_0 &  \ddots & \vdots \\
\vdots &  \ddots & \ddots & x_1 \\
0 & 0 & \ldots  & x_0
\end{array}
\right)
\ \right\vert\ 
\begin{array}{l}
\forall \: 0\leq l\leq m_{ij}-1, \\
x_l\in\Hom_{k}(k^{\oplus r_1f_{21}},k^{\oplus r_2f_{12}})
\end{array}
\right\}
.
\end{align}
Below we will often further subdivide the blocks $x_i$ into blocks of size $r_2\times r_1$.
\end{notation}

\begin{lemma}\label{Lem/extGrad1}
Let $m_1,r_1,m_2,r_2\geq1$. Let also $\alpha_1,\alpha_2\in\bbZ$. Then the action
\[
\begin{array}{rccl}
*_{\alpha_1,\alpha_2} \colon & \GG_m\times\Hom_{k}(k^{\oplus m_1r_1},k^{\oplus m_2r_2}) & \rightarrow & \Hom_{k}(k^{\oplus m_1r_1},k^{\oplus m_2r_2}) \\ &
(t,x) & \mapsto & \Phi_{m_2,r_2}(t^{\alpha_2})x\Phi_{m_1,r_1}(t^{\alpha_1})^{-1}
\end{array}
\]
preserves the subspace $\Hom_{k_{m_{12}}}(k_{m_1}^{\oplus r_1},k_{m_2}^{\oplus r_2})\subset \Hom_{k}(k^{\oplus m_1r_1},k^{\oplus m_2r_2})$ if and only if we have $m_{12}=1$ or $\alpha_1m_1=\alpha_2m_2$.
\end{lemma}

\begin{proof}
Note that, when $m_{12}=1$, we have that:
\[
\Hom_{k_{m_{12}}}(k_{m_1}^{\oplus r_1},k_{m_2}^{\oplus r_2})
=
\Hom_{k}(k^{\oplus m_1r_1},k^{\oplus m_2r_2})
\]
and the left-hand side is trivially preserved under the above $\GG_m$-action. Let us then assume that $m_{12}\ne1$.

Any matrix $x\in\Hom_{k_{m_{12}}}(k_{m_1}^{\oplus r_1},k_{m_2}^{\oplus r_2})$ as in \eqref{Eqn/blockDecompositionRep} can be further subdivided into blocks of size $r_2\times r_1$. The action $*_{\alpha_1,\alpha_2}$ scales each block of size $r_2\times r_1$ with a fixed weight, which decreases by $-\alpha_1$ when moving from right to left and increases by $\alpha_2$ when moving from bottom to top among blocks.

Therefore, the subspace $\Hom_{k_{m_{12}}}(k_{m_1}^{\oplus r_1},k_{m_2}^{\oplus r_2})$ is preserved under $*_{\alpha_1,\alpha_2}$ if and only if the aforementioned weights remained unchanged by moving simultaneously $f_{21}$ blocks to the left and $f_{12}$ blocks upwards, i.e.\ $-\alpha_1f_{21}+\alpha_2f_{12}=0$, or equivalently $\alpha_1m_1=\alpha_2m_2$.
\end{proof}

To define the external grading, consider $\bm{\alpha} = (\alpha_i)_{i \in Q_0}\in\bbZ_{>0}^{Q_0}$ such that
\begin{equation}\label{eqns for vertex weights in ext grading}
  \forall i,j\in Q_0,\ \alpha_im_i=\alpha_jm_j  
\end{equation}
and $\bm{\beta}=(\beta_a)_{a\in Q_1}\in\bbZ^{Q_1}$. We define the following actions:
\[
\begin{array}{rccl}
*_{\bm{\alpha}} \colon & \GG_m\times \GL_{\bfm,\bfr} & \rightarrow & \GL_{\bfm,\bfr} \\ &
\left(t,(g_i)_{i\in Q_0}\right) & \mapsto & \left(\Phi_{m_i,r_i}(t^{\alpha_i})g_i\Phi_{m_i,r_i}(t^{\alpha_i})^{-1}\right)_{i\in Q_0},
\end{array}
\]
\[
\begin{array}{rccl}
*_{\bm{\alpha},\bm{\beta}} \colon & \GG_m\times R(Q,\bfm;\bfr) & \rightarrow & R(Q,\bfm;\bfr) \\ &
\left(t,(x_a)_{a\in Q_1}\right) & \mapsto & \left(t^{\beta_a}\cdot\Phi_{m_j,r_j}(t^{\alpha_j})x_a\Phi_{m_i,r_i}(t^{\alpha_i})^{-1}\right)_{\substack{a\in Q_1 \\ a \colon i\rightarrow j}}.
\end{array}
\]

\begin{lemma}\label{Lem/extGrad2}
For the action $*_{\bm{\alpha},\bm{\beta}}$, the following statements hold.
\begin{enumerate}[label=(\roman*)]
\item The action $*_{\bm{\alpha},\bm{\beta}}$ extends to an action 
\[
\begin{array}{rccl}
*_{\bm{\alpha},\bm{\beta}} \colon & (\GL_{\bfm,\bfr}\rtimes_{\bm{\alpha}}\GG_m)\times R(Q,\bfm;\bfr) & \rightarrow & R(Q,\bfm;\bfr) \\ &
\left((g,t),(x_a)_{a\in Q_1}\right) & \mapsto & g\cdot(t*_{\bm{\alpha},\bm{\beta}}x),
\end{array}
\]
where the semi-direct product $\GL_{\bfm,\bfr}\rtimes_{\bm{\alpha}}\GG_m$ is taken with respect to $*_{\bm{\alpha}}$.

\item The $\GL_{\bfm,\bfr}\rtimes_{\bm{\alpha}}\GG_m$-action $*_{\bm{\alpha},\bm{\beta}}$ induces an externally graded action (see \cref{Rmk/extGradedAction}) if and only $\beta_a-\alpha_i(f_{ji}-1)\geq0$ for all $a \colon i \rightarrow j$ in $Q_1$.

\item If $\beta_a\geq \alpha_i(f_{ji}-1)$ for all $a \colon i \rightarrow j$ in $Q_1$, then we have a $\GL_\bfr$-equivariant isomorphism:
\[
R(Q,\bfm;\bfr)^{\GG_m}\simeq R(Q,\bfr)
\]
if and only if we have $\beta_a= \alpha_i(f_{ji}-1)$ for all $a \colon i \rightarrow j$ in $Q_1$. In that case, we have an induced externally graded equivariant action
\[
(\GL_{\bfm,\bfr}\twoheadrightarrow\GL_\bfr)\curvearrowright (\tau \colon R(Q,\bfm;\bfr)\rightarrow R(Q,\bfr)).
\]
\end{enumerate}
\end{lemma}

\begin{proof}
Due to the assumption \eqref{eqns for vertex weights in ext grading}, the $\GG_m$-action $*_{\bm{\alpha},\bm{\beta}}$ is well-defined by \cref{Lem/extGrad1} as
\[
t*_{\bm{\alpha},\bm{\beta}}(x_a)_{a\in Q_1}=(t^{\beta_a}\cdot (t*_{\alpha_i,\alpha_j}x))_{\begin{subarray}{l}a\in Q_1 \\ a \colon i \rightarrow j\end{subarray}}.
\]
Part (i) holds as the $\GG_m$-action $*_{\bm{\alpha},\mathbf{0}}$ extends to a $\GL_{\bfm,\bfr}\rtimes_{\bm{\alpha}}\GG_m$-action, by construction, and the $\GG_m$-action $*_{\mathbf{0},\bm{\beta}}$ commutes with the $\GL_{\bfm,\bfr}\rtimes_{\bm{\alpha}}\GG_m$-action $*_{\bm{\alpha},\mathbf{0}}$.

For (ii), we only need to characterise when the $\GG_m$-actions $*_{\bm{\alpha}}$ and $*_{\bm{\alpha},\bm{\beta}}$ induce non-positive gradings on $\cO(\GL_{\bfm,\bfr})$ and $\cO(R(Q,\bfm;\bfr))$ respectively by \cref{Rmk/extGradedAction}. Since $\bm{\alpha}\in\bbZ_{>0}^{Q_0}$, the action $*_{\bm{\alpha}}$ induces a non-positive grading on $\cO(\GL_{\bfm,\bfr})$. For computing the weights of the $\GG_m$-action $*_{\bm{\alpha},\bm{\beta}}$ on $R(Q,\bfm;\bfr)$, we view $x\in R(Q,\bfm;\bfr)$ as a collection of matrices (indexed by $a \colon i\rightarrow j$ in $Q_1$)
\[
x_a\in\Hom_{k_{m_{ij}}}(k_{m_{ij}}^{\oplus r_if_{ji}},k_{m_{ij}}^{\oplus r_jf_{ij}})
\]
as in \eqref{Eqn/blockDecompositionRep}. Then the $\GG_m$-action $*_{\bm{\alpha},\bm{\beta}}$ scales the blocks of size $r_j\times r_i$ with fixed weight, and the block of $x_a$ with minimal weight is the bottom-left block of $(x_a)_0$, whose weight is $\beta_a-\alpha_i(f_{ji}-1)$, which completes (ii).

For (iii), by the above discussion on $\GG_m$-weights of $R(Q,\bfm;\bfr)$, we have:
\[
R(Q,\bfm;\bfr)^{\GG_m}=\prod_{\substack{a \colon i\rightarrow j \\ \beta_a=\alpha_i(f_{ji}-1)}}\Hom_{k}(k^{\oplus r_i},k^{\oplus r_j}).
\]
For the final statement in (iii), we analyse the limit maps under these $\GG_m$-actions. For $g=(g_i)_{i\in Q_0}\in\GL_{\bfm,\bfr}$, viewing $g_i$ as in \eqref{Eqn/blockDecompositionGroup} for every $i\in Q_0$, we have:
\[
\limit t*_{\bm{\alpha}} g=((g_i)_0)_{i\in Q_0}\in\GL_\bfr.
\]
For $x\in R(Q,\bfm;\bfr)$, viewing $x_a$ as in \eqref{Eqn/blockDecompositionRep} for every $a \colon i \rightarrow j$ in $Q_1$, we obtain:
\[
\limit t*_{\bm{\alpha},\bm{\beta}} x=(\tau(x_a))_{a\in Q_1}\in R(Q,\bfr),
\]
since the bottom-left $r_j\times r_i$ block of $x_a$ is the block of $(x_a)_0$ with coordinates $(f_{ij},1)$, which is $\tau(x_a)$ as in \cref{Def/matrixTruncation}. This finishes the proof.
\end{proof}

\begin{remark}
\cref{Lem/extGrad2} shows there are many possible external gradings of the action $\GL_{\bfm,\bfr}\curvearrowright R(Q,\bfm;\bfr)$. One of these gradings yields a more natural limit map $\tau$ (see \cref{Rmk/motivationTruncation}). Nevertheless, it can be useful to adjust the weight vector $\bm{\beta}$ for other purposes - for instance to make quiver moment maps equivariant as below.
\end{remark}

For the remainder of this subsection, we consider the doubled quiver $\overline{Q}$. The following results are key to the definition of the modified Nakajima quiver varieties appearing in \cref{main thm purity}.

\begin{lemma}\label{Lem/extGrad3} The above $\GG_m$-actions interact with the trace pairing (see \cref{Not/TracePairing}) and moment map as follows.
\begin{enumerate}[label=(\roman*)]
\item Fix $m_1,r_1,m_2,r_2$ and  $\alpha_1,\alpha_2,\beta_a,\beta_{a^*}\in\bbZ_{>0}$ satisfying $\alpha_1 m_1 = \alpha_2 m_2$. For $(x,y)\in\Hom_{k_{m_{12}}}(k_{m_1}^{\oplus r_1},k_{m_2}^{\oplus r_2})\times\Hom_{k_{m_{12}}}(k_{m_2}^{\oplus r_2},k_{m_1}^{\oplus r_1})$, we have
\[
\langle t*_{\alpha_1,\alpha_2,\beta_a}x,t*_{\alpha_2,\alpha_1,\beta_{a^*}}y \rangle
=
t^{\alpha_1f_{21}(m_{12}-1)+\beta_{a}+\beta_{a^*}} \langle x,y \rangle.
\]

\item Fix $\bm{\alpha}\in\bbZ_{>0}^{Q_0}$ satisfying \eqref{eqns for vertex weights in ext grading} and consider $\bm{\beta} = (\beta_a,\beta_{a^*})_{a \in Q_1}\in\bbZ_{>0}^{\overline{Q}_1}$. For all $(x,y)\in R(\overline{Q},\bfm;\bfr)$, we have:
\end{enumerate}\[ 
\mu_{(Q,\bfm),\bfr}(t*_{\bm{\alpha},\bm{\beta}}(x,y))=
t*_{\bm\alpha}
\left(
\sum_{\substack{a\in Q_1 \\ a \colon j\rightarrow i}}
t^{w_a}
\sum_{f=0}^{f_{ji}-1}
\epsilon^{f}x_ay_a\epsilon^{f_{ji}-1-f}
-
\sum_{\substack{a\in Q_1 \\ a \colon i\rightarrow j}}
t^{w_{a^*}}
\sum_{f=0}^{f_{ji}-1}
\epsilon^{f}y_ax_a\epsilon^{f_{ji}-1-f}
\right)_{\hspace{-6pt}i\in Q_0}
\] 
\hspace{1cm} where, for $a \colon i\rightarrow j$ in $ Q_1$, we have:
\begin{align*}
w_a & = \beta_a+\beta_{a^*}-\alpha_j(f_{ij}-1), \\
w_{a^*} & = \beta_a+\beta_{a^*}-\alpha_i(f_{ji}-1).
\end{align*}
\end{lemma}

\begin{proof}
For (i) we view the morphisms $x$ and $y$ as in \eqref{Eqn/blockDecompositionRep}. Under the isomorphisms:
\begin{align*}
\Hom_{k_{m_{12}}}(k_{m_1}^{\oplus r_1},k_{m_2}^{\oplus r_2})
\simeq
\Hom_{k_{m_{12}}}(k_{m_{12}}^{\oplus r_1f_{21}},k_{m_{12}}^{\oplus r_2f_{12}}),
\\
\Hom_{k_{m_{12}}}(k_{m_2}^{\oplus r_2},k_{m_1}^{\oplus r_1})
\simeq
\Hom_{k_{m_{12}}}(k_{m_{12}}^{\oplus r_2f_{12}},k_{m_{12}}^{\oplus r_1f_{21}}),
\end{align*}
we find that $x$ (resp.\ $y$) corresponds to the following $r_2f_{12}\times r_1f_{21}$ (resp.\ $r_1f_{21}\times r_2f_{12}$) matrix with coefficients in $k_{m_{12}}$:
\[
x_0+\epsilon x_1+\ldots+\epsilon^{m_{12}-1}x_{m_{12}-1}
\text{ (resp. }
y_0+\epsilon y_1+\ldots+\epsilon^{m_{12}-1}y_{m_{12}-1}
\text{).}
\]
We thus obtain:
\[
\langle x,y\rangle=\Tr(x_0y_{m_{12}-1})+\ldots+\Tr(x_{m_{12}-1}y_0).
\]
Now fix $0\leq m\leq m_{12}-1$ and $1\leq i\leq f_{12}$, $1\leq j\leq f_{21}$. In computing $\Tr(x_my_{m_{12}-1-m})$, the $r_2\times r_1$ block of $x_m$ with coordinates $(i,j)$ is paired with the $r_1\times r_2$ block of $y_{m_{12}-1-m}$ with coordinates $(j,i)$. The former has weight $\alpha_2(mf_{12}+f_{12}-i)-\alpha_1(f_{21}-j)+\beta_a$ under the $\GG_m$-action, whereas the latter has weight $\alpha_1((m_{12}-1-m)f_{21}+f_{21}-j)-\alpha_2(f_{12}-i)+\beta_{a^*}$. Summing up, we obtain that $\Tr(x_my_{m_{12}-1-m})$ is scaled with weight $\alpha_1f_{21}(m_{12}-1)+\beta_a+\beta_{a^*}$ under the $\GG_m$-action, which proves the claim.

For (ii), it suffices to prove the claim for a quiver with multiplicities $(Q,\bfm)$ having only one arrow $a \colon i\rightarrow j$ and for this we use (i). For $\xi\in\gl_{\bfm,\bfr}$, by the defining property of the moment map $\mu:=\mu_{(Q,\bfm),\bfr}$, we have
\[
\begin{split}
\langle\xi,\mu(t*_{\bm{\alpha},\bm{\beta}}(x,y))\rangle &
=
\langle\xi\cdot(t*_{\bm{\alpha},\bm{\beta}}x),t*_{\bm{\alpha},\bm{\beta}}y\rangle \\
& =
\langle t*_{\bm{\alpha},\bm{\beta}}((t^{-1}*_{\bm{\alpha}}\xi)\cdot x),t*_{\bm{\alpha},\bm{\beta}}y\rangle \\
& =
t^{w_{ij}}\cdot\langle(t^{-1}*_{\bm{\alpha}}\xi)\cdot x,y\rangle \\
& =
t^{w_{ij}}\cdot\langle t^{-1}*_{\bm{\alpha}}\xi,\mu(x,y)\rangle \\
& =
t^{w_{ij}}\cdot\left(\langle t^{-1}*_{\bm{\alpha}}\xi_j,\mu(x,y)_j\rangle + \langle t^{-1}*_{\bm{\alpha}}\xi_i,\mu(x,y)_i\rangle\right) \\
& =
t^{w_{ij}-\alpha_j(m_j-1)}\cdot\langle\xi_j,t*_{\bm{\alpha}}\mu(x,y)_j\rangle+t^{w_{ij}-\alpha_i(m_i-1)}\cdot\langle\xi_i,t*_{\bm{\alpha}}\mu(x,y)_i\rangle \\
& =
\left\langle (\xi_i,\xi_j),t*_{\bm{\alpha}}(t^{w_{a^*}}\cdot \mu(x,y)_i,t^{w_a}\cdot \mu(x,y)_j)\right\rangle,
\end{split}
\]
where $w_{ij}:=\alpha_if_{ji}(m_{ij}-1)+\beta_a+\beta_{a^*}$ and we used the first part (i) in the third and sixth equalities. Since this is true for all $\xi\in\gl_{\bfm,\bfr}$, we obtain the desired formula.
\end{proof}

\begin{corollary}\label{Cor/revisedExtGrad}
Set $\beta_a=\alpha_i(f_{ji}-1)$ for every $a \colon i\rightarrow j$ in $Q_1$. Fix $C>0$ and consider the $\GG_m$-action on $\gl_{\bfm,\bfr}$ given by:
\[
t*_{\bm{\alpha},\bm{\alpha}+C}\xi:=(t^{\alpha_i+C}\cdot (t*_{\alpha_i}\xi_i))_{i\in Q_0}.
\]
Then there exist $(\beta_{a^*})_{a\in Q_1}\in\bbZ_{>0}^{Q_1}$ such that the moment map $\mu_{(Q,\bfm),\bfr}$ is equivariant with respect to the $\GG_m$-actions $*_{\bm{\alpha},\bm{\beta}}$ on $R(\overline{Q},\bfm;\bfr)$ and $*_{\bm{\alpha},\bm{\alpha}+C}$ on $\gl_{\bfm,\bfr}$.
\end{corollary}

\begin{proof}
The moment map $\mu_{(Q,\bfm),\bfr}$ is $\GG_m$-equivariant if and only if for all $i\in Q_0$, we have $w_a=\alpha_i+C$ for all arrows $a\in Q_1$ pointing towards $i$ and $w_{a^*}=\alpha_i+C$ for all arrows $a\in Q_1$ starting from $i$. In other words, we require that, for all $a\colon i \rightarrow j$ in $Q_1$, we have $w_a=\alpha_j+C$ and $w_{a^*}=\alpha_i+C$. By definition of $w_a,w_{a^*}$, this amounts to $\beta_{a^*}=\alpha_i+C$.
\end{proof}

\subsection{Unipotent stabilisers}\label{Sect/unipotentStab}

 The aim of this section is to determine sufficient conditions to guarantee the unipotent stabiliser assumption required in \cref{Thm/extGradQuot} for the truncation map $\Rep(Q,\bfm;\bfr) \rightarrow \Rep(Q,\bfr)$ and the good $\GL_{\bfr}$-quotient $q_\theta \colon \Rep(Q,\bfr)^{\theta-\sst} \rightarrow M_{Q,\bfr}^{\theta-ss}$ taken with respect to a fixed stability parameter $\theta\in\bbZ^{Q_0}$. Let $U=U_{\bfm,\bfr}$ be the unipotent radical of $\GL_{\bfm,\bfr}$. Since the $\GL_{\bfm,\bfr}$-action has generic stabiliser $\Delta_{\bfm}$, we want to determine conditions for when $U$-stabiliser subgroups are contained in $\Delta_{\bfm}$ (strictly speaking, we should work with $G_{\bfm,\bfr}:=\GL_{\bfm,\bfr}/\Delta_{\bfm}$). The aim of this section is to prove the following result.

\begin{proposition}\label{prop suff cond for trivial unip stabilisers}
Suppose that one of the following conditions holds:
\begin{enumerate}[label=\roman*)]
\item the rank vector $\bfr$ is indivisible and $\theta$ is generic with respect to $\bfr$;
\item $R(Q,\bfr)^{\theta-\sst} = R(Q,\bfr)^{\theta-\st}$;
\item any neighbouring vertices of $(Q,\bfm)$ have coprime multiplicities and for all $x\in R(Q,\bfr)^{\theta-\sst}$, we have that, at each vertex of $Q$, ingoing maps are jointly surjective or outgoing maps are jointly injective.
\end{enumerate}
Then Assumption \ref{U} holds (for the action of $G_{\bfm,\bfr}$); that is, for every $x\in R(Q,\bfr)^{\theta-\sst}$, we have:
\[\Stab_U(\iota(x))\subset \Delta_\bfm.\]
\end{proposition}

Note that (i) implies (ii). We will prove the proposition under assumption (ii) in \cref{cor: ss=s implies trivial unip stabilisers}, and under assumption (iii) in \cref{lem ustab condition when adjacent vertices are coprime}.

\begin{notation}
We denote by $L\subset\bbQ^{Q_0}$ the line cut out by the equations:
\[
\forall i,j\in Q_0,\ m_jl_i=m_il_j,
\]
and by $I$ be the open interval of $L$ joining the origin and $\bfm$.

Given $x\in R(Q,\bfr)$ and $\bfq\in\bbQ_{\geq0}^{Q_0}$, we further define:
\[
\End(x)_\bfq:=
\left\{
\xi\in\End(x)\ \vert\ \xi_i\ne0\Rightarrow q_i\in\{1,\ldots, m_i-1\}
\right\}
,
\]
where $\End(x)\subset\gl_\bfr$ is the Lie algebra $\GL_{\bfr}$-stabiliser of $x\in R(Q,\bfr)$.
\end{notation}

\begin{remark}
First note that, for any $\bfq\in I$, we have $q_i\ne0$ for all $i\in Q_0$. Second, to see that the above equations indeed cut out a line, note that $(l_i)_{i\in Q_0} \in L$ is determined by any single entry and that the coefficients $m_i$ satisfy the following consistency condition:
\[
\forall i,j,k\in Q_0,\ \frac{m_i}{m_k}\frac{m_k}{m_j}\frac{m_j}{m_i}=1,
\]
which guarantees the existence of a solution in $\bbQ^{Q_0}$.
\end{remark}

By Lemma \ref{Lem/unipotentRad}, the unipotent radical $U$ can be described as:
\[
U=
\left\{
\left(
\left.
u_i=\Id_{r_i}+\sum_{m=1}^{m_i-1}u_{i,m}\cdot\epsilon^m
\right)_{i\in Q_0}
\ \right\vert\ 
\begin{array}{l} u_{i,m}\in\gl_{r_i} \:
\forall i\in Q_0\\
\forall m\in\{1,\ldots,m_i-1\}
\end{array}
\right\}
.
\]

The key to the proof of \cref{prop suff cond for trivial unip stabilisers} is the following explicit description of unipotent stabilisers for truncated representations of $(Q,\bfm)$.

\begin{proposition}
Let $x\in R(Q,\bfr)$. Then we have:
\[
\Stab_U(\iota(x))=
\left\{
\left.
\left(\Id_{r_i}+\sum_{\bfq\in I}(\xi_{\bfq})_i\cdot\epsilon^{q_i}\right)_{i\in Q_0}
\ \right\vert \ 
\forall\bfq\in I,\ \xi_{\bfq}\in\End(x)_{\bfq}
\right\}
,
\]
where, by construction, the coefficient of $\epsilon^{q_i}$ is zero whenever $q_i\notin\{0,\ldots,m_i-1\}$, so that the above sums are well-defined and finite.
\end{proposition}

\begin{proof}
Let $x\in R(Q,\bfr)$ and $u\in U$. Let us solve the equations:
\[
u_{t(a)}\iota(x_a)=\iota(x_a)u_{s(a)},\ a\in Q_1.
\]
For a fixed arrow $a \colon i\rightarrow j$, we work with block matrices for the following decompositions:
\[
k_{m_i}^{\oplus r_i}=(\epsilon^{m_i-1}\cdot k^{\oplus r_i}\oplus\ldots\oplus \epsilon^{(m_{ij}-1)f_{ji}}\cdot k^{\oplus r_i})\oplus\ldots\oplus (\epsilon^{f_{ji}-1}\cdot k^{\oplus r_i}\oplus\ldots\oplus k^{\oplus r_i}),
\]
\[
k_{m_j}^{\oplus r_j}=(\epsilon^{m_j-1}\cdot k^{\oplus r_j}\oplus\ldots\oplus \epsilon^{(m_{ij}-1)f_{ij}}\cdot k^{\oplus r_j})\oplus\ldots\oplus (\epsilon^{f_{ij}-1}\cdot k^{\oplus r_j}\oplus\ldots\oplus k^{\oplus r_j}).
\]
From the discussion in Section \ref{Sect/Truncation}, we have that the matrix of $u_i$ (resp. $u_j$) has the following decomposition as a $m_{ij}\times m_{ij}$ block matrix, with blocks of size $r_if_{ji}$ (resp. $r_jf_{ij}$):
\[
\left(
\begin{array}{cccc}
U_{i,0} & U_{i,1} & \ldots & U_{i,m_{ij}-1} \\
0 & U_{i,0} & \ddots & \vdots \\
\vdots & \ddots & \ddots & U_{i,1} \\
0 & \ldots & 0 & U_{i,0}
\end{array}
\right)
\text{ (resp. }
\left(
\begin{array}{cccc}
U_{j,0} & U_{j,1} & \ldots & U_{j,m_{ij}-1} \\
0 & U_{j,0} & \ddots & \vdots \\
\vdots & \ddots & \ddots & U_{j,1} \\
0 & \ldots & 0 & U_{j,0}
\end{array}
\right)
),
\]
where, for $0\leq k\leq m_{ij-1}$, the matrix $U_{i,k}$ (resp. $U_{j,k}$) is the following block-matrix of size $f_{ji}\times f_{ji}$ (resp. $f_{ij}\times f_{ij}$), with blocks of size $r_i\times r_i$ (resp. $r_j\times r_j$):
\[
U_{i,k}
=
\left(
\begin{array}{ccccc}
u_{i,kf_{ji}} & u_{i,kf_{ji}+1} & \ldots & u_{i,(k+1)f_{ji}-2} & u_{i,(k+1)f_{ji}-1} \\
u_{i,kf_{ji}-1} & u_{i,kf_{ji}} & \ldots & u_{i,(k+1)f_{ji}-3} & u_{i,(k+1)f_{ji}-2} \\
\vdots & \vdots & \ddots & \vdots & \vdots \\
u_{i,(k-1)f_{ji}+2} & u_{i,(k-1)f_{ji}+3} & \ldots & u_{i,kf_{ji}} & u_{i,kf_{ji}+1} \\
u_{i,(k-1)f_{ji}+1} & u_{i,(k-1)f_{ji}+2} & \ldots & u_{i,kf_{ji}-1} & u_{i,kf_{ji}}
\end{array}
\right)
\]
\[
\text{(resp. }
U_{j,k}
=
\left(
\begin{array}{ccccc}
u_{j,kf_{ij}} & u_{j,kf_{ij}+1} & \ldots & u_{j,(k+1)f_{ij}-2} & u_{j,(k+1)f_{ij}-1} \\
u_{j,kf_{ij}-1} & u_{j,kf_{ij}} & \ldots & u_{j,(k+1)f_{ij}-3} & u_{j,(k+1)f_{ij}-2} \\
\vdots & \vdots & \ddots & \vdots & \vdots \\
u_{j,(k-1)f_{ij}+2} & u_{j,(k-1)f_{ij}+3} & \ldots & u_{j,kf_{ij}} & u_{j,kf_{ij}+1} \\
u_{j,(k-1)f_{ij}+1} & u_{j,(k-1)f_{ij}+2} & \ldots & u_{j,kf_{ij}-1} & u_{j,kf_{ij}}
\end{array}
\right)
).
\]
By convention, we set $u_{i,k}=0$ when $k<0$. 

Likewise, in the bases above, the matrix of $x_a$ is block-diagonal of size $m_{ij}$, with blocks of size $r_jf_{ij}\times r_if_{ji}$:
\[
\left(
\begin{array}{cccc}
X_a & 0 & \ldots & 0 \\
0 & X_a & \ddots & \vdots \\
\vdots & \ddots & \ddots & 0 \\
0 & \ldots & 0 & X_a
\end{array}
\right)
,
\]
where $X_a$ is a block-rectangular matrix of size $f_{ij}\times f_{ji}$, with blocks of size $r_j\times r_i$:
\[
\left(
\begin{array}{cccc}
0 & 0 & \ldots & 0 \\
0 & 0 & \ldots & 0 \\
\vdots & \vdots & \ddots & \vdots \\
x_a & 0 & \ldots & 0
\end{array}
\right)
.
\]
In this setup, the equation:
\[
u_{t(a)}\iota(x_a)=\iota(x_a)u_{s(a)},\ a\in Q_1.
\]
is equivalent to $U_{j,k}X_a=X_aU_{i,k}$, for $0\leq k\leq m_{ij}-1$. In other words:
\[
\left(
\begin{array}{cccc}
u_{j,(k+1)f_{ij}-1}x_a & 0 & \ldots & 0 \\
u_{j,(k+1)f_{ij}-2}x_a & 0 & \ldots & 0 \\
\vdots & \vdots & \ddots & \vdots \\
u_{j,kf_{ij}}x_a & 0 & \ldots & 0
\end{array}
\right)
=
\left(
\begin{array}{cccc}
0 & 0 & \ldots & 0 \\
\vdots & \vdots & \ddots & \vdots \\
0 & 0 & \ldots & 0 \\
x_au_{i,kf_{ji}} & x_au_{i,kf_{ji}+1} & \ldots & x_au_{i,(k+1)f_{ji}-1}
\end{array}
\right)
,
\]
which amounts to: for $q_i\in\{1,\ldots,m_i-1\}$ and $q_j\in\{1,\ldots,m_j-1\}$,
\begin{enumerate}
\item $u_{j,q_j}x_a=x_au_{i,q_i}$ whenever $q_i$ and $q_j$ satisfy $q_if_{ij}=q_jf_{ji}$ (equivalently, $q_im_j=q_jm_i$);
\item $u_{j,q_j}x_a=0$ (resp. $x_au_{i,q_i}=0$) whenever $q_j$ (resp. $q_i$) is not divisible by $f_{ij}$ (resp. by $f_{ji}$); or equivalently when $\frac{f_{ji}}{f_{ij}}q_j$ (resp. $\frac{f_{ij}}{f_{ji}}q_i$) is not an integer.
\end{enumerate}
Note that the second characterisation in 2. holds because $f_{ji}$ and $f_{ij}$ are coprime. We reformulate the equations above as $u_{j,q_j}x_a=x_au_{i,q_i}$, where $\mathbf{q}\in\bbQ_{>0}^{Q_0}$ lies on the interval $I$ and, by convention, $u_{i,q_i}=0$ when $q_i\not\in\{1,\ldots, m_i-1\}$. This concludes the proof.
\end{proof}

Let us write $\bfm=\delta\bfm'$, where $\delta\in\bbZ_{>0}$ and $\bfm'$ is indivisible. We are particularly interested in situations where unipotent stabilisers are contained in the diagonal subgroup
\[
\Delta_\bfm=k_{\delta}^{\times}\cdot\Id\subset\GL_{\bfm,\bfr},
\]
which acts trivially on $R(Q,\bfm;\bfr)$.

\begin{corollary}\label{Cor/U-stabilisers when multiplicities are divisible}
Suppose that $\delta>1$, i.e. $\bfm$ is divisible. Then:
\[
\Stab_U(\iota(x))\subset\Delta_\bfm\Rightarrow\End(x)=k\cdot\Id.
\]
\end{corollary}

In particular, if $\bfm$ is divisible and $\theta\in\bbZ^{Q_0}$, then unipotent stabilisers are trivial on $\iota(R(Q,\bfr)^{\theta-ss})$ only if $R(Q,\bfr)^{\theta-\sst}=R(Q,\bfr)^{\theta-\st}$. We can also deduce a converse statement.

\begin{corollary}\label{cor: ss=s implies trivial unip stabilisers}
Consider $\theta\in\bbZ^{Q_0}$ such that $R(Q,\bfr)^{\theta-\sst}=R(Q,\bfr)^{\theta-\st}$. Then, for every $x\in R(Q,\bfr)^{\theta-\sst}$, we have:
\[
\Stab_U(\iota(x))\subset \Delta_\bfm.
\]
\end{corollary}

\begin{proof}
Set $x\in R(Q,\bfm;\bfr)^{\theta-\sst}$. By assumption on $\theta$, we have $\End(x)=k\cdot\Id$, where $k$ is embedded diagonally into $\gl_\bfr$ as scalar matrices. Thus, for $\bfq\in I$, we have $\End(x)_\bfq\ne\{0\}$ only if all entries of $\bfq$ are integral, i.e. $\bfq$ is a multiple of $\bfm'$. Therefore the only possible non-zero entries of $u\in\Stab_U(\iota(x))$ are $u_{i,dm'_i},\ i\in Q_0,\ 0< d<\delta$. In other words, $\Stab_U(\iota(x))\subset \Delta_\bfm$, which is what we wanted to show.
\end{proof}

When $\bfm$ is indivisible, unipotent stabilisers may be trivial on $R(Q,\bfr)^{\theta-\sst}$ even in the presence of strictly semistable representations (see \cref{Sect/examplesTwistedModuli}). We give some sufficient conditions below.

\begin{notation}
We will denote by $\End(x)_i$ those endomorphisms supported only at the vertex $i \in Q_0$.
\end{notation}

We can easily obtain the following characterisation.

\begin{lemma} \label{lemma characterising End_i}
    Fix $i \in Q_0$. Then we have \[ \End(x)_i = \left\{ \phi \in \End(V_i) \;\middle|\; \begin{aligned}  &\mathrm{Im}\: \phi \subseteq \Ker x_a  \text{ for all } a \in s^{-1}(i) \\ &\mathrm{Im}\: x_a \subseteq \Ker \phi \text{ for all }a \in t^{-1}(i) \end{aligned} \right\}.\] In particular, we have $\End(x)_i = 0$ if and only if one of the following holds \begin{enumerate}[label=\roman*)]
        \item we have $\underset{a \colon i \rightarrow j}{\bigcap}\Ker x_a = 0$, in which case we say that $x$ is \emph{jointly injective} at $i$;
        \item we have $\sum\limits_{a \colon j \rightarrow i}\mathrm{Im} x_a = V_i$, in which case we say that $x$ is \emph{jointly surjective} at $i$.
    \end{enumerate} 
\end{lemma}

\begin{proof}
    The first statement is direct from the definition. For the second, note that to have a nonzero element of $\End(x)_i$ we must choose a nonzero map $V_i/\sum_{t(a)=i}
     \text{Im}x_a \rightarrow \bigcap_{s(a)=i}\Ker x_a$, and any such map will do.
\end{proof}

\begin{lemma}\label{Lem/isolatedVertices} 
    Fix $x \in R(Q,\bfr)$. Suppose that $\Stab_U(\iota(x)) = \{e\}$. Consider a vertex $i \in Q_0$ such that, in the undirected graph underlying $Q$, $i$ shares edges only with vertices of multiplicity not divisible by $m_i$. Then the representation $V(x)$ must be either jointly injective or jointly surjective at $i$. 
\end{lemma}

\begin{proof}
Let $i \in Q_0$ be a vertex sharing edges only with vertices of multiplicity not divisible by $m_i$. Consider the unique element $\bfq \in I$ whose $i$th entry is $1$. By assumption on $i$, none of its neighbouring vertices have integer entry in $\bfq$, so that $\End(x)_\bfq \supset \End(x)_i$. In particular, to have $\Stab_U(\iota(x)) = \{e\}$ we must have $\End(x)_i=0$, which gives us the statement by Lemma \ref{lemma characterising End_i}. 
\end{proof}

We obtain a partial converse to this statement, under an assumption on multiplicities.

\begin{lemma} \label{lem ustab condition when adjacent vertices are coprime}
    Suppose that any two vertices which are connected by an arrow have coprime multiplicities. Then $\Stab_U(\iota(x)) = \{e\}$ if and only if for every vertex $i$ the representation $x$ is either jointly surjective or jointly injective at $i$.
\end{lemma}

\begin{proof}
    Note that the assumption on the multiplicity ensures that all vertices satisfy the assumption of \cref{Lem/isolatedVertices}, which gives one direction. Moreover, any $\bfq \in I$ has the property that for any arrow $a \colon  i \rightarrow j$, the entries of $\bfq$ at $i$ and $j$ cannot both be integers. For such $\bfq$ we have $\End(x)_\bfq = \bigoplus_{i : q_i\in \ZZ} \End(x)_i$, and hence we will have $\Stab_U(\iota(x)) = \{e\}$ as long as each $\End(x)_i$ vanishes.
\end{proof}

\subsection{Hilbert-Mumford analysis}\label{Sect/HManalysis}

Let $\theta,\rho\in\bbZ^{Q_0}$ be two stability parameters. In this section, we describe the GIT semistable locus for the $\rho$-twisted quotient of quiver representatives with multiplicities relative to the classical moduli space of $\theta$-(semi)stable quiver representations.

\begin{assumption*}
Until the end of \cref{sec/constr moduli spaces}, we assume that the weights of the external grading action satisfy $\beta_a=\alpha_i(f_{ji}-1)$ for any arrow $a \colon i\rightarrow j$ (see \cref{Sect/extGrad}).
\end{assumption*}

\subsubsection{Semistable locus for the Levi action}

By \cref{rmk/alt descr HM}, the first step is to analyse the semistable locus for the Levi subgroup $\GL_\bfr\times\GG_m\subset\GL_{\bfm,\bfr}\rtimes\GG_m$. We build a thickened quiver $Q_\bfm$, such that there is an isomorphism:
\[
R(Q,\bfm;\bfr)\simeq R(Q_\bfm,\bfr),
\]
which is equivariant with respect to the action of the Levi subgroup:
\[
\tilde{\GL}_\bfr:=\GL_\bfr\times\GG_m
\]

\begin{definition}
Let $(Q,\bfm)$ be a quiver with multiplicities. The thickened quiver $Q_\bfm$ is defined as follows:
\begin{itemize}
\item the set of vertices is $(Q_\bfm)_0:=Q_0$;
\item the set of arrows is
\[
(Q_\bfm)_1=\left\{ (a \colon i\rightarrow j,m,f_1,f_2)\in Q_1\times\bbZ_{\geq0}^3\ \left\vert\ 
\begin{array}{l}
0\leq m\leq m_{ij}-1 \\
0\leq f_1\leq f_{ji}-1 \\
0\leq f_2\leq f_{ij}-1
\end{array}
\right.\right\}
,
\]
where $s(a,m,f_1,f_2)=s(a)$ and $t(a,m,f_1,f_2)=t(a)$.
\end{itemize}
\end{definition}

We define a $\tilde{\GL}_\bfr=\GL_\bfr\times\GG_m$-action on $R(Q_\bfm,\bfr)$ as follows. The $\GL_\bfr$-action is the usual action on the representation space of a quiver, while $\GG_m$ acts on $x\in R(Q_\bfm,\bfr)$ by scaling $x_{(a,m,f_1,f_2)}$ with weight
\begin{equation}\label{arrow weights in thickened quiver}
   \mathrm{wt}(a,m,f_1,f_2):=\alpha_j(mf_{ij}+f_2)+\alpha_i(f_{ji}-1-f_1)\geq0. 
\end{equation}
Then we have:

\begin{lemma}\label{Lem/thickQuiver}
There is a $\tilde{\GL}_\bfr$-equivariant isomorphism:
\begin{equation}\label{Eqn/thickQuiver}
R(Q,\bfm;\bfr)\simeq R(Q_\bfm,\bfr).
\end{equation}
Moreover, the composition $R(Q_\bfm,\bfr)\simeq R(Q,\bfm;\bfr)\overset{\tau}{\rightarrow} R(Q,\bfr)$ 
is $\GL_\bfr$-equivariant and $\GG_m$-invariant, and is given by the assignment:
\[
(x_{(a,m,f_1,f_2)})_{(a,m,f_1,f_2)\in (Q_\bfm)_1}
\mapsto
(x_{a,0,f_{t(a)s(a)}-1,0})_{a\in Q_1}.
\]
In particular, the image of the section $R(Q,\bfr) \stackrel{\iota}{\rightarrow}  R(Q,\bfm;\bfr) \simeq R(Q_\bfm,\bfr)$ is fixed by $\GG_m$.
\end{lemma}

\begin{remark}
We will often abuse notation and also call $\tau$ the composition:
\[
R(Q_\bfm,\bfr)\simeq R(Q,\bfm;\bfr)\overset{\tau}{\rightarrow} R(Q,\bfr).
\]
\end{remark}

\begin{proof}
Let us first specify the isomorphism $R(Q_\bfm,\bfr)\simeq R(Q,\bfm;\bfr)$. Given an arrow of $Q$, called $a \colon i\rightarrow j$, we have the following $\GL_{r_i}\times\GL_{r_j}$-equivariant decomposition:
\[
\begin{split}
\Hom_{k_{m_{ij}}}(k_{m_i}^{\oplus r_i},k_{m_j}^{\oplus r_j}) &
\simeq
\Hom_{k_{m_{ij}}}(k^{\oplus r_i}\otimes_kk_{m_i},k^{\oplus r_j}\otimes_kk_{m_j}) \\
& \simeq
\Hom_k(k^{\oplus r_i},k^{\oplus r_j})\otimes_k\Hom_{k_{m_{ij}}}(k_{m_i},k_{m_j}) \\
& \simeq
\Hom_k(k^{\oplus r_i},k^{\oplus r_j})\otimes_k\Hom_{k_{m_{ij}}}(k_{m_{ij}}^{\oplus f_{ji}},k_{m_{ij}}^{\oplus f_{ij}}) \\
& \simeq
\bigoplus_{f_1=0}^{f_{ji}-1}\bigoplus_{f_2=0}^{f_{ij}-1}\Hom_k(k^{\oplus r_i},k^{\oplus r_j})\otimes_kk_{m_{ij}} \\
& \simeq
\bigoplus_{f_1=0}^{f_{ji}-1}\bigoplus_{f_2=0}^{f_{ij}-1}\bigoplus_{m=0}^{m_{ij}-1}\Hom_k(k^{\oplus r_i},k^{\oplus r_j}),
\end{split}
\]
where the direct sum decomposition of $k_{m_{ij}}$ (resp. $k_{m_i}$, $k_{m_j}$) is indexed using the following basis:
\[
k_{m_{ij}}=\bigoplus_{m=0}^{m_{ij}-1}k\epsilon^m
\ \left(\text{resp. }
k_{m_i}=\bigoplus_{f_1=0}^{f_{ji}-1}k_{m_{ij}}\epsilon^{f_1}
\text{ ; }
k_{m_j}=\bigoplus_{f_2=0}^{f_{ij}-1}k_{m_{ij}}\epsilon^{f_2}
\right) .
\]
By matching the right-hand side summand labeled by $(m,f_1,f_2)$ with the arrow $(a,m,f_1,f_2)$ of $Q_{\bfm}$, we obtain a $\GL_\bfr$-equivariant isomorphism $R(Q_\bfm,\bfr)\simeq R(Q,\bfm;\bfr)$.

We now check that this isomorphism is also $\GG_m$-equivariant. Fix an arrow of $Q$, called $a \colon i\rightarrow j$, and $x\in R(Q,\bfm;\bfr)$. Using \cref{Not/BlockMatrices}, the above isomorphism has the following explicit expression:
\[
\begin{array}{rcl}
R(Q,\bfm;\bfr) & \rightarrow & R(Q_\bfm,\bfr) \\
(x_a)_{a\in Q_1} & \mapsto & ([(x_a)_m]_{f_{ij}-f_2,f_{ji}-f_1})_{(a,m,f_1,f_2)\in (Q_\bfm)_1},
\end{array}
\]
where $[(x_a)_m]_{f',f''}$ denotes the $r_j\times r_i$-block of $(x_a)_m$ with coordinates $(f',f'')$. Moreover, under the standing $\GG_m$-action, the block $[(x_a)_m]_{f_{ij}-f_2,f_{ji}-f_1}$ is scaled with weight:
\[
\begin{split}
\alpha_j(mf_{ij}+f_2)-\alpha_if_1+\beta_a & = \alpha_j(mf_{ij}+f_2)+\alpha_i(f_{ji}-1-f_1) = \mathrm{wt}(a,m,f_1,f_2).
\end{split}
\]
This proves $\GG_m$-equivariance.

Finally, recall from \cref{Not/BlockMatrices} that $\tau(x)$ coincides with $[x_0]_{f_{ij},1}$, which in the above decomposition corresponds to $(m,f_1,f_2)=(0,f_{ji}-1,0)$. This yields the wanted description of $\tau$ under the isomorphism.
\end{proof}

\begin{example}
Suppose $\bfm=m\mathbf{1}$ for some integer $m\geq1$. Then arrows of $Q_\bfm$ are simply given by $(Q_\bfm)_1=Q_1\times\{0,\ldots, m-1\}$. The isomorphism $R(Q,\bfm;\bfr)\simeq R(Q_\bfm,\bfr)$ reads, in coordinates:
\[
(x_a)_{a\in Q_1}\mapsto (x_{a,n})_{
\begin{subarray}{l}
a\in Q_1 \\
0\leq n\leq m-1
\end{subarray}
},
\]
where $x_{a,n}$ is the coefficient of $\epsilon^n$ in $x_a\in\Hom_{k_m}(k_m^{\oplus r_i},k_m^{\oplus r_j})\simeq k_m\otimes_k\Hom_k(k^{\oplus r_i},k^{\oplus r_j})$.
\end{example}

Let us now analyse the semistable locus for the action of the Levi subgroup
\[
\GL_\bfr\times\GG_m\subset\GL_{\bfm,\bfr}\rtimes\GG_m
\]
on $R(Q,\bfm;\bfr)\times\bbA^1$. The action on $R(Q,\bfm;\bfr)$ is as described in Section \ref{Sect/extGrad}, while the action on $\bbA^1$ is given by $(g,t)\cdot z=tz$. Before we analyse the Hilbert-Mumford criterion, we set up some notation.

\begin{notation}
We will write one-parameter subgroups of $\tilde{\GL}_\bfr=\GL_\bfr\times\GG_m$ as $\tilde{\lambda}=(\lambda,l)$, where $
\tilde{\lambda}(t)=(\lambda(t),t^l)$.
Likewise, we will write characters of $\tilde{\GL}_\bfr$ as $\tilde{\rho}=(\rho,n)$, given by $\tilde{\rho}(g,t)=\rho(g)t^n$. 
Moreover, the datum of a one-parameter subgroup $\lambda$ of $\GL_\bfr$ is equivalent to the datum of a $\bbZ$-grading of the $Q_0$-graded vector space $k^{\oplus\bfr}$. We denote by $W^p\subset k^{\oplus\bfr}$ the $Q_0$-graded weight-subspace of weight $p$ under the action of $\lambda$ and:
\[
V^p:=\bigoplus_{q\geq p}W^q,
\]
so that $V^{\bullet}$ is a descending filtration of $k^{\oplus\bfr}$ by $Q_0$-graded vector spaces.
\end{notation}

We next determine when the flow of a point under a one-parameter subgroup admits a limit.

\begin{lemma}\label{Lem/1PSlimit}
Let $(x,z)\in R(Q_\bfm,\bfr)\times(\bbA^1\setminus\{0\})$ and $\tilde{\lambda}=(\lambda,l):\GG_m\rightarrow\GL_\bfr\times\GG_m$ be a one parameter subgroup. Then $\lim_{t\rightarrow0}\tilde{\lambda}(t)\cdot (x,z)$ exists if and only if:
\[
\left\{
\begin{array}{l}
l\geq0\ , \\
\forall (a,m,f_1,f_2)\in (Q_\bfm)_1,\ x_{a,m,f_1,f_2}(V_{s(a)}^\bullet)\subset V_{t(a)}^{\bullet-l\mathrm{wt}(a,m,f_1,f_2)}.
\end{array}
\right.
\]
In particular, if $l=0$ or $x$ lies in the image of $\iota \colon R(Q,\bfr) \rightarrow R(Q_\bfm,\bfr)$, the second condition is equivalent to $V^p$ being a subrepresentation of $V(x)$ for all $p\in\bbZ$. 
\end{lemma}

\begin{proof}
By the preceding discussion and Lemma \ref{Lem/thickQuiver}, for any $(a,m,f_1,f_2)\in (Q_\bfm)_1$, the component:
\[
x_{a,m,f_1,f_2}^{(p,q)} \colon W_{s(a)}^q\rightarrow W_{t(a)}^p
\]
is scaled with weight $p-q+l\mathrm{wt}(a,m,f_1,f_2)$
under the action of $\tilde{\lambda}$. So $\lim_{t\rightarrow0}\tilde{\lambda}(t)\cdot x$ exists if and only if:
\[
\forall (a,m,f_1,f_2)\in (Q_\bfm)_1,\ x_{a,m,f_1,f_2}(V_{s(a)}^\bullet)\subset V_{t(a)}^{\bullet-l\mathrm{wt}(a,m,f_1,f_2)}.
\]
Besides, $\lim_{t\rightarrow0} t^l\cdot z$ exists if and only if $l\geq0$. The final claim follows, as if $l = 0$, then there is no index shift in the second condition. Similarly if $x$ is in the image of $\iota$, then there is no index shift, as $\mathrm{wt}(a,m,f_1,f_2) =0$ whenever $x_{a,m,f_1,f_2}$ is non-zero (see \cref{Lem/thickQuiver}).
\end{proof}

We now proceed to the core of the Hilbert-Mumford analysis.

\begin{lemma}\label{Lem/HM1}
There exists $n_0\geq0$, depending only on $(Q,\bfm)$, $\bfr$, $\rho$ and $\bm{\alpha},\bm{\beta}$, such that for all $n\geq n_0$ and $(x,z)\in \tau^{-1}(R(Q,\bfr)^{\theta-\sst})\times(\bbA^1\setminus\{0\})$, the following are equivalent:
\begin{enumerate}[label=\roman*)]
\item $(x,z)$ is $\tilde{\rho}$-semistable (resp. stable) for the action of $\tilde{\GL}_\bfr$; 
\item for any subrepresentation $V\subset k^{\oplus\bfr}$ of $V(x)$ such that $\theta\cdot\dim V=0$, we have $\rho\cdot\dim V\geq0$ (resp. $\rho\cdot\dim V>0$).
\end{enumerate}
\end{lemma}

\begin{proof}[Proof of Lemma \ref{Lem/HM1}]
Let us define:
\begin{align*}
& A(\bfr)=\prod_{i\in Q_0}(r_i+1)\ , \\
& B=\underset{\substack{0\leq\bfr'\leq\bfr \\ \rho\cdot\bfr'<0}}{\max}\left\{\vert\rho\cdot\bfr'\vert\right\}\ , \\
& w=\underset{(a,m,f_1,f_2)\in (Q_\bfm)_1}{\max}\left\{\mathrm{wt}(a,m,f_1,f_2)\right\}\ .
\end{align*}
Fix $n> A(\bfr)Bw$ and let $(x,z)\in\tau^{-1}(R(Q,\bfr)^{\theta-ss})\times\left(\bbA^1\setminus\{0\}\right)$.

Suppose that for any subrepresentation $V\subset k^{\oplus\bfr}$ of $V(x)$ such that $\theta\cdot\dim V=0$, we have $\rho\cdot\dim V\geq0$ (resp. $\rho\cdot\dim V>0$). We want to check that, for any one-parameter subgroup $\tilde{\lambda} \colon \GG_m\rightarrow \tilde{\GL}_\bfr$ such that $\lim_{t\rightarrow0}\tilde{\lambda}(t)\cdot (x,z)$ exists and lies in $\tau^{-1}(R(Q,\bfr)^{\theta-\sst})\times\bbA^1$, we have $\langle\tilde{\lambda},\tilde{\rho}\rangle\geq0$ (resp. $\langle\tilde{\lambda},\tilde{\rho}\rangle>0$).

Let $\tilde{\lambda}$ be such a one-parameter subgroup and $V^\bullet$ the associated filtration of $k^{\oplus\bfr}$. As initially observed by King \cite[\S 3]{Kin94}, we have:
\[
\langle\tilde{\lambda},\tilde{\rho}\rangle
=
\sum_{p\in\bbZ}\rho\cdot\dim V^p+ln.
\]
Since, by assumption,
\[
\tau\left(\limit\lambda(t)\cdot x\right)=\limit\left(\tilde{\lambda}(t)\cdot\tau(x)\right)
\]
is $\theta$-semistable and, by Lemma \ref{Lem/1PSlimit}, the subspace $V^p\subset k^{\oplus\bfr}$ is a subrepresentation of $V(\tau(x))$, we have that $\theta\cdot\dim V^p=0$ for all $p\in\bbZ$, by Lemma \ref{Lem/ssGr} below.

Let $p_{\min}<p_{\max}$ denote the two integers satisfying:
\[
k^{\oplus\bfr}=V^{p_{\min}}\supsetneq V^{p_{\min}+1} \supset V^{p_{\max}}\supsetneq V^{p_{\max}+1}=\{ 0\},
\]
so that there are $p_{\max}-p_{\min}$ steps in the filtration where $V^p\subset k^{\oplus\bfr}$ is a non-zero proper subspace. Note that at most
\[
A(\bfr)=\prod_{i\in Q_0}(r_i+1)
\]
distinct ($Q_0$-graded) subspaces $V\subset k^{\oplus\bfr}$ occur among the $V^p$ with $ p_{\min}+1\leq p\leq p_{\max}$. Consider
\[
w=\underset{(a,m,f_1,f_2)\in (Q_\bfm)_1}{\max}\left\{\mathrm{wt}(a,m,f_1,f_2)\right\}.
\]
By the pigeonhole principle, there can be at most $lwA(\bfr)$ steps in the filtration which are taken by subspaces $V\subset k^{\oplus\bfr}$ occurring $lw$ or fewer times in the filtration. The remaining steps are then taken by subspaces $V\subset k^{\oplus\bfr}$ which are subrepresentations of $V(x)$, by Lemma \ref{Lem/1PSlimit}. Note that if $l=0$, then all subspaces $V^p\subset k^{\oplus\bfr}$ are subrepresentations of $V(x)$.

If the graded subspace $V^p \subset k^{\oplus\bfr}$ is a subrepresentation of $V(x)$, then by assumption we have $\rho\cdot\dim V^p \geq 0$ (resp. $\rho\cdot\dim V^p>0$). If not, then by definition of the constant $B$, we have $\rho\cdot\dim V^p \geq -B$ and there can be at most $lwA(\bfr)$ values of $p$ in the range $p_{\min} < p \leq p_{\max}$ for which $V^p$ is not a subrepresentation of $V(x)$. Thus by the choice of $n> A(\bfr)Bw$, we obtain the estimate:
\[
\begin{split}
\langle\tilde{\lambda},\tilde{\rho}\rangle
& =\sum_{p\in\bbZ}\rho\cdot\dim V^p+ln \\
& \geq l(n-A(\bfr)Bw) \\
& \geq0.
\end{split}
\]
(resp.\ we obtain the estimate:
\[
\begin{split}
\langle\tilde{\lambda},\tilde{\rho}\rangle
& =\sum_{p\in\bbZ}\rho\cdot\dim V^p+ln \\
& \geq l(n-A(\bfr)Bw) \\
& \geq0,
\end{split}
\]
where the first inequality is strict if $l=0$, whereas the second inequality is strict if $l>0$). This completes the proof of the implication (ii) $\Rightarrow$ (i).

Conversely, suppose that for any one-parameter subgroup $\tilde{\lambda} \colon \GG_m\rightarrow \tilde{\GL}_\bfr$ such that $\lim_{t\rightarrow0}\tilde{\lambda}(t)\cdot (x,z)$ exists and lies in $\tau^{-1}(R(Q,\bfr)^{\theta-\sst})\times\bbA^1$, we have $\langle\tilde{\lambda},\tilde{\rho}\rangle\geq0$ (resp. $\langle\tilde{\lambda},\tilde{\rho}\rangle>0$).

Let $V\subset k^{\oplus\bfr}$ be a subrepresentation of $V(x)$ such that $\theta\cdot\dim V=0$. Consider the one-parameter subgroup $\tilde{\lambda}=(\lambda,0)$, where $\lambda$ corresponds to the filtration:
\[
V^p=
\left\{
\begin{array}{ll}
k^{\oplus\bfr} & ,\ p<0 \\
V & ,\ p=0 \\
\{0\} & ,\ p>0.
\end{array}
\right.
\]
By Lemma \ref{Lem/1PSlimit}, $\lim_{t\rightarrow0}\tilde{\lambda}(t)\cdot(x,z)$ exists, so we obtain:
\[
\langle\tilde{\lambda},\tilde{\rho}\rangle=\rho\cdot\dim V\geq0
\text{ (resp. }\rho\cdot\dim V>0).
\]
This proves the lemma.
\end{proof}

The following lemma is well-known and already implicit in King's work \cite{Kin94}.

\begin{lemma}\label{Lem/ssGr}
Let $x\in R(Q,\bfr)^{\theta-\sst}$ and $\lambda \colon \GG_m\rightarrow\GL_\bfr$ be a one-parameter subgroup such that $x_0=\lim_{t\rightarrow0}\lambda(t)\cdot x$ exists. Let us call $V^\bullet$ the filtration of $k^{\oplus\bfr}$ associated to $\lambda$. Then the following are equivalent:
\begin{enumerate}[label=\roman*)]
\item $x_0\in R(Q,\bfr)^{\theta-\sst}$;
\item $\langle\lambda,\theta\rangle=0$;
\item for all $p\in\bbZ$, we have $\theta\cdot\dim V^p=0$.
\end{enumerate}
\end{lemma}

\begin{proof}
By Lemma \ref{Lem/1PSlimit} (for a quiver without multiplicities), we have that $x(V^p)\subset V^p$. Since $x$ is $\theta$-semistable, we obtain that $\theta\cdot\dim V^p\geq0$ for all $p\in\bbZ$. Moreover, we have \cite[\S 3]{Kin94}:
\[
\langle\lambda,\theta\rangle=\sum_{p\in\bbZ}\theta\cdot\dim V^p.
\]
Thus we get (ii) $\Leftrightarrow$ (iii). The equivalence (i) $\Leftrightarrow$ (ii) is well-known. It can be seen as a very degenerate case of \cref{Thm/extGradQuot}.
\end{proof}

We now translate the stability criterion from Lemma \ref{Lem/HM1} in terms of representations of $(Q,\bfm)$.

\begin{lemma}\label{Lem/HM2}
Let $x\in R(Q_\bfm,\bfr)$ and $y\in R(Q,\bfm;\bfr)$ be the corresponding representation under the isomorphism (\ref{Eqn/thickQuiver}). The following are equivalent:
\begin{enumerate}[label=\roman*)]
\item for any subrepresentation $V\subset V(x)$ such that $\theta\cdot\dim V=0$, we have $\rho\cdot\dim V\geq0$ (resp. $\rho\cdot\dim V>0$); 
\item for any $Q_0$-graded $k$-subspaces $V\subset k^{\oplus\bfr}$ such that $\theta\cdot\dim V=0$ and $(V_i\otimes_kk_{m_i})_{i\in Q_0}$ form a subrepresentation of $M(y)$, we have $\rho\cdot\dim V\geq0$ (resp. $\rho\cdot\dim V>0$).
\end{enumerate}
\end{lemma}

\begin{proof}[Proof of Lemma \ref{Lem/HM2}]
The lemma is an immediate consequence of the following observation (see Section \ref{Sect/Truncation}): for any $(a,m,f_1,f_2)\in (Q_\bfm)_1$, the matrix $x_{a,m,f_1,f_2}$ corresponds to the morphism
\[
\epsilon^{f_2}\cdot k^{\oplus r_i}
\rightarrow
\epsilon^{mf_{ij}+f_1}\cdot k^{\oplus r_j}
\]
induced by $y_a$.
\end{proof}

\subsubsection{Semistable locus for the non-reductive action}

We now turn to the semistable locus of $R(Q,\bfm;\bfr)\times\bbA^1$ under the action of $\tilde{\GL}_{\bfm,\bfr}:=\GL_{\bfm,\bfr}\times\GG_m$. Consider the unipotent radical $U_{\bfm,\bfr}:=\mathrm{Ker}(\GL_{\bfm,\bfr}\rightarrow\GL_r)$. We want to understand the $U_{\bfm,\bfr}$-sweep of the previous semistable locus. We will need the following lemma.

\begin{lemma}\label{Lem/HM3}
Let $m\geq1$ and $r\geq1$. Consider $M\subset k_m^{\oplus r}$ a free $k_m$-submodule. Then there exists $u\in U_{m,r}$ and a $k$-subspace $V\subset k^{\oplus r}$ such that $M=u\cdot (V\otimes_kk_m)$.
\end{lemma}

\begin{proof}[Proof of Lemma \ref{Lem/HM3}]

Let $r'$ be the rank of $M$ and let $\mathrm{Gr}_m(r',r)(k)$ denote the set of all locally free submodules of $k_m^{\oplus r}$ of rank $r'$.  

Let $\Mat_{m,r\times r'}^0(k)$ be the set of matrices of size $r\times r'$ and coefficients in $k_m$, whose reduction modulo $\epsilon$ has rank $r'$. The action of $\GL_{m,r'}(k)$ by right multiplication is free and we have a bijection:
\[
\begin{array}{rcl}
\Mat_{m,r\times r'}^0(k)/\GL_{m,r'}(k) & \rightarrow & \mathrm{Gr}_m(r',r)(k) \\
(v_1\vert\ldots\vert v_{r'}) & \mapsto & \sum_{i=1}^{r'}k_m\cdot v_i,
\end{array}
\]
which is equivariant for the natural transitive $\GL_{m,r}(k)$-actions (left multiplication on the left-hand side).

Fix the subspace:
\[
V_0=\bigoplus_{i=1}^{r'}k\cdot e_i\subset k^{\oplus r}.
\]
Then locally free submodules of $k_m^{\oplus r}$ of the form $V\otimes_kk_m$ are exactly those given by:
\[
g\cdot(V_0\otimes_kk_m)=(g\cdot V_0)\otimes_kk_m,\ g\in\GL_{r'}(k).
\]
On the other hand, since the action of $\GL_{m,r}(k)$ on $\mathrm{Gr}_m(r',r)(k)$ is transitive, we have that:
\[
\mathrm{Gr}_m(r',r)(k)=U_{m,r}(k)\GL_r(k)\cdot (V_0\otimes_kk_m).
\]
This proves the claim.
\end{proof}

We can now proceed to the proof of the main result of this section:

\begin{theorem}\label{Thm/HMcrit}
The GIT semistable (resp.\ stable) locus in $R(Q,\bfm;\bfr)$ with respect to $q_\theta \colon R(Q,\bfr)^{\theta-\sst} \rightarrow M_{Q,\bfr}^{\theta-\sst}$ and $\rho$ is the open subset whose geometric points correspond to $(\theta,\rho)$-(semi)stable representations of $(Q,\bfm)$, in the sense of Definition \ref{def/stabilityQuivMult}.
\end{theorem}

\begin{remark}
Note that, when $\theta$ is generic with respect to $\bfr$ (see Definition \ref{Def/genericStabilityParameter}), the second condition in Definition \ref{def/stabilityQuivMult} is vacuous. Therefore, in this setup, twisted relative affine quotients differ from their untwisted analogues only if there exist strictly $\theta$-semistable points in the base of the relative affine quotient.
\end{remark}

\begin{proof}
We use the description of the semistable locus in \cref{Thm/extGradQuot}, which involves considering the associated internally graded action:
\[
\GL_{\bfm,\bfr}\rtimes\GG_m\curvearrowright R(Q,\bfm;\bfr)\times\bbA^1.
\]
Set $n\geq0$ and consider the character $\tilde{\rho}=(\rho,n)$. By \cref{rmk/alt descr HM}, we have:
\[
\left(R(Q,\bfm;\bfr)\times(\bbA^1\setminus\{0\})\right)^{q_\theta,\tilde{\rho}-\sst, \tilde{\GL}_{\bfm,\bfr}}
=
\bigcap_{u\in U_{\bfm,\bfr}}u\cdot\left(R(Q,\bfm;\bfr)\times(\bbA^1\setminus\{0\})\right)^{q_\theta,\tilde{\rho}-\sst, \tilde{\GL}_\bfr}.
\]
Consider $x\in R(Q,\bfm;\bfr)$ and $M_i\subset k_{m_i}^{\oplus r_i}$. Note that $x$ preserves the modules $M_i,\ i\in Q_0$ if and only if $u\cdot x$ preserves the modules $u_i\cdot M_i,\ i\in Q_0$. The result then follows by combining Lemmas \ref{Lem/HM1}, \ref{Lem/HM2} and \ref{Lem/HM3}.
\end{proof}

\subsection{S-equivalence classes}

In this section, we describe closed orbits and S-equivalence classes of $(\theta,\rho)$-semistable representations of $(Q,\bfm)$, for some stability parameters $\theta,\rho$.

We will need a refinement of \cref{Lem/1PSlimit}.

\begin{lemma}\label{Lem/1PSvsJHfilt}
Let $x\in R(Q,\bfm;\bfr)^{q_\theta,\rho-\sst}$ and $\lambda \colon \GG_m\rightarrow\GL_{\bfm,\bfr}$ be a one-parameter subgroup. Then the datum of $\lambda$ is equivalent to that of a filtration $F_\bullet$ of $\bigoplus_{i\in Q_0}k_{m_i}^{\oplus r_i}$ by locally free $Q_0$-graded submodules. 

Moreover, $x_0:=\lim_{t \rightarrow 0}\lambda(t)\cdot x$ exists in $R(Q,\bfm;\bfr)^{q_\theta,\rho-\sst}$ if and only if $F_\bullet$ is a coarsening of a \emph{naive} Jordan-H\"older filtration of $M(x)$. In that case, we have $M(x_0)\simeq\mathrm{gr}_{F_\bullet}M(x)$.
\end{lemma}

\begin{proof}
Let us prove the first claim. Let $U_{\bfm,\bfr}$ denote the unipotent radical of $\GL_{\bfm,\bfr}$. There exists $u\in U_{\bfm,\bfr}$ such that $u\lambda u^{-1}$ is a one-parameter subgroup of $\GL_\bfr\subset\GL_{\bfm,\bfr}$. Then, by the discussion in \cref{Sect/HManalysis} and \cref{Lem/HM3}, $\lambda$ induces a filtration of $\bigoplus_{i\in Q_0}k_{m_i}^{\oplus r_i}$ by locally free $Q_0$-graded submodules of the form
\[
u\cdot\left(\bigoplus_{i\in Q_0}W_i\otimes_kk_{m_i}\right)
,
\]
where $W_i\subset k^{\oplus r_i}$ is a $k$-vector subspace. Conversely, given such a filtration $F_\bullet$, one recovers the one parameter subgroup $\lambda \colon \GG_m\rightarrow\GL_{\bfm,\bfr}$ acting with weight $i$ on the $i$th subquotient of $F_\bullet$.

Combining \cref{Lem/1PSlimit} and \cref{Lem/HM3}, we obtain that $\lim_{t\rightarrow 0}\lambda(t)\cdot x$ exists in $R(Q,\bfm;\bfr)$ if and only if the filtration $F_\bullet$ consists of subrepresentations of $M(x)$. Moreover, we have:\[
\langle\lambda,\theta\rangle=\sum_i\theta\cdot\rk F_iM(x)
\]
and each of the above summands is non-negative, as $M(x)$ is $(\theta,\rho)$-semistable. If $x_0\in R(Q,\bfm;\bfr)^{q_\theta,\rho-\sst}$, then we have $\langle\lambda,\theta\rangle=0$, hence $\theta\cdot\rk F_iM(x)=0$ for all $i$. Since  $M(x)$ is $(\theta,\rho)$-semistable, this further implies that $\rho\cdot\rk F_iM(x)\geq0$ for all $i$. As $x_0\in R(Q,\bfm;\bfr)^{q_\theta,\rho-\sst}$ also gives $\langle\lambda,\rho\rangle=0$ by \cref{Thm/extGradQuot}, we deduce from
\[
\langle\lambda,\rho\rangle=\sum_i\rho\cdot\rk F_iM(x)
\]
that $\rho\cdot\rk F_iM(x)=0$ for all i. This shows that $F_\bullet$ is a coarsening of a \emph{naive} Jordan-H\"older filtration of $M(x)$. Conversely, if $\lambda$ corresponds to such a filtration of $M(x)$, then the above formulas yield $\langle\lambda,\theta\rangle=\langle\lambda,\rho\rangle=0$, so we indeed have $x_0\in R(Q,\bfm;\bfr)^{q_\theta,\rho-\sst}$.

Finally, we have $M(x_0)\simeq\mathrm{gr}_{F_\bullet}M(x)$, by construction (the reasoning is the same as in \cite[\S 3]{Kin94}).
\end{proof}

We can now easily characterise closed orbits and ensure that every $(\theta,\rho)$-semistable representation of $(Q,\bfm)$ has a well-defined associated $(\theta,\rho)$-polystable representation.

\begin{proposition}\label{Prop/closedOrbits}
Let $x\in R(Q,\bfm;\bfr)^{q_\theta,\rho-\sst}$. The following are equivalent:
\begin{enumerate}[label=\roman*)]
\item The $\GL_{\bfm,\bfr}$-orbit of $x$ is closed in $R(Q,\bfm;\bfr)^{q_\theta,\rho-\sst}$.
\item The representation $M(x)$ is $(\theta,\rho)$-polystable.
\end{enumerate}
\end{proposition}

\begin{proof}
The implication i) $\Rightarrow$ ii) is a direct consequence of \cref{Thm/extGradQuot} and \cref{Lem/1PSvsJHfilt}. Let us prove ii) $\Rightarrow$ i). Suppose that $M(x)$ is $(\theta,\rho)$-polystable. By \cref{Thm/extGradQuot}, it suffices to prove that, for any one-parameter subgroup $\lambda \colon \GG_m\rightarrow\GL_{\bfm,\bfr}$ such that $x_0:=\lim_{t\rightarrow0}\lambda(t)\cdot x$ exists in $R(Q,\bfm;\bfr)^{q_\theta,\rho-\sst}$, we have $x_0\in\GL_{\bfm,\bfr}\cdot x$. Let $\lambda$ be a one-parameter subgroup such that this limit exists. By \cref{Lem/1PSvsJHfilt}, it corresponds to a coarsening $F_\bullet$ of a \emph{naive} Jordan-H\"older filtration $F'_\bullet$ of $M(x)$. Let $\lambda':\GG_m\rightarrow\GL_{\bfm,\bfr}$ be the one-parameter subgroup corresponding to $F'_\bullet$ by \cref{Lem/1PSvsJHfilt}. Since $M(x)$ is $(\theta,\rho)$-polystable, we obtain that $M(x)\simeq\mathrm{gr}_{F'_\bullet}M(x)$. Moreover, we have $M(x)\simeq\mathrm{gr}_{F'_\bullet}\left(\mathrm{gr}_{F_\bullet}M(x)\right)$, by construction. Therefore, we have that \[
x'_0:=\limit\lambda'(t)\cdot x_0=\limit\lambda'(t)\cdot x\in\GL_{\bfm,\bfr}\cdot x.
\]
This ensures that $x_0$ lies in the orbit of $x$, as otherwise, by examining dimensions as in the proof of \cite[Prop.\ 1.66]{Mil17}, the orbit closure of $x_0$ would contain $x'_0$, but not $x$, which gives a contradiction. This finishes the proof.
\end{proof}

\begin{proposition}\label{Prop/JHassociatedGraded}
Let $M$ be a $(\theta,\rho)$-semistable rank $\bfr$ representation of $(Q,\bfm)$. Then $M$ admits a Jordan-H\"older filtration $F_\bullet$. Any two $(\theta,\rho)$-polystable representations obtained as the associated graded of a Jordan-H\"older filtration of $M$ are isomorphic.
\end{proposition}

\begin{proof}
Consider $x\in R(Q,\bfm;\bfr)^{q_\theta,\rho-\sst}$ such that $M\simeq M(x)$. By \cref{Thm/extGradQuot}, there is a unique closed orbit in the closure of $\GL_{\bfm,\bfr}\cdot x$ in this semistable locus and there exists a one-parameter subgroup $\lambda \colon \GG_m\rightarrow\GL_{\bfm,\bfr}$ such that $x_0=\lim_{t\rightarrow0}\lambda(t)\cdot x$ lies in that closed orbit. By \cref{Lem/1PSvsJHfilt}, $\lambda$ corresponds to a coarsening $F_\bullet$ of a \emph{naive} Jordan-H\"older filtration $F'_\bullet$ of $M(x)$. Let $\lambda'$ be the one-parameter subgroup associated to $F'_\bullet$ by \cref{Lem/1PSvsJHfilt}. Then
\[
x'_0:=\limit\lambda'(t)\cdot x_0=\limit\lambda'(t)\cdot x
\]
also lies in the closed orbit mentioned above. Thus $F'_\bullet$ is a Jordan-H\"older filtration of $M(x)$, which finishes the proof.
\end{proof}

From the above results we deduce that the notions of S-equivalence introduced in \cref{Def/JHfiltrations} and \cref{Thm/extGradQuot} coincide.

\subsection{Construction of moduli spaces}

\subsubsection{Applying relative non-reductive GIT to prove \cref{first main thm}}\label{sec proof of first main theorem}

Let $(Q,\bfm)$ be a quiver with multiplicities,  $\bfr\in\bbZ_{\geq0}^{Q_0}$ be a rank vector and $\theta,\rho\in\bbZ^{Q_0}$ be stability parameters. The aim of this section is to prove \cref{first main thm} by applying \cref{Thm/extGradQuot} to the equivariant action \eqref{equiv action quivers} using the external grading in $\S$\ref{Sect/extGrad}. For this we suppose the assumption \ref{U} holds for rank $r$ representations of $(Q,\bfm)$, which precisely ensures the unipotent stabiliser condition appearing in \cref{Thm/extGradQuot} is satisfied.

\begin{proof}[Proof of \cref{first main thm}]
  We apply \cref{Thm/extGradQuot} to the equivariant action of $\overline{\tau} \colon G_{\bfm,\bfr} \rightarrow G_{\bfr}$ on $\tau\colon R(Q,\bfm;\bfr) \rightarrow R(Q,\bfr)$, which is externally graded using the $\GG_m$-action in \cref{Lem/extGrad2} where $\beta_a= \alpha_i(f_{ji}-1)$. We take this quotient relative to King's GIT quotient of $R(Q,\bfr)$ with respect to $\theta$
  \[ q_\theta \colon R(Q,\bfr)^{\theta-\sst} \rightarrow M^{\theta-\sst}_{Q,\bfr} :=  R(Q,\bfr) \git_\theta \GL_{\bfr}. \]
  By \cref{Thm/extGradQuot}, we obtain a $\rho$-twisted relative affine GIT quotient 
  \[ R(Q,\bfm;\bfr)^{\theta,\rho-\sst}:=R(Q,\bfm;\bfr)^{q_\theta,\rho-\sst} \rightarrow M_{Q,\bfm;\bfr}^{\theta,\rho-\sst}:=R(Q,\bfm;\bfr) \git_{\! q_\theta,\rho} \GL_{\bfm,\bfr}\]
  which satisfies i) and the first claim in ii). Since $R(Q,\bfm;\bfr)$ is an affine space, and in particular normal and irreducible, so is the good quotient $M_{Q,\bfm;\bfr}^{\theta,\rho-\sst}$. Part iii) follows as by \cref{Thm/HMcrit}, the geometric points of $R(Q,\bfm;\bfr)^{q_\theta,\rho-\sst}$ correspond precisely to the $(\theta,\rho)$-semistable representations and similarly for stability. The claim about the geometric points correspond to S-equivalence classes follows from \cref{Thm/extGradQuot} v), \cref{Prop/closedOrbits} and \cref{Prop/JHassociatedGraded}.

  For the final dimension statement in part ii) and for part iv), we claim that the map \[R(Q,\bfm;\bfr)^{\theta,\rho-\st}:=R(Q,\bfm;\bfr)^{q_\theta,\rho-\st} \rightarrow M_{Q,\bfm;\bfr}^{\theta,\rho-\st}:=R(Q,\bfm;\bfr)^{q_\theta,\rho-\st}/G_{\bfm,\bfr}\] is a principal $G_{\bfm,\bfr}$-bundle. Indeed, by \cref{Lem/stabilisersStableLocus}, all $k$-points of $R(Q,\bfm;\bfr)^{q_\theta,\rho-\st}$ have trivial stabilisers, so the étale slice theorem \cite[Thm 4.5]{AHR20} guarantees that the quotient map is a principal bundle in the neighbourhood of any $k$-point. The claim follows, as $k$-points are dense on the quotient \cite[\S 10]{Gro66}. In particular, when the stable locus is non-empty, then $M_{Q,\bfm;\bfr}^{\theta,\rho-\sst}$ contains $M_{Q,\bfm;\bfr}^{\theta,\rho-\st}$ as a dense open subset, and so 
  \[ \dim M_{Q,\bfm;\bfr}^{\theta,\rho-\sst}  = \dim R(Q,\bfm;\bfr) - \dim G_{\bfm,\bfr} = \delta - \langle \bfr,\bfr \rangle_{Q,\bfm},\]
  as  $G_{\bfm,\bfr} = \GL_{\bfm,\bfr}/\Delta_\bfm$ and $\dim \Delta_\bfm = \delta$ (see \cref{Notation Euler pairing} for the definition of $\langle \bullet,\bullet \rangle_{Q,\bfm}$).
  
  Finally under the assumptions of part iv), we have $M_{Q,\bfm;\bfr}^{\theta,\rho-\sst}=M_{Q,\bfm;\bfr}^{\theta,\rho-\st}$ and the quotient is a principal $G_{\bfr,\bfm}$-bundle, so in particular, the base $M_{Q,\bfm;\bfr}^{\theta,\rho-\sst}$ is smooth. To prove that this is also a fine moduli space, one must prove the tautological family on $R(Q,\bfm;\bfr)$ descends to a universal family on $M_{Q,\bfm;\bfr}^{\theta,\rho-\sst}$.

  The tautological family descends to a universal family on the moduli stack $\Rep_{Q,\bfm;\bfr}=[R(Q,\bfm;\bfr)/\GL_{\bfm,\bfr}]$, which can be restricted to $[R(Q,\bfm;\bfr)^{q_{\theta},\rho-\st}/\GL_{\bfm,\bfr}]$. Moreover, the natural map \[ [R(Q,\bfm;\bfr)^{q_{\theta},\rho-\st}/\GL_{\bfm,\bfr}]\rightarrow M_{Q,\bfm;\bfr}^{\theta,\rho-\st}\] is a $\Delta_{\bfm}$-gerbe by \cref{Lem/stabilisersStableLocus} below. Therefore, a family on $[R(Q,\bfm;\bfr)^{q_{\theta},\rho-\st}/\GL_{\bfm,\bfr}]$ descends to $M_{Q,\bfm;\bfr}^{\theta,\rho-\st}$ if and only if the action of $\Delta_{\bfm}$ on its fibres is trivial. To obtain this, one proceeds exactly as in \cite[Proposition 5.3]{Kin94} to modify the $\Delta_{\bfm}$-action on the tautological bundles over $R(Q,\bfm;\bfr)$ by multiplying by a character of opposite weight, which exists as $\bfr$ is indivisible.
\end{proof}

\begin{lemma}\label{Lem/stabilisersStableLocus}
The group $G_{\bfm,\bfr}$ acts (set-theoretically) freely on $R(Q,\bfm;\bfr)^{q_{\theta},\rho-\st}$.
\end{lemma}

\begin{proof}
We just need to check that any $k$-point of $R(Q,\bfm;\bfr)^{q_{\theta},\rho-\st}$ has trivial stabiliser in $G_{\bfm,\bfr}$. Note that we have a short exact sequence of algebraic groups over $k$:
\[
1\rightarrow\Delta_\bfm\rightarrow\Stab_{\GL_{\bfm,\bfr}}(x)\rightarrow\Stab_{G_{\bfm,\bfr}}(x)\rightarrow 1.
\]
Since $\Delta_\bfm$ is connected and $\Stab_{G_{\bfm,\bfr}}(x)$ is finite (as $x$ is stable), we obtain that $\Stab_{G_{\bfm,\bfr}}(x)$ is the group of components of $\Stab_{\GL_{\bfm,\bfr}}(x)$. Moreover, $\Stab_{\GL_{\bfm,\bfr}}(x)$ is an open subset in the affine space $\Stab_{\gl_{\bfm,\bfr}}(x)$, so it is connected. This yields that $\Stab_{G_{\bfm,\bfr}}(x)$ is trivial.
\end{proof}

\begin{remark}\label{Rmk/constantMultiplicitiesJetBundles}
When $\bfm=m\mathbf{1}$ for some $m\geq1$, we obtain the $(m-1)$th jet space of $M_{Q,\bfr}^{\theta-\sst}$, which is an affine bundle with fibre $\bbA^{(m-1)\dim M_{Q,\bfr}}$.

Indeed, in that case, $\theta$ must be generic with respect to $\bfr$ by \cref{Cor/U-stabilisers when multiplicities are divisible}. In particular, the quotient map $q_\theta \colon R(Q,\bfr)^{\theta-\sst}\rightarrow M_{Q,\bfr}^{\theta-\sst}$ is a $G_\bfr$-principal bundle. Moreover, applying the $(m-1)$th jet space functor to the action $\GL_\bfr\curvearrowright R(Q,\bfr)$ yields the action $\GL_{\bfm,\bfr}\curvearrowright R(Q,\bfm;\bfr)$ and the truncation map $\tau$ corresponds to truncating jets. Since jet space functors commute with fibre products \cite[Prop.\ 1.2.6]{CLNS18} and \'etale pullbacks \cite[Prop.\ 1.1]{Mus01}, the universal property of categorical quotients ensures that $M_{Q,\bfm;\bfr}^{\theta-\sst}$ is the $(m-1)$th jet space of $M_{Q,\bfr}^{\theta-\sst}$.
\end{remark}

\begin{remark}\label{rmk moduli with non-trivial unip stabilisers}
    As mentioned in \cref{rmk NRGIT with non-trivial stabilisers}, relative NRGIT also applies more generally when, rather than requiring the assumption that all unipotent stabilisers in $R(Q,\bfr)^{\theta-\sst}$ are trivial (i.e.\ contained in $\Delta_{\bfm}$), we instead fix the dimensions of the unipotent stabilisers of points in $R(Q,\bfr)^{\theta-\sst}$ for a sequence of normal subgroups in $U$ whose subquotients are abelian. This results in a moduli space constructed as a NRGIT quotient which is projective-over-affine over a subvariety of King's classical moduli space $M^{\theta-\sst}_{Q,\bfr}$ determined by fixing these stabiliser dimensions. Often these stabiliser assumptions can be interpreted moduli-theoretically.
\end{remark}

We can apply \cref{first main thm} to construct moduli spaces of framed semistable quiver representations with multiplicities, where unipotent stabiliser assumptions are satisfied by \cref{prop suff cond for trivial unip stabilisers}, as the stability condition for framed representations (without multiplicities) is generic.

\begin{corollary}\label{thm framed moduli with multiplicities}
Let $(Q,\bfm)$ be a quiver with multiplicities, let $\bfr\in\bbZ_{\geq0}^{Q_0}$ be a rank vector and $\bff\in\bbZ_{\geq0}^{Q_0}$ a framing vector. Let further $\theta\in\bbZ^{Q_0}$ be a stability parameter. Denote by $q_{\bff,\theta} \colon R(Q_\bff,\hat{\bfr})^{\hat{\theta}-\sst}\rightarrow M_{Q,\bfr}^{\bff;\theta-\sst}$ the quotient map.
Then we have an irreducible, smooth moduli space:
\[
M_{Q,\bfm;\bfr}^{\bff;\theta-\sst}:=R(Q_{\bff},\hat{\bfm};\hat{\bfr})\git_{\! q_{\bff,\theta},\rho} \GL_{\bfm,\bfr}
\]
of dimension $\sum_{i\in Q_0}m_if_ir_i-\langle\bfr,\bfr\rangle_{Q,\bfm}$. Moreover, the quotient map is a principal $G_{\widehat{\bfm},\widehat{\bfr}}=\GL_{\bfm,\bfr}$-bundle, and $M_{Q,\bfm;\bfr}^{\bff;\theta-\sst}$ is projective over affine over $M_{Q,\bfr}^{\bff;\theta-\sst}$.
\end{corollary}

As in the classical case \cite{ER09}, forgetting the framing vectors induces a Hilbert-Chow map, called $\mathrm{HC}$ below:
\[
\begin{tikzcd}[ampersand replacement=\&]
M_{Q,\bfm;\bfr}^{\bff;\theta-\sst} \ar[d, "\mathrm{HC}_\bfm"]\ar[r] \& M_{Q,\bfr}^{\bff;\theta-\sst} \ar[d, "\mathrm{HC}"] \\
M_{Q,\bfm;\bfr}^{\theta-\sst} \ar[r] \& M_{Q,\bfr}^{\theta-\sst}.
\end{tikzcd}
\]
We plan to study the properties of the Hilbert-Chow map for quivers with multiplicities in future work.

\subsubsection{Nakajima quiver varieties with multiplicities}

In this section, we construct symplectic reductions associated to quivers with multiplicities via a moment map for the doubled quiver, generalising Nakajima quiver varieties \cite{N94}.

\begin{theorem}\label{Thm/quiverVarieties}
Let $(Q,\bfm)$ be a quiver with multiplicities and $\bfr\in\bbZ_{\geq0}^{Q_0}$ a rank vector. Fix $\gamma=(\gamma_i\cdot\Id)_{i\in Q_0}\in\gl_{\bfm,\bfr}$ and stability parameters $\theta,\rho\in\bbZ_{\geq0}^{Q_0}$. Suppose assumption \ref{U} holds for rank $\bfr$ representations of $(\overline{Q},\bfm)$. Then the closed subscheme
\[
\mu_{(Q,\bfm),\bfr}^{-1}(\gamma)^{q_\theta,\rho-\sst}:=\mu_{(Q,\bfm),\bfr}^{-1}(\gamma)\cap R(\overline{Q},\bfm;\bfr)^{q_\theta,\rho-\sst}
\]
admits a quasi-projective good quotient $N_{Q,\bfm;\bfr}^{\theta,\rho-\sst}(\gamma)$, which we call a Nakajima quiver variety with multiplicities.

If moreover $(\theta,\rho)$ is generic with respect to $\bfr$, then $N_{Q,\bfm;\bfr}^{\theta,\rho-\sst}(\gamma)$ is a smooth (algebraic) symplectic variety, of dimension $2(\delta-\langle\bfr,\bfr\rangle_{Q,\bfm})$.
\end{theorem}

\begin{proof}[Proof of \cref{Thm/quiverVarieties}]
By \cref{first main thm}, $R(\overline{Q},\bfm;\bfr)^{q_\theta,\rho-\sst}$ admits a quasi-projective good quotient, so this is also true for the closed $\GL_{\bfm,\bfr}$-invariant subset $\mu_{(Q,\bfm),\bfr}^{-1}(\gamma)^{q_\theta,\rho-\sst}$.

Let us now suppose that $(\theta,\rho)$ is generic with respect to $\bfr$. Then all points in $\mu_{(Q,\bfm),\bfr}^{-1}(\gamma)^{q_\theta,\rho-\sst}$ are stable. Smoothness follows from \cite[Lem.\ 2.1.5]{CBVB04}, applied to $G_{\bfm,\bfr}\curvearrowright R(\overline{Q},\bfm;\bfr)$, as $\mu_{Q,\bfm;\bfr}$ is a submersion on restriction to the (semi)stable locus. Moreover
\[
\begin{split}
\dim N_{Q,\bfm;\bfr}^{\theta,\rho-\sst}(\gamma) &
=
\dim \mu_{(Q,\bfm),\bfr}^{-1}(\gamma)^{q_\theta,\rho-\sst}-\dim G_{\bfm,\bfr} \\
& =
\dim R(\overline{Q},\bfm;\bfr)-\dim \gl_{\bfm,\bfr}^0-(\dim\GL_{\bfm,\bfr}-\dim\Delta_{\bfm,\bfr}) \\
& =
2(\dim\gl_{\bfm,\bfr}-\langle\bfr,\bfr\rangle_{Q,\bfm})-(\dim\gl_{\bfm,\bfr}-\delta)-(\dim\gl_{\bfm,\bfr}-\delta) \\
& =
2(\delta-\langle\bfr,\bfr\rangle_{Q,\bfm}).
\end{split}
\]
The fact that $N_{Q,\bfm;\bfr}^{\theta,\rho-\sst}(\gamma)$ is algebraic symplectic follows from the algebraic version of the Marsden-Weinstein reduction \cite{MW74}. The discussion above \cite[Cor.\ 3.18]{HSS20} applies to our setup: on affine opens in $N_{Q,\bfm;\bfr}^{\theta,\rho-\sst}(\gamma)$, the Poisson bracket inherited by the algebra of invariant functions induces the symplectic form coming from symplectic slices.
\end{proof}

\subsection{Examples}

In this section, we illustrate the above constructions in some example. We proceed by increasing complexity, from jet spaces of smooth classical quiver moduli to moduli spaces involving two stability parameters.

\subsubsection{Grassmannians as quiver moduli spaces}

Let $\mathrm{Gr}(n,r)$ denote the classical Grassmannian parametrising $n$-dimensional vector subspaces of $k^{\oplus r}$. We recall that $\mathrm{Gr}(n,r)$ can be constructed as a moduli space $M^{r,0-\sst}_{\bullet,n}$ of $0$-semistable $r$-framed $n$-dimensional representation of the one vertex quiver $Q := \bullet$. Note that a $r$-framed $n$-dimensional representation $(V,b)$ of $Q$ consists of a $n$-dimensional vector space $V$ and a linear map $b \colon k^{\oplus r} \rightarrow V$. Such a representation is $0$-(semi)stable (in the sense of \cref{Def framed semistability}) if and only if $b$ is surjective (i.e. the linear map $b$ has rank $n$). Hence we have $q \colon R(\bullet_r,\widehat{n})^{{0}-\sst} = \Hom(k^{\oplus r},k^{\oplus n})^{\rk n} \rightarrow M^{r,0-\sst}_{\bullet,n} = R(\bullet_r,\widehat{n})^{{0}-\sst}/\GL_n(k)$ where the $\GL_n(k)$-action is free. Moreover, we have a projective morphism to the unframed moduli space:
\[ M^{r,0-\sst}_{\bullet,n}:= R(\bullet_r,\widehat{n})^{{0}-\sst} \git \GL_{n}(k)   \rightarrow  M^{0-\sst}_{\bullet,n} = \Rep(\bullet,n)\git \GL_n(k) = \Spec k. \]

Let us now explain how to construct the Grassmannian with multiplicities $\mathrm{Gr}_m(n,r)$, which parametrises rank $n$ locally free submodules of $k_m^{\oplus r}$ (and appeared in the proof of \cref{Lem/HM3}), as a moduli space of framed quivers with multiplicities, where at the framing vertex $\infty$ we take the multiplicity $m_\infty = m$ (as opposed to setting $m_\infty = 1$ as we did in \cref{Def quiver with mult associated to framing}). More precisely, we will show this naturally arises as a GIT quotient for the equivariant action of $\GL_n(k_m) \rightarrow \GL_n(k)$ on $\tau \colon \Rep(\bullet_r,\bfm; \widehat{n}) \rightarrow \Rep(\bullet_r,\bfm; \widehat{n})$ where $\bfm = (m,m)$ that is taken with respect to the good (in fact geometric) $\GL_n(k)$-quotient $q \colon R(\bullet_r,\widehat{n})^{{0}-\sst} \rightarrow M^{r,0-\sst}_{\bullet,n}$ and trivial choice of twisting character $\rho$. We note that as semistability coincides with stability for $r$-framed $n$-dimensional representations of $Q = \bullet$, we have the necessary unipotent stabiliser condition. We obtain a geometric quotient
\[  \Rep(\bullet_r,\bfm; \widehat{n})^{ss} \rightarrow M^{r,0-\sst}_{\bullet,m;n}:= R(\bullet_r,\bfm,\widehat{n}) \git_q \GL_{n}(k_m) \]
where $\Rep(\bullet_r,\bfm; \widehat{n})^{ss} =\tau^{-1} (R(\bullet_r,\widehat{n})^{{0}-\sst})= \Mat^\circ_{r \times n}(k_m)$ consists of $r \times n$-matrices whose reduction modulo $\epsilon$ has rank $n$, and by construction there is an affine morphism 
\[\mathrm{Gr}_m(n,r)=M^{r,0-\sst}_{\bullet,m;n} \rightarrow \mathrm{Gr}(n,r) =M^{r,0-\sst}_{\bullet,n}.\] In fact, as we are working with a constant tuple of multiplicities $\bfm = (m,m)$, this morphism is just realising $\mathrm{Gr}_m(n,r)$ as the bundle of $m$-jets over $\mathrm{Gr}(n,r)$, see \cref{Rmk/constantMultiplicitiesJetBundles}.

\subsubsection{Untwisted moduli of stable representations with varying multiplicities}

\begin{example}
Let us consider the following quiver with multiplicities:
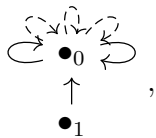
\[
\begin{tikzcd}[row sep=small, column sep=small]
\bullet_0 \ar[loop left]\ar[loop left,rotate=-60,dashed]\ar[loop above, dashed]\ar[loop right,rotate=60,dashed]\ar[loop right] \\
\bullet_1 \ar[u]
\end{tikzcd}
,
\]
where the vertex $0$ carries $g\geq1$ loops and $\bfm=(m_0,m_1)$. We set $\bfr=(r,1)$ and choose the stability parameter $\theta=(1,-r)$.

Note that for $0\leq\bfr'\leq\bfr$, we have $\theta\cdot\bfr\geq0$ unless $r'_1=1$ and $r'_0<r_0$. Thus, a representation $(x_1,\ldots,x_g,v)\in R(Q,\bfr)$ is $\theta$-(semi)stable if and only if the vector $v\in\Hom_k(k,k^{\oplus r})\simeq k^{\oplus r}$ spans $k^{\oplus r}$ under the joint action of $x_1,\ldots, x_g$.

Now consider $x=(x_1,\ldots,x_g,v)\in R(Q,\bfm;\bfr)$. Since
\[
\Hom_{k_{m_{01}}}(k_{m_1},k_{m_0}^{\oplus r})\simeq\bigoplus_{f=0}^{f_{01}-1}k_{m_0}^{\oplus r},
\]
we can view the vector $v$ as the datum of $f_{01}$ vectors $v_f\in k_{m_0}^{\oplus r},\ 0\leq f\leq f_{01}-1$, where $v_f$ corresponds to the image of $\epsilon^{f_{01}-1-f}$. Then $\tau(x)$ is $\theta$-(semi)stable if and only if the vector $v_0$ spans $k_{m_0}^{\oplus r}$ under the joint action of $x_1,\ldots,x_g$. Indeed, precomposing the map $v\in\Hom_{k_{m_{01}}}(k_{m_1},k_{m_0}^{\oplus r})$ with $\epsilon^{f_{01}-1}$ amounts to replacing $(v_0,\ldots,v_{f_{01}-1})$ with $(\epsilon^{f_{10}}v_1,\ldots,\epsilon^{f_{10}}v_{f_{01}-1},v_0)$. Note that, although the vectors $v_0,\ldots,v_{f_{01}-1}$ remind us of framing vectors, the stability condition only bears on $v_0$. So $M(x)$ is not a stable framed representation, as in Definition \ref{Def/stabilityFramedQuivMult}.

Call $M_{g,f_{01}}$ the moduli space parametrising $g$-tuples of $r\times r$ matrices with coefficients in $k$ and $f_{01}$-tuples $(v_0,\ldots, v_{f_{01}-1})$ of vectors in $k^{\oplus r}$, such that $v_0$ generates $k^{\oplus r}$ under the joint action of $x_1,\ldots,x_g$, up to simultaneous action of $\GL_r$. Then $M_{Q,\bfm;\bfr}^{\theta-\sst}$ is isomorphic to the geometric quotient of the $(m_0-1)$th jet scheme of $M_{g,f_{01}}$ by $k_{m_1}^{\times}$ acting on framing vectors.
\end{example}

\begin{example}
Let us consider the Kronecker quiver
\[
\begin{tikzcd}[ampersand replacement=\&]
\bullet_1\ar[r,shift left]\ar[r, shift right] \& \bullet_2
\end{tikzcd}
\]
with multiplicities $\bfm=(2,m)$, where $m\geq1$ is an odd integer. Fix the rank vector $\bfr=\mathbf{1}$ and the stability parameter $\theta=(-1,1)$. Note that $\theta$ is generic with respect to $\bfr$.

The datum of a $k$-linear map $\Lambda \colon k_2\rightarrow k_m$ is equivalent to the datum of two vectors $\lambda,\lambda'\in k_m$, corresponding respectively to the image of $1,\epsilon\in k_2$. Then the truncation map can be expressed as follows:
\[
\begin{array}{rcl}
R(Q,\bfm;\bfr)\simeq k_m^{\oplus 2}\oplus k_m^{\oplus 2} & \rightarrow & k\oplus k\simeq R(Q,\bfr) \\
(\lambda,\lambda',\mu,\mu') & \mapsto & (\lambda'_0,\mu'_0).
\end{array}
\]
It is well-known that $M_{Q,\bfr}^{\theta-\sst}\simeq\mathbb{P}_k^1$, with homogeneous coordinates $[\lambda'_0:\mu'_0]$. Let us compute the ring of functions of $M_{Q,\bfm;\bfr}^{\theta-\sst}$ on the open subset $D_{\lambda'_0}=\{\lambda'_0\ne0\}$ as the computation for the other open subset is analogous. By construction, this ring is:
\[
\mathcal{O}_{M_{Q,\bfm;\bfr}^{\theta-\sst}}(D_{\lambda'_0})
=
\left(\mathcal{O}(R(Q,\bfm;\bfr))\left[\frac{1}{\lambda'_0}\right]\right)^{\GL_{\bfm,\bfr}}.
\]
One can check that the following $3m-1$ functions are $\GL_{\bfm,\bfr}$-invariant and algebraically independent:
\[
\left\{
\begin{array}{ll}
\left(\frac{\mu'}{\lambda'}\right)_l, & 0\leq l\leq m-1, \\
\left(\frac{\mu}{\lambda'}\right)_l, & 1\leq l\leq m-1, \\
\left(\frac{\lambda}{\lambda'}\right)_l, & 1\leq l\leq m-1, \\
\left(\frac{\lambda}{\lambda'}-\frac{\mu}{\lambda'}\right)_0, & 
\end{array}
\right.
\]
where the subscript $l$ denotes taking the coefficient of $\epsilon^l$.

Moreover, we have $\dim M_{Q,\bfm;\bfr}^{\theta-\sst}=3m-1$, so $\mathcal{O}_{M_{Q,\bfm;\bfr}^{\theta-\sst}}(D_{\lambda'_0})$ is the polynomial ring generated by the above invariant functions. Note that ${\mu'_0}/{\lambda'_0}$ is the coordinate of the affine chart $\{\lambda'_0\ne0\}$ in $M_{Q,\bfr}^{\theta-\sst}\simeq\mathbb{P}_k^1$. Thus, we obtain that the morphism $M_{Q,\bfm;\bfr}^{\theta-\sst}\rightarrow M_{Q,\bfr}^{\theta-\sst}$ is an $\mathbb{A}_k^{3m-2}$-fibration trivialised by the open cover $D_{\lambda'_0}\cup D_{\mu'_0}$.
\end{example}

\subsubsection{Twisted moduli of representations with varying multiplicities}\label{Sect/examplesTwistedModuli}

Let us spell out stability conditions for the following quiver
\[
Q=
\begin{tikzcd}[row sep=small, column sep=small]
\bullet_1 \ar[dr] & \bullet_2 \ar[d] & \bullet_3 \ar[dl] \\
\bullet_n \ar[r] & \bullet_0 & \ldots \\
\bullet_{n-1} \ar[ur] & \ldots & \ldots
\end{tikzcd}
\]
(where $n\geq1$) with multiplicities $\bfm=(m_0,m_1,\ldots,m_1)$, where $\gcd(m_0,m_1)=1$. Set a rank vector $\bfr:=(r_0,r_1,\ldots,r_1)$. We use the notation $(x_{i})_{1\leq i\leq n}:=(x_{i})_{(i\rightarrow 0)\in Q_1}\in R(Q,\bfm;\bfr)$.

Consider stability parameters $\theta:=(nr_1,-r_0,\ldots,-r_0)$ and $\rho\in\bbZ^{Q_0}$. First note that $z\in R(Q,\bfr)$ is $\theta$-semistable if and only if for any collection of vector subspaces $W_i\subseteq k^{\oplus r_1},\ 1\leq i\leq n$, we have:
\[
\frac{\dim_k\left(\sum_{i=1}^nz_i(W_i)\right)}{\sum_{i=1}^n\dim_kW_i}\geq\frac{r_0}{nr_1}.
\]
One can easily check that, if $z$ is $\theta$-semistable, then, for all $1\leq i\leq n$, we have that $z_i$ is injective and moreover:
\[
\sum_{i=1}^n\mathrm{Im}(z_i)=k^{\oplus r_0}.
\]
From Lemma \ref{lem ustab condition when adjacent vertices are coprime}, it follows that, for any $z\in R(Q,\bfr)^{\theta-\sst}$, the point $\iota(z)\in R(Q,\bfm;\bfr)$ has trivial unipotent stabilisers.

Now let $x\in R(Q,\bfm;\bfr)$. By \cref{Def/intrinsicTruncation}, we obtain that $\tau(x)$ is $\theta$-semistable if and only if for any locally free submodules $N_i\subset k_{m_1}^{\oplus r_1},\ 1\leq i\leq n$, we have:
\[
\frac{\rk\left(\sum_{i=1}^nx_i\epsilon^{m_1-1}(N_i)\right)}{\sum_{i=1}^n\rk N_i}\geq\frac{r_0}{nr_1}.
\]

Let us now add the $\rho$-stability condition. Suppose that $\tau(x)$ is $\theta$-semistable and $N\subseteq M(x)$ is a locally free subrepresentation which satisfies $\theta\cdot\rk N=0$. Consider the subrepresentation $N'\subseteq N$ given by:
\[
N'_i
=
\left\{
\begin{array}{ll}
\sum_{i=1}^nx_i\epsilon^{m_1-1}(N_i) & \text{if }i=0, \\
N_i & \text{else},
\end{array}
\right.
\]
which might not be locally free \emph{a priori}. Since $\tau(x)$ is $\theta$-semistable, we have that $\theta\cdot\rk N'\geq0$. On the other hand, since $\theta_0>0$, we obtain:
\[
0\leq\theta\cdot\rk N'\leq\theta\cdot\rk N=0.
\]
From this, we deduce that $\rk N'_0=\rk N_0$, hence $N'_0=N_0$. In other words, the joint image of the maps $x_i\epsilon^{m_1-1},\ 1\leq i\leq n$, as well as the joint image of the maps $x_i,\ 1\leq i\leq n$, is locally free.

Therefore, we only need to additionally check that, for any locally free submodules $N_i\subseteq k_{m_1}^{\oplus r_1}$ such that the joint image $N_0$ of the maps $x_i,\ 1\leq i\leq n$ is locally free and $\theta\cdot\rk N=0$, we have $\rho\cdot\rk N\geq0$.

Summing up, we obtain that $x\in R(Q,\bfm;\bfr)$ is semistable in the sense of Theorem \ref{Thm/HMcrit} if and only if for any locally free submodules $N_i\subseteq k_{m_1}^{\oplus r_1}$, we have:
\[
\frac{\rk\left(\sum_{i=1}^nx_i\epsilon^{m_1-1}(N_i)\right)}{\sum_{i=1}^n\rk N_i}\geq\frac{r_0}{nr_1},
\]
and if moreover the joint image $N_0$ of the maps $x_i,\ 1\leq i\leq n$ is locally free of rank
\[
\frac{r_0}{nr_1}\cdot\sum_{i=1}^n\rk N_i,
\]
we have $\rho\cdot\rk N\geq0$.

\section{Cohomological results}\label{sec/purity}

In this section, we prove \cref{main thm purity}. Our proof relies on a general purity criterion for smooth varieties, which involves a contracting torus action. We show that this criterion applies to a wide range of non-reductive GIT quotients. 

\subsection{General techniques for purity}

One technique that is frequently employed to deduce cohomological purity of smooth quasi-projective complex varieties is to exhibit a $\GG_m$-action which is \emph{semi-projective} in the following sense.

\begin{definition}
    A $\GG_m$-action on a quasi-projective variety $X$ is semi-projective if the following conditions hold:
\begin{itemize}
\item for all closed points $x\in X$, the limit $\lim_{t\rightarrow0} t\cdot x$ exists;
\item the fixed-point locus $X^{\GG_m}$ is proper.
\end{itemize}
\end{definition}

The following lemma is an elementary consequence of the valuative criterion for properness. See \cite[\S 2.4]{CBVB04} or \cite[Thm.\ 2.2.7]{HLRV11} for earlier occurrences of this argument.

\begin{lemma}\label{Lem/semiprojectivity and proper morphisms}
Let $f \colon X \rightarrow Y$ be a proper $\GG_m$-equivariant map between quasi-projective $\GG_m$-varieties. Then the induced morphism $X^{\GG_m}\rightarrow Y^{\GG_m}$ is proper. Moreover, if $Y$ is semi-projective, then $X$ is semi-projective.
\end{lemma}

The following propositions collect some standard results concerning how semi-projective $\GG_m$-actions can be used to deduce cohomological purity. We start with an absolute statement.

\begin{proposition}\label{Prop purity}
\emph{\cite[Cor.\ 1.3.2]{HRV13}}
Let $X$ be a smooth quasi-projective $\GG_m$-variety.
\begin{enumerate}[label=\roman*)]
    \item\label{Prop purity item 1} If the $\GG_m$-action on $X$ is semi-projective, then the cohomology of $X$ carries a pure Hodge structure.
    \item\label{Prop purity item 2} If there is a proper $\GG_m$-equivariant map $f \colon X \rightarrow Y$ to a (not necessarily smooth) $\GG_m$-variety $Y$ such that the $\GG_m$-action on $Y$ is semi-projective, then the cohomology of $X$ carries a pure Hodge structure.
\end{enumerate}
\end{proposition}

Point ii. above follows from \cref{Lem/semiprojectivity and proper morphisms}.

This criterion (or its variant for $l$-adic cohomology \cite[App.\ A]{CBVB04}) can be applied to several moduli spaces: moduli of (framed) quiver representations \cite{ER09}, Nakajima quiver varieties \cite{CBVB04,Hau10}, hypertoric varieties \cite{HS02} and moduli spaces of Higgs bundles on a curve \cite{HRV13}.

As in \cite{CBVB04} and \cite{HLRV11}, we will also study the cohomology of fibres in a $\GG_m$-equivariant $\mathbb{A}^1$-family. For this, we will need a relative version of the previous criterion (see also \cite[App.\ A]{HPL21} for a similar result concerning Voevodsky motives). 

\begin{proposition}\label{Prop/purityCriterion}
\emph{\cite[Thm.\ B.1]{HLRV11}}
Let $f \colon \mathcal{X}\rightarrow\bbA_{\bbC}^1$ be a smooth morphism of complex quasi-projective varieties. Suppose that $\mathcal{X}$ is endowed with a semi-projective $\GG_m$-action such that $f$ is equivariant with respect to a scaling action of positive weight on $\bbA_{\bbC}^1$. Then the cohomology groups of any two fibres of $f$ are isomorphic and carry pure Hodge structures, which are preserved by this isomorphism.
\end{proposition}

\subsection{Purity results in a relative setting}\label{Subsect/RelativePurityTech}

In this section we will develop results that show, under certain assumptions, that non-reductive relative GIT quotients have pure cohomology. Since these quotients are described as relative spectra, we study the behaviour of fixed-point loci under affine morphisms.

\begin{lemma}\label{Lem/fixedPoints}
Let $\pi \colon Y =\rSpec_X\mathcal{B} \rightarrow X$ be an equivariant affine finite-type surjective morphism of complex finite-type $\GG_m$-schemes. Suppose the following conditions hold:
\begin{enumerate}[label=\roman*)]
\item\label{item 1 Lem/fixedPoints} For every closed point $x\in X$, the limit $\underset{t\rightarrow0}{\lim}\ t\cdot x$ exists;
\item\label{item 2 Lem/fixedPoints} For every $\GG_m$-fixed closed point $x_0\in X$, the $\widehat{\mathcal{O}_{X,x_0}}$-algebra
\[
\widehat{\mathcal{B}_{x_0}}
:=
\mathcal{B}_{x_0}\otimes_{\mathcal{O}_{X,x_0}}\widehat{\mathcal{O}_{X,x_0}}
\]
is generated by elements of negative weight with respect to the $\GG_m$-action.
\end{enumerate}
Then for every closed point $y\in Y$, the limit $\underset{t\rightarrow0}{\lim}\ t\cdot y$ exists and $\pi$ induces a bijective closed immersion $Y^{\GG_m}\hookrightarrow X^{\GG_m}$.
\end{lemma}

\begin{proof}
Let $y\in Y$ be a closed point and $x=\pi(y)$.  The orbit map $\sigma_x \colon \GG_m \rightarrow X$ given by $t \mapsto t \cdot x$ extends to $\mathbb{A}^1$ by Assumption \ref{item 1 Lem/fixedPoints}, and we let $x_0:=\sigma_x(0)$ denote the limit point. The extended orbit map $\sigma_x \colon \mathbb{A}^1 \rightarrow X$, when restricted to an open affine set $U \subset X$ containing $x_0$, induces the following commutative diagram of $k$-algebras:
\[
\begin{tikzcd} 
\mathcal{O}_X(U) \ar[d] \ar[r] & \widehat{\mathcal{O}_{X,x_0}} \ar[d]\ar[r, "\widehat{\sigma^{\#}_{x,0}}"] & \bbC\lpow t\rpow . \ar[d] 
\\
\mathcal{B}(U) \ar[r] & \widehat{\mathcal{B}_{x_0}} \ar[r] & \bbC\llau t\rlau
\end{tikzcd}
\]
Assumption \ref{item 2 Lem/fixedPoints} implies that the homomorphism $\mathcal{B}(U)\rightarrow \bbC\llau t\rlau$ induced by the bottom arrows of the diagram factors through $\bbC\lpow t\rpow$. Therefore, the orbit map for $y$ induces a homomorphism $\mathcal{B}(U)\rightarrow \bbC[t]_{(0)}:= \cO_{\mathbb{A}^1,0}$, which means that the orbit map $\sigma_y \colon \GG_m \rightarrow Y$ extends to $\mathbb{A}^1$, so $\lim_{t\rightarrow0}\ t\cdot y$ exists.

Moreover, since $\pi$ is surjective, the induced morphism $Y^{\GG_m}\rightarrow X^{\GG_m}$ is also surjective. Indeed, given any $x_0\in X^{\GG_m}$, there exists $y\in Y$ such that $\pi(y)=x_0$. Then $y_0:=\lim_{t\rightarrow0} t\cdot y$ is a preimage of $x_0$ in $Y^{\GG_m}$.

We now examine the morphism $Y^{\GG_m}\rightarrow X^{\GG_m}$ locally. For $x_0 \in X^{\GG_m}$, let $I_{x_0}\subset\mathcal{O}_{X,x_0}$ (resp. $J_{x_0}\subset\mathcal{B}_{x_0}$) be the ideal generated by functions which have non-zero weight in $\widehat{\mathcal{B}_{x_0}}$. Then the morphism $Y^{\GG_m}\rightarrow X^{\GG_m}$ is locally given by the algebra homomorphism
\[
\mathcal{O}_{X,x_0}/I_{x_0}
\rightarrow
\mathcal{B}_{x_0}/J_{x_0}.
\]
Assumption \ref{item 2 Lem/fixedPoints} implies that generators of $\mathcal{B}_{x_0}$ over $\mathcal{O}_{X,x_0}$ are contained in $J_{x_0}$. Thus the above algebra homomorphism is surjective, which proves that $Y^{\GG_m}\rightarrow X^{\GG_m}$ is a closed immersion. This completes the proof, as we have already observed that $Y^{\GG_m}\rightarrow X^{\GG_m}$ is surjective.
\end{proof}

\begin{remark}
In \cref{Lem/fixedPoints}, Assumption \ref{item 2 Lem/fixedPoints} refers to the completed local ring $\widehat{\mathcal{O}_{X,x_0}}$ (resp. $\widehat{\mathcal{B}_{x_0}}$), in order to make sense of the assumptions on weights. Indeed, the local ring $\mathcal{O}_{X,x_0}$ (resp. $\mathcal{B}_{x_0}$) does not necessarily inherit a grading from the $\GG_m$-action, whereas the completed local ring is the additive span (topologically, for the adic topology) of its homogeneous elements.

Alternatively, one could work with affine $\GG_m$-invariant neighbourhoods of $x_0$. Such neighbourhoods exist, for instance, when $X$ (resp. $Y$) is normal \cite{Sum74}. The proof of Proposition \ref{Prop/weightAssumption} below uses such neighbourhoods, which can be constructed in an elementary way in the setup of the proposition.
\end{remark}

\begin{corollary}\label{corollary affine map semiproj}
    Let $\pi \colon Y =\rSpec_X \mathcal{B} \rightarrow X$ be an equivariant affine finite-type surjective morphism of complex finite-type $\GG_m$-schemes. Suppose the $\GG_m$-action on $X$ is semi-projective and \cref{Lem/fixedPoints} \ref{item 2 Lem/fixedPoints} holds. Then the $\GG_m$-action on $Y$ is semi-projective.
\end{corollary}

Assumption \ref{item 2 Lem/fixedPoints} in Proposition \ref{Lem/fixedPoints} may seem difficult to check when $\mathcal{B}$ is an algebra sheaf of invariants. However, we prove a sufficient condition for that assumption to be met, in a favourable setup.

\paragraph{Setup.}
Let $f \colon X = \rSpec_{Z}\cA \rightarrow Z$ be an affine morphism between affine schemes of finite type, which is endowed with an equivariant action of $G\rtimes T_g=(U\rtimes R)\rtimes T_g\rightarrow R$, where $R$ (resp.\ $U$) is a reductive (resp.\ unipotent) algebraic group and $T_g \cong \GG_m$ is a multiplicative group that (externally) grades the equivariant action of $G \rightarrow R$.

Let $\theta$ be a character of $R$ and let $q \colon Z^{\theta-\sst} \rightarrow W:=Z\git_{\hspace{-2pt} \theta} R$ denote the reductive GIT quotient with respect to $\theta$. Denote by $j \colon Z^{\theta-\sst}\hookrightarrow Z$ the open immersion of the semistable locus. We also write $X^{q-\sst}:=f^{-1}(Z^{\theta-\sst})$ and $f\colon X^{q-\sst} =\rSpec_{Z^{\theta-\sst}} j^*\mathcal{A} \rightarrow Z$ for the induced affine morphism. 

Suppose all points of $Z^{\theta-\sst}\subseteq X^{q-\sst}$ have trivial $U$-stabilisers. Then by Theorem \ref{Thm/extGradQuot} (for the trivial choice of $\rho$), the scheme $X^{q-\sst}$ admits a good $G$-quotient $p \colon X^{q-\sst} \rightarrow Q:=X\git_{\hspace{-2pt} q} G$ given by the (untwisted) relative affine GIT quotient $X\git_{\hspace{-2pt} q} G :=\rSpec_W q_*j^*\cA^G$. In particular, there is an induced affine morphism $\overline{f} \colon Q =X\git_{\hspace{-2pt} q} G\rightarrow W=Z\git_{\hspace{-2pt} \theta} G$ on quotients.

Now, suppose that $X$ and $Z$ are endowed with conical actions of $T_c \cong \GG_m$ (i.e.\ all non-constant functions have negative $T_c$-weight), which commute with the actions of $G\rtimes T_g$ and $R$, and such that $f$ is $T_c$-equivariant. For a positive integer $N$, we will consider the action of the following one-dimensional torus $T_N \subset T:=T_g\times T_c$
\[
\begin{array}{rcl}
T_N & \rightarrow & T_g\times T_c=T \\
t & \mapsto & (t^N,t).
\end{array}
\]

To apply \cref{Lem/fixedPoints} to the induced map $\overline{f} \colon Q =X\git_{\hspace{-2pt} q} G\rightarrow W=Z\git_{\hspace{-2pt} \theta} G$ between quotients, we plan to use the following proposition. 

\begin{proposition}\label{Prop/weightAssumption}
In the above setup, there exists an integer $N_0$ such that, for $N\geq N_0$, the sheaf of $\mathcal{O}_W$-algebras $\mathcal{B}=q_*j^*\mathcal{A}^{G}$ satisfies Assumption \ref{item 2 Lem/fixedPoints} of \cref{Lem/fixedPoints} for the action of $T_N \cong \GG_m$. In particular, if the $T_N$-action on $W$ is semi-projective, then so is the $T_N$-action on $Q = \rSpec_W  \mathcal{B}$.
\end{proposition}

\begin{proof}
Let $\sigma\in\cO(Z)^R_{\theta^m}$ be a $\theta^m$-semi-invariant function, which we can assume without loss of generality is homogeneous for the grading induced by the $T$-action. Then $Z_\sigma:= \{ \sigma \neq 0\}$ is an affine open subscheme contained in $Z^{\theta-ss}$, so that $Z_\sigma$ admits a good quotient $W_\sigma = Z_\sigma \git R \hookrightarrow W$. We define the open affine subschemes $X_\sigma:=f^{-1}(Z_\sigma)$ and $Q_\sigma:=(\overline{f})^{-1}(W_\sigma) = X_\sigma \git G$, so that we have the following open immersions:
\[
\begin{tikzcd}[ampersand replacement=\&]
X_\sigma \ar[r,"p_\sigma"]\ar[d,"f_\sigma"] \& Q_\sigma \ar[d,"\overline{f}_\sigma"] \\
Z_\sigma \ar[r,"q_\sigma"] \& W_\sigma
\end{tikzcd}
\quad \stackrel{\mathrm{open}}{\hookrightarrow} \quad 
\begin{tikzcd}[ampersand replacement=\&]
X^{q-\sst} \ar[r,"p"]\ar[d,"f"] \& Q \ar[d,"\overline{f}"] \\
Z^{\theta-\sst} \ar[r,"q"] \& W.
\end{tikzcd}
\]
Note that $\overline{f}_\sigma$ corresponds to the algebra homomorphism $\mathcal{O}_W(W_\sigma)\rightarrow\mathcal{A}(Z_\sigma)^G$.

Since the $T$-action on $Z$ commutes with the $R$-action, the good $R$-quotient $W$ can be covered with finitely many $T$-invariant affine open subschemes $W_\sigma$ defined as above. Moreover, for any $N\geq1$ and for every $w\in W_\sigma^{T_N}$, we have:
\[
(q_*\cA^G)_w\otimes_{\cO_{W,w}}\widehat{\cO_{W,w}}
\simeq
\mathcal{A}(Z_\sigma)^G\otimes_{\cO_W(W_\sigma)}\widehat{\cO_{W,w}}.
\]
Therefore, it is enough to show that, for fixed $W_\sigma$ and for $N$ large enough, the algebra $\mathcal{A}(Z_\sigma)^G$ is generated over $\cO_W(W_\sigma)$ by functions of negative weight with respect to the grading induced by the $T_N$-action.

Let us pick generators of the following algebras of semi-invariants
\[
\cO(Z)^R[\sigma_1,\ldots,\sigma_r] \twoheadrightarrow \cO(Z)^{R,\theta},
\]
\[
\cA(Z)^R[\sigma_1,\ldots,\sigma_r,\tau_1,\ldots,\tau_s] \twoheadrightarrow \cA(Z)^{R,\theta},
\]
such that $\sigma_i$ for $1\leq i\leq r$ (resp.\ $\tau_j$ for $ 1\leq j\leq s$) are homogeneous with respect to the $T$-grading and semi-invariant with respect to $\theta^{m_i}$ for $ 1\leq i\leq r$ (resp.\ $\theta^{n_j}$ for $ 1\leq j\leq s$). We moreover assume that $\tau_j\notin\cO(Z)^{R,\theta}$ for all $1\leq j\leq s$, so that every $\tau_j$ has negative $T_g$-weight. 

Since $Z_\sigma =\{ \sigma \neq 0 \}$ and localisation commutes with taking invariants, we have:
\[
\mathcal{A}(Z_\sigma)^G\subset\mathcal{A}(Z_\sigma)^R
\twoheadleftarrow
\cA(Z)^R\left[\frac{\sigma_i^{m}}{\sigma^{m_i}},\frac{\tau_j^{m}}{\sigma^{n_j}}\right]_{\substack{1\leq i\leq r\\ 1\leq j\leq s}}.
\]
Note that all non-constant functions in $\cA(Z)^R$ have negative $T_c$-weight and non-positive $T_g$-weight; thus they have negative $T_N$-weight for $N \geq 1$. Furthermore,  ${\sigma_i^{m}}/{\sigma^{m_i}}\in\cO_W(W_\sigma)$. Thus is remains to check that the generators ${\tau_j^{m}}/{\sigma^{n_j}}$ have negative $T_N$ weight for $N$ sufficiently large. The $T$-weights of $\sigma$ and $\tau_j$ have the following form
\[ T\mathrm{-wt}(\sigma) = (0,w) \quad \quad T\mathrm{-wt}(\tau_i)= (v_i,w_i)\]
where $v_i,w_i$ and $w$ are all negative. Hence for $N$-sufficiently large, the functions ${\tau_j^{m}}/{\sigma^{n_j}}$ for $ 1\leq j\leq s$ have negative $T_N$-weight. Therefore, the graded algebra $\cA(Z_\sigma)^R$ is generated by functions of negative $T_N$-weight over $\cO_W(W_\sigma)$. The same then holds for the graded subalgebra $\cA(Z_\sigma)^G$.

The final claim then follows from \cref{corollary affine map semiproj}.
\end{proof}

Let us finally explain how these results can be used to deduce purity statements for relative affine non-reductive GIT quotients $Q = X\git_{\hspace{-2pt} q} G$ that are smooth, without needing the underlying reductive GIT quotient $W = Z\git_{\hspace{-2pt} \theta} R$ to be smooth. We can also extend these purity results to the case of relative twisted affine non-reductive GIT quotients, where a non-trivial character $\rho \colon G \rightarrow \GG_m$ is used to obtain a quotient $Q_\rho := X\git_{\hspace{-2pt}q,\rho} G$ which is projective over $Q = X \git_{\hspace{-2pt} q} G$.

\begin{theorem}\label{Thm/purityNRGITquotients}
In the above set-up, we have the following.
\begin{enumerate}[label=\roman*)]
    \item\label{item 1 purity nonred GIT} The $T_N$-action on $W$ is semi-projective (for any $N\geq1$).
    \item\label{item 2 purity nonred GIT} For $N$ sufficiently large, the $T_N$-action on $Q$ and $Q_\rho$ is semi-projective.
    \item\label{item 3 purity nonred GIT} In particular, if $Q$ (resp.\ $Q_\rho$) is smooth, then its cohomology is pure.
\end{enumerate}
\end{theorem}
\begin{proof}
For \ref{item 1 purity nonred GIT}, we first note that as $T_g$ acts trivially on $Z$ and $T_c$ acts conically, the induced $T_N$-action on $Z$ is conical (for any $N\geq1$). Hence the induced $T_N$-action on the affine reductive GIT quotient $Z\git R$ is conical, and thus also semi-projective. Since $W = Z\git_{\hspace{-2pt}\theta} R$ is a twisted affine GIT quotient, it is projective over the affine GIT quotient $Z\git R$ and so its induced $T_N$-action is also semi-projective by \cref{Lem/semiprojectivity and proper morphisms}.

For \ref{item 2 purity nonred GIT}, we apply \cref{Prop/weightAssumption} to show that for $N$ sufficiently large, the $T_N$-action on $Q$ is semi-projective. As above, since $Q_\rho \rightarrow Q$ is proper, the induced $T_N$-action on $Q_\rho$ is semi-projective. Then we conclude \ref{item 3 purity nonred GIT} by applying \cref{Prop purity} \ref{Prop purity item 1}.
\end{proof}

\subsection{Purity for smooth quiver moduli spaces with multiplicities}\label{sec/purity for moduli of quiver reps with multiplicities}

We will now apply \cref{Thm/purityNRGITquotients} to the moduli spaces of $(\theta,\rho)$-semistable representations of $(Q,\bfm)$ we have constructed.

\begin{proof}[Proof of \cref{main thm purity} \ref{main thm purity first part}]
Under these assumptions, the moduli space $M^{\theta,\rho-\sst}_{Q,\bfm;\bfr}$ is constructed as a relative NRGIT quotient for the truncation map $\tau \colon R(Q,\bfm;\bfr) \rightarrow R(Q,\bfr)$ as in \cref{first main thm}. As in $\S$\ref{Subsect/RelativePurityTech} above, we will denote the multiplicative group used for the external grading by $T_g$. We consider another multiplicative group denoted by $T_c$ which scales all arrows and so acts on both the domain and codomain of $\tau$ with weight $1$. This $T_c$-action is conical such that $\tau$ is $T_c$-equivariant and the $T_c$-action commutes with the actions of $\GL_{\bfm,\bfr} \rtimes T_g$ and $\GL_{\bfr}$, so we are in precisely the set-up of $\S$\ref{Subsect/RelativePurityTech}. Hence, by \cref{Thm/purityNRGITquotients}, for $N >\!> 0$, the action of the subtorus
\[ T_N=\{ (t^N,t) \in T_g \times T_c \} \cong \GG_m\]
on $M^{\theta,\rho-\sst}_{Q,\bfm;\bfr} = R(Q,\bfm;\bfr)\git_{\theta,\rho} \GL_{\bfm,\bfr}$ is semi-projective, and since the assumption that $(\theta,\rho)$ is generic with respect to $\bfr$ ensures that $M^{\theta,\rho-\sst}_{Q,\bfm;\bfr}$ is smooth, we can conclude it is cohomologically pure.
\end{proof}

\subsection{Purity for Nakajima quiver varieties with multiplicities}\label{sec/purity for quiver varieties with multiplicities}

In this section, we prove the second part of \cref{main thm purity}. This requires modifying the stability condition for Nakajima quiver varieties with multiplicities. We explain this below.

\begin{notation}
Given a representation $V$ of $\overline{Q}$, we call $\pi(V)$ its restriction as a representation of $Q$.
\end{notation}

\begin{definition}\label{Def/modifStabilityCond}
Let $M$ be a locally free representation of $(\overline{Q},\bfm)$. We say that $M$ is $(\theta,\rho,\pi)$-semistable (resp. stable) if the following conditions hold:
\begin{itemize}
\item for any (\emph{not necessarily locally free}) subrepresentation $N\subset \pi\circ\sigma(M)$ such that $0<\rk N<\rk M$, we have:
\[
\theta\cdot\rk N\geq0;
\]
\item for any \emph{locally free} subrepresentation $N\subset M$ such that $0<\rk N<\rk M$ and $\theta\cdot\rk N=0$, we have:
\[
\rho\cdot\rk N\geq0
\text{ (resp. }
\rho\cdot\rk N>0
\text{).}
\]
\end{itemize}
\end{definition}

It turns out that this stability condition matches the semistable locus of $R(\overline{Q},\bfm;\bfr)$, for a different choice of external grading. Let us fix $\bm{\alpha},\bm{\beta}\in\bbZ_{>0}^{Q_0}\times\bbZ^{\overline{Q}_1}$ as in \cref{Cor/revisedExtGrad}, so that the moment map $\mu_{(Q,\bfm),\bfr}$ is equivariant with respect to the $\GG_m$-actions $*_{\bm{\alpha},\bm{\beta}}$ and $*_{\bm{\alpha},\bm{\alpha}+C}$ (for a fixed value of $C>0$).

\begin{proposition}\label{prop/ext grading making moment map equiv}
The action $*_{\bm{\alpha},\bm{\beta}}$ induces an externally graded action:
\[
\GL_{\bfm,\bfr}\rtimes\GG_m\curvearrowright R(\overline{Q},\bfm;\bfr)
\]
such that $R(\overline{Q},\bfm;\bfr)^{\GG_m}=R(Q,\bfr)$ and the moment map $\mu_{(Q,\bfm),\bfr}$ is $\GL_{\bfm,\bfr}\rtimes\GG_m$-equivariant.

Moreover, the semistable (resp.\ stable) locus in $R(\overline{Q},\bfm;\bfr)$ for that action (with respect to $q_\theta$ and $\rho$) is the open subset whose geometric points correspond to $(\theta,\rho,\pi)$-semistable (resp.\ stable) representations as in \cref{Def/modifStabilityCond}.
\end{proposition}

\begin{proof}
The first claim follows directly from \cref{Lem/extGrad2} and \cref{Cor/revisedExtGrad}. The second claim is proved in the same way as \cref{Thm/HMcrit}. We briefly indicate which steps in the proof require modification.

Since we modified the externally grading $\GG_m$-action, the weights of the $\GG_m$-action on the arrows of the thickened quiver $Q_{\bfm}$ change as follows:
\[
\widetilde{\mathrm{wt}}(a,m,f_1,f_2):=\alpha_j(mf_{ij}+f_2)-\alpha_if_1+\beta_a\geq0
\]
for $a \colon i\rightarrow j$ in $\overline{Q}_1$ and $0\leq m\leq m_{ij}-1$, $0\leq f_1\leq f_{ji}-1$ and $0\leq f_2\leq f_{ij}-1$. The description of $\tau$ on $R(\overline{Q}_\bfm,\bfr)$ is changed accordingly: the same formula applies, but only for arrows $a\in Q_1$. The description of $\GL_\bfr$-semistable points in \cref{Lem/HM1} remains valid \emph{mutatis mutandis} (the constant $w$ is modified). \cref{Lem/HM2} and \cref{Lem/HM3} hold without modification. Finally, the proof of \cref{Thm/HMcrit} from these three lemmas is unchanged.
\end{proof}

We can now finish the proof of our second main result.

\begin{proof}[Proof of \cref{main thm purity}\ref{main thm purity second part}]
We will apply \cref{Prop purity}\ref{Prop purity item 2} to a well-chosen one-parameter family of modified quiver varieties $\widetilde{N}_{Q,\bfm;\bfr}^{\theta,\rho-\sst}(t\bullet\gamma)$ over $\ t\in\bbA_{\bbC}^1$.

For the equivariant action of $\GL_{\bfm,\bfr}\rightarrow\GL_\bfr$ on the affine map $\pi\circ\tau \colon R(\overline{Q},\bfm;\bfr)\rightarrow R(Q,\bfr)$ with the external grading of \cref{prop/ext grading making moment map equiv}, we build a NRGIT quotient $\widetilde{M}_{\overline{Q},\bfm;\bfr}^{\theta,\rho-\sst}$ (with respect to $q_\theta$ and $\rho$). By construction, for every $\gamma=(\gamma_i\cdot\Id)\in\gl_{\bfm,\bfr}$, the inclusion $\mu_{(Q,\bfm),\bfr}^{-1}(\gamma)\subset R(\overline{Q},\bfm;\bfr)$ induces a closed immersion $\widetilde{N}_{Q,\bfm;\bfr}^{\theta,\rho-\sst}(\gamma)\hookrightarrow\widetilde{M}_{\overline{Q},\bfm;\bfr}^{\theta,\rho-\sst}$.

Consider the scaling action of $T_c=\GG_m$ on $R(\overline{Q},\bfm;\bfr)$ with weight 1 and the external grading action of $T_g=\GG_m$. As in subsection \cref{Subsect/RelativePurityTech}, we define for every $N\geq1$ the one-dimensional subtorus:
\[
\begin{array}{rcl}
T_N & \hookrightarrow & T_c\times T_g \\
t & \mapsto & (t,t^N).
\end{array}
\]
Then by \cref{Thm/purityNRGITquotients}, for $N$ large enough, the induced action of $T_N$ on $\widetilde{M}_{\overline{Q},\bfm;\bfr}^{\theta,\rho-\sst}$ is semi-projective.

Now, by \cref{Cor/revisedExtGrad} and since the moment map is bilinear, the moment map $\mu_{(Q,\bfm),\bfr}$ is equivariant with respect to the action of $T_N$ on $R(\overline{Q},\bfm;\bfr)$ and a certain scaling action on $\gl_{\bfm,\bfr}$ with positive weights, which we call $\bullet_N$. Thus the second projection on
\[
\left\{
(x,y,t)\in R(\overline{Q},\bfm;\bfr)\times\bbA_\bbC^1\ \vert\ \mu_{(Q,\bfm),\bfr}(x,y)=t\bullet_N\gamma
\right\}
\]
is equivariant. By passing to the NRGIT quotient as above, we obtain a one-parameter family $\widetilde{\mathcal{N}}\rightarrow\bbA_\bbC^1$ whose fibre over $t\in\bbA^1$ is $\widetilde{N}_{Q,\bfm;\bfr}^{\theta,\rho-\sst}(t\bullet_N\gamma)$.

This family is endowed with an equivariant semi-projective action by the discussion above. We further argue that the family is smooth. Indeed, since $(\theta,\rho)$ is generic, we have that, at any point $(x,y)\in\mu_{(Q,\bfm),\bfr}^{-1}(\gamma)^{\theta,\rho,\pi-\sst}$, the differential
\[
\begin{array}{rcl}
R(\overline{Q},\bfm;\bfr)\times\bbA_{\bbC}^1 & \rightarrow & \gl_{\bfm,\bfr}^0 \\
(x,y,t) & \mapsto & d\mu(x,y)-t\gamma
\end{array}
\]
is surjective, by \cite[Lem.\ 2.1.5]{CBVB04}. We conclude by applying \cref{Prop/purityCriterion}.
\end{proof}

\begin{appendices}

\section{Unfolding quivers with multiplicities}\label{sec/unfolding} \label{appendix}

In this appendix, we explain how the space of representations $R(Q,\bfm;\bfr)$ can be embedded into the space of representations of a quiver with constant multiplicities $(Q',M\mathbf{1})$, obtained by unfolding $(Q,\bfm)$. We show that, under this embedding, the external grading $\GG_m$-action used in the proof of \cref{first main thm} is induced by the $\GG_m$-action introduced in \cite[\S 7]{HHJ24}, making it more natural.

\begin{definitionApp}
Let $(Q,\bfm)$ be a quiver with multiplicities. The unfolding of $(Q,\bfm)$ is the quiver $Q'$ with the following vertices and arrows:
\begin{itemize}
\item $Q'_0:=\bigsqcup_{i\in Q_0}\{i\}\times\bbZ/m_i\bbZ$,
\item $Q'_1:=\bigsqcup_{i,j\in Q_0}\left\{(a,n_i,n_j):(i,n_i)\rightarrow (j,n_j)\ \left\vert\ 
\begin{array}{l}
a\in Q_1 \\
n_i,n_j\in\bbZ/m_i\bbZ\times\bbZ/m_j\bbZ \\
n_i\equiv n_j\ \mathrm{mod}\ m_{ij}
\end{array}
\right.\right\}$.
\end{itemize}
\end{definitionApp}

Given $i,j\in Q_0$, there are $\mu_{ij}=\mathrm{lcm}(m_i,m_j)$ arrows of the form $(i,n_i)\rightarrow (j,n_j)$ for some $(n_i,n_j)\in\bbZ/m_i\bbZ\times\bbZ/m_j\bbZ$. Indeed, the couples $(n_i,n_j)$ which satisfy $n_i\equiv n_j\ \mathrm{mod}\ m_{ij}$ are precisely those of the form $n\cdot(1,1)$, by the Chinese remainder theorem. So these couples form the image of the first arrow in the following short exact sequence:
\[
0\rightarrow\bbZ/\mu_{ij}\bbZ\underset{(1,1)}{\longrightarrow}\bbZ/m_i\bbZ\times\bbZ/m_j\bbZ\underset{(n_i,n_j)\mapsto n_i-n_j}{\longrightarrow}\bbZ/m_{ij}\bbZ\rightarrow0.
\]

\begin{exampleApp}
Consider the quiver $Q=\bullet\rightarrow\bullet$ with multiplicities $\bfm=(2,1)$. The unfolded quiver is as follows:
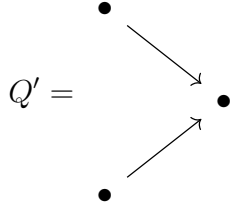
\[
Q'=
\begin{tikzcd}
\bullet \ar[dr] & \\
& \bullet \\
\bullet \ar[ur] & 
\end{tikzcd}
\]
This corresponds to folding the Dynkin diagram of $A_3$ to obtain the diagram of type $B_2$.
\end{exampleApp}

\begin{notationApp}
Given an integer $m\geq1$, define:
\begin{itemize}
\item $k_m:=k[\epsilon^{1/m}]/(\epsilon)$;
\item $R_m:=k[\![\epsilon^{1/m}]\!]$;
\item $K_m:=k(\!(\epsilon^{1/m})\!)$.
\end{itemize}
and write $\epsilon_m:=\epsilon^{1/m}$. When $m=1$, we will write $R$ (resp. $K,k$) for short. Given a primitive $m$th root of unity, called $\zeta\in k$, there is a natural (Galois) action of $\bbZ/m\bbZ$ on $R_m$ (resp. $K_m,k_m$), given by $[n]\cdot\lambda(\epsilon_m):=\lambda(\zeta^n\epsilon_m)$ for  $[n]\in\bbZ/m\bbZ$ and $\lambda(\epsilon_m)\in R_m$. We will write $\zeta^n*\lambda:=[n]\cdot\lambda(\epsilon_m)$ for short.
\end{notationApp}

Let $(Q,\bfm)$ be a quiver with multiplicities. For $i,j\in Q_0$, we have the following isomorphisms:
\[
\Hom_{k_{m_{ij}}}\left(k_{m_i}^{\oplus r_i},k_{m_j}^{\oplus r_j}\right)
\simeq
\Hom_{k_{m_j}}\left((k_{m_j}\otimes_{k_{m_{ij}}}k_{m_i})\otimes_{k_{m_i}}k_{m_i}^{\oplus r_i},k_{m_j}^{\oplus r_j}\right),
\]
\[
\Hom_{k_{m_{ij}}}\left(k_{m_i}^{\oplus r_i},k_{m_j}^{\oplus r_j}\right)
\simeq
\Hom_{k_{m_i}}\left(k_{m_i}^{\oplus r_i},(k_{m_i}\otimes_{k_{m_{ij}}}k_{m_j})\otimes_{k_{m_j}}k_{m_j}^{\oplus r_j}\right).
\]
Note that $k_{m_i}\otimes_{k_{m_{ij}}}k_{m_j}=\bigoplus_{f=0}^{f_{ji}-1}k_{m_j}\cdot (\epsilon^{f/m_i}\otimes1)$. Moreover, there is an isomorphism of bimodules:
\[
\begin{array}{rcl}
k_{m_i}\otimes_{k_{m_{ij}}}k_{m_j} & \simeq & \Hom_{k_{m_j}}(k_{m_j}\otimes_{k_{m_{ij}}}k_{m_i},k_{m_j}) \\
\epsilon^{f/m_i}\otimes 1 & \mapsto & \left(\epsilon^{(f_{ji}-1-f)/m_i}\right)^*.
\end{array}
\]
So decomposing $\varphi\in\Hom_{k_{m_{ij}}}\left(k_{m_i}^{\oplus r_i},k_{m_j}^{\oplus r_j}\right)$ along:
\begin{equation}\label{Eqn/matrixDecomposition}
\Hom_{k_{m_i}}\left(k_{m_i}^{\oplus r_i},(k_{m_i}\otimes_{k_{m_{ij}}}k_{m_j})\otimes_{k_{m_j}}k_{m_j}^{\oplus r_j}\right)
\simeq
\bigoplus_{f=0}^{f_{ji}-1}\Hom_{k_{m_j}}\left(k_{m_j}^{\oplus r_i},k_{m_j}^{\oplus r_j}\right)
\end{equation}
amounts to giving the restrictions of $\varphi$ to $\epsilon^{(f_{ji}-1-f)/m_i}\cdot k_{m_{ij}}^{\oplus r_i}\subset k_{m_i}^{\oplus r_i}$, for $0\leq f\leq f_{ji}-1$.

To compare with our previous notation, let us work with entries of the corresponding $r_j\times r_i$ matrix, i.e. assume that $r_i=r_j=1$. Then for $\lambda\in k_{m_j}$, consider $(0,\ldots,\lambda,\ldots,0)$ on the right-hand side of (\ref{Eqn/matrixDecomposition}) (where $\lambda$ is placed in $f$th position). In the bases:
\[
k_{m_i}=(k\cdot\epsilon_{m_i}^{m_i-1}\oplus\ldots\oplus k\cdot\epsilon_{m_i}^{f_{ji}(m_{ij}-1)})\oplus\ldots\oplus(k\cdot\epsilon_{m_i}^{f_{ji}-1} \oplus\ldots\oplus k\cdot 1),
\]
\[
k_{m_j}=(k\cdot\epsilon_{m_j}^{m_j-1}\oplus\ldots\oplus k\cdot\epsilon_{m_j}^{f_{ij}(m_{ij}-1)})\oplus\ldots\oplus(k\cdot\epsilon_{m_j}^{f_{ij}-1} \oplus\ldots\oplus k\cdot 1),
\]
the corresponding homomorphism $\varphi\in\Hom_{k_{m_{ij}}}\left(k_{m_i},k_{m_j}\right)$ is given by:
\[
\Lambda
=
\left(
\begin{array}{cccc}
\Lambda_0 & \Lambda_1 & \ldots & \Lambda_{m_{ij}-1} \\
0 & \Lambda_0 & \ddots & \vdots \\
\vdots & \ddots & \ddots & \Lambda_1 \\
0 & \ldots & 0 & \Lambda_0
\end{array}
\right)
,
\]
where columns of $\Lambda_m$ ($0\leq m\leq m_{ij}-1$) are all zero, except for the $(f+1)$th column, which is equal to:
\[
\left(
\begin{array}{c}
\lambda_{(m+1)f_{ij}-1} \\
\vdots \\
\lambda_{mf_{ij}}
\end{array}
\right)
\]
and $\lambda=\lambda_0+\ldots+\lambda_{m_j-1}\epsilon_{m_j}^{m_j-1}$.

\begin{propositionApp}
Let $(Q,\bfm)$ be a quiver with multiplicities. Let $M:=\mathrm{lcm}(m_i,\ i\in Q_0)$ and $\zeta$ be a primitive $M$th root of unity. Let $Q'$ be the unfolding of $(Q,\bfm)$ and $\bfr'\in\bbZ_{\geq0}^{Q'_0}$ the rank vector given by $r'_{i,n_i}:=r_i$.

Then there exists a $k$-linear embedding $\iota \colon R(Q,\bfm;\bfr)\hookrightarrow R(Q',M\mathbf{1};\bfr')$ which is equivariant with respect to the embedding of algebraic groups:
\[
\begin{array}{rcl}
\GL_{\bfm,\bfr} & \hookrightarrow & \GL_{M\mathbf{1},\bfr'} \\
(g_i)_{i\in Q_0} & \mapsto & \left(\zeta^{n_iM/m_i}*g_i\right)_{(i,n_i)\in Q'_0}.
\end{array}
\]
The embedding $\iota$ is characterised as follows: consider the representation $x$ defined by:
\[
x_a:=
\left\{
\begin{array}{ll}
\Lambda\otimes E_{p,q} & a=a_0, \\
0 & a\ne a_0,
\end{array}
\right.
\]
where $a_0 \colon i\rightarrow j$ is a fixed arrow, the matrix $\Lambda$ is as described above and $E_{p,q}$ is the coordinate $r_j\times r_i$ matrix for $1\leq p\leq r_j$, $1\leq q\leq r_i$. Then, for all $n\in\bbZ/\mu_{ij}\bbZ$, we have:
\[
\iota(x)_{(a,n)}:=
\left\{
\begin{array}{ll}
\zeta^{nM/\mu_{ij}}*(\lambda\epsilon_{m_i}^f)\otimes E_{p,q} & a=a_0, \\
0 & a\ne a_0.
\end{array}
\right.
\]
\end{propositionApp}

\begin{proof}
We informally consider $x\in R(Q,\bfm;\bfr)$ as a collection of matrices with coefficients in Laurent series, rather than truncated power series. This relies on the following diagram of $k$-vector spaces:
\[
\begin{tikzcd}
k_m\ar[r,dashed,bend left] & R_m\ar[l,"\mathrm{mod}\ \epsilon"]\ar[r,"\subset"] & K_m.
\end{tikzcd}
\]
Similarly, there is a diagram of \emph{bimodules}:
\[
\begin{tikzcd}
k_{m_i}\otimes_{k_{m_{ij}}} k_{m_j}\ar[r,dashed,bend left] & R_{m_i}\otimes_{R_{m_{ij}}} R_{m_j}\ar[l,"\mathrm{mod}\ \epsilon"]\ar[r,"\subset"]\ar[d,"\sim"] & K_{m_i}\otimes_{K_{m_{ij}}} K_{m_j}\ar[d,"\sim"] \\
 & k[\![\epsilon^{1/m_i},\epsilon^{1/m_j}]\!]\ar[r,"\subset"] & K_{\mu_{ij}}
\end{tikzcd}
\]
The analogue of the decomposition (\ref{Eqn/matrixDecomposition}) over fields of Laurent series is then:
\[
\Hom_{K_{m_i}}\left(K_{m_i}^{\oplus r_i},K_{\mu_{ij}}\otimes_{K_{m_j}}K_{m_j}^{\oplus r_j}\right)
\simeq
\bigoplus_{f=0}^{f_{ji}-1}\Hom_{K_{m_j}}\left(K_{m_j}^{\oplus r_i},K_{m_j}^{\oplus r_j}\right).
\]
Observe that, for all $a\in Q_1$:
\[
\begin{split}
& K_M\otimes_{K}\Hom_{K_{m_i}}\left(K_{m_i}^{\oplus r_i},K_{\mu_{ij}}\otimes_{K_{m_j}}K_{m_j}^{\oplus r_j}\right) \\
& \simeq
\Hom_{K_M\otimes_K K_{m_i}}\left(K_M\otimes_K K_{m_i}^{\oplus r_i},(K_M\otimes_K K_{\mu_{ij}})\otimes_{K_M\otimes_K K_{m_j}}K_M\otimes_K K_{m_j}^{\oplus r_j}\right) \\
& \simeq \Hom_{(K_M)^{m_i}}\left((K_M^{m_i})^{\oplus r_i},K_M^{\mu_{ij}}\otimes_{K_M^{m_j}}(K_M^{m_j})^{\oplus r_j}\right) \\
& \simeq
\bigoplus_{\substack{n_i,n_j\ \mathrm{s.t.} \\ (a,n_i,n_j)\in Q'_1}}\Hom_{K_M}(K_M^{\oplus r_i},K_M^{\oplus r_j}).
\end{split}
\]
In other words, we may see $(1\otimes x_a)_{a\in Q_1}$ as a representation of $Q'$ with coefficients in $K_M$. Under the isomorphisms above, the representation $(1\otimes x_a)_{a\in Q_1}$ corresponds to the following representation of $(Q',M\mathbf{1})$: for $Q'_1\ni a' \colon i'=(i,n_i)\rightarrow j'=(j,n_j)$ corresponding to $n\in\bbZ/\mu_{ij}\bbZ$, we associate the matrix $\zeta^{n\frac{M}{\mu_{ij}}}*(\lambda\epsilon_{m_i}^f)\otimes E_{p,q}$.
\end{proof}

Using the grading of $R(Q',M\mathbf{1};\bfr)$ introduced in \cite[\S 7]{HHJ24}, we obtain that for $0\leq g\leq m_j-1$, the coefficient $\lambda_g$ has weight $\frac{M}{m_j}g-\frac{M}{m_i}f+\frac{M}{m_i}(f_{ji}-1)$. This grading is compatible with the external grading of $\GL_{\bfm,\bfr}$ given by the substitution $\epsilon_M\mapsto t\epsilon_M$. Moreover, we recover this way the external grading action $*_{\bm{\alpha},\bm{\beta}}$ on $R(Q,\bfm;\bfr)$, where $\alpha_i=\frac{M}{m_i}$ for all $i\in Q_0$ and $\beta_a=\alpha_i(f_{ji}-1)$ for all arrows $a \colon i\rightarrow j$.

\end{appendices}

\printbibliography

\end{document}